\documentclass[11pt]{article}
\usepackage[a4paper,margin=2.5cm]{geometry}
\usepackage[T1]{fontenc}
\usepackage{lmodern,microtype}
\usepackage{amsmath,amssymb,mathtools,amsthm}
\usepackage{graphicx,pgfplotstable}
\usepackage{tabularx,booktabs,multirow}
\usepackage{enumitem}
\usepackage[colorlinks=true,allcolors=blue]{hyperref}
\usepackage[numbers,square,sort&compress]{natbib}

\setlist{itemsep=0.5ex,topsep=1ex,parsep=0pt,partopsep=0pt,label=(\alph*),leftmargin=*}
\pgfplotsset{compat=1.18}
\newtheorem{theorem}{Theorem}
\newtheorem{lemma}{Lemma}
\newtheorem{proposition}{Proposition}
\theoremstyle{definition}
\newtheorem{definition}{Definition}
\theoremstyle{remark}
\newtheorem{remark}{Remark}
\title{Generalized Quadratic Matrix Manifolds: \\
Riemannian Geometry and Computations}
\author{Zhen Peng\thanks{State Key Laboratory of Mathematical Sciences, Academy of Mathematics and Systems Science, Chinese Academy of Sciences, China (\href{mailto:pengzhen@lsec.cc.ac.cn}{pengzhen@lsec.cc.ac.cn}, \href{mailto:gaobin@lsec.cc.ac.cn}{gaobin@lsec.cc.ac.cn}).}
\and Bin Gao\footnotemark[1]}
\date{}

\begin{document}
\maketitle

\begin{abstract}
    We study the set $\mathrm{Gq}(P,Q)$ of real $n\times p$ matrices $X$ satisfying $X^\top Q X = P$, where the nonsingular matrices $P$ and $Q$ are both symmetric or both skew-symmetric and the set is nonempty.
    Special choices of $P$ and $Q$ yield the Stiefel, generalized Stiefel, indefinite Stiefel, and symplectic Stiefel manifolds, as well as the quadratic matrix Lie groups.
    We prove that $\mathrm{Gq}(P,Q)$ is a closed embedded submanifold of $\mathbb{R}^{n \times p}$ and that it is compact exactly when $P$ and $Q$ are both symmetric and definite of the same sign.
    We derive closed-form expressions for its tangent space and, under a canonical-like Riemannian metric, for the normal space, the orthogonal projection, and the Riemannian gradient and Hessian.
    We extend the Cayley, quasi-geodesic, and polar retractions to $\mathrm{Gq}(P,Q)$, with explicit domains, vector transports, and closed-form second-order corrections.
    These objects are the ingredients of Riemannian first- and second-order methods.
    In trace minimization tests with symmetric positive-definite, skew-symmetric, and symmetric indefinite pairs $(P,Q)$, the Riemannian conjugate gradient and trust-region methods reach the same objective value with different retractions and keep the iterates on $\mathrm{Gq}(P,Q)$ up to rounding error.
\end{abstract}



\section{Introduction}
\label{sec:introduction}
    Let $Q\in\mathbb{R}^{n\times n}$ and $P\in\mathbb{R}^{p\times p}$, with $p\le n$, be nonsingular matrices that are both symmetric or both skew-symmetric.
    This paper studies the set
    \begin{equation}\label{eq:GQ-constraint}
        \mathrm{Gq}(P,Q) := \{X \in \mathbb{R}^{n\times p} : X^\top Q X = P\}.
    \end{equation}
    The matrix $Q$ defines the bilinear form $\langle x,y\rangle_Q := x^\top Q y$ on $\mathbb{R}^n$, and a matrix $X$ with columns $x_1,\dots,x_p$ lies in $\mathrm{Gq}(P,Q)$ when $\langle x_i,x_j\rangle_Q=P_{ij}$ for all $i$ and $j$.
    In other words, $\mathrm{Gq}(P,Q)$ collects the $p$-tuples of vectors in $\mathbb{R}^n$ whose Gram matrix with respect to $\langle\cdot,\cdot\rangle_Q$ equals $P$.
    Particular choices of $P$ and $Q$ give well-studied matrix manifolds.
    Identity matrices give the Stiefel manifold of matrices with orthonormal columns, which arises in deep learning~\cite{huang2018orthogonal}, statistics~\cite{chakraborty2019statistics}, scientific computing~\cite{wen2013feasible}, and computer vision~\cite{lui2012advances}.
    Symmetric positive-definite (SPD) weights give the generalized Stiefel manifolds, and a symmetric indefinite $Q$ gives the indefinite Stiefel manifold.
    The standard symplectic matrices give the symplectic Stiefel manifold, whose elements serve as reduced bases in the symplectic model reduction of Hamiltonian systems~\cite{peng2016symplectic}.
    For $p=n$ and $P=Q$, the set $\mathrm{Gq}(Q,Q)$ is the group of matrices that preserve $\langle\cdot,\cdot\rangle_Q$.
    These quadratic matrix Lie groups include the orthogonal, symplectic, and indefinite orthogonal groups and appear in optimal control~\cite{bloch2008optimal}, the control of quantum gates~\cite{wu2008optimal}, and Hamiltonian mechanics~\cite{arnol2001dynamical}.
    The existence of $\mathrm{Gq}(P,Q)$ depends only on congruence invariants of $P$ and $Q$~\cite{borobia2021consistency, borobia2023consistency}.
    In the symmetric case, Sylvester's law of inertia~\cite[Thm.~8.1.17]{golub2013matrix} requires the numbers of positive and negative eigenvalues of $P$ not to exceed those of $Q$.
    In the skew-symmetric case, nonsingularity requires $p$ and $n$ to be even, and all nonsingular skew-symmetric matrices of a fixed even size are congruent to the standard symplectic matrix~\cite{golub2013matrix, youla1961normal}.
    We call such pairs $(P,Q)$ \emph{admissible} and consider only admissible pairs.

    Many of these applications lead to the minimization problem
    \begin{equation}\label{eq:GQ-problem}
        \min_{X\in\mathrm{Gq}(P, Q)} f(X)
    \end{equation}
    with a smooth objective $f:\mathbb{R}^{n\times p}\to\mathbb{R}$.
    Riemannian optimization~\cite{absil2008optimization, boumal2023introduction, hu2020brief} solves~\eqref{eq:GQ-problem} by iterating on the manifold $\mathrm{Gq}(P,Q)$ itself.
    Each iteration computes a step in the tangent space from the Riemannian gradient, and in second-order methods also from the Riemannian Hessian, both taken with respect to a chosen Riemannian metric.
    A retraction maps the step back to $\mathrm{Gq}(P,Q)$, ensuring that every iterate satisfies $X^\top QX=P$.
    Conjugate gradient and quasi-Newton methods also use a vector transport, which moves tangent vectors between tangent spaces.
    For each of the special cases above, these objects have been derived separately, with a metric, retractions, and vector transports specific to that case.
    Table~\ref{tab:summary-manifolds} lists representative works.

    The Riemannian geometry of $\mathrm{Gq}(P,Q)$ for general $P$ and $Q$ has not been studied systematically, and it does not follow from the special cases by substituting $P$ and $Q$.
    Congruence reduces every admissible pair to a canonical one, and the induced linear change of variables maps $\mathrm{Gq}(P,Q)$ onto a Stiefel, an indefinite Stiefel, or a symplectic Stiefel manifold.
    This change of variables settles the smooth structure of $\mathrm{Gq}(P,Q)$.
    Every Riemannian metric on the canonical model pulls back to one on $\mathrm{Gq}(P,Q)$.
    Computing on that model requires expressing the objective and the chosen metric in the transformed coordinates.
    We derive explicit formulas for the selected metric, gradient, and Hessian in the original variables.

    Three features of the general pair have no counterpart on the Stiefel manifold.
    First, $\langle\cdot,\cdot\rangle_Q$ is an inner product only when $Q$ is SPD.
    Even when $Q$ is indefinite or skew-symmetric, the nonsingularity of $X^\top QX=P$ makes its restriction to $\mathrm{range}(X)$ nondegenerate.
    This gives a $Q$-orthogonal splitting of $\mathbb{R}^n$ into the column space of $X$ and its complement.
    The associated projectors are generally oblique in the Euclidean inner product, requiring the Stiefel formulas to be adapted to this splitting.
    In the same way, the $(P,Q)$-adjoint $A\mapsto P^{-1}A^\top Q$ replaces the transpose, since $P^{-1}X^\top QX=I_p$ takes the place of $X^\top X=I_p$.
    The Sylvester equations, principal square roots, and polar decompositions behind the polar retraction, its inverse, and its vector transport become their structured analogues for the forms of $P$ and $Q$~\cite{mackey2005structured, higham2010canonical, benner2022stable}.

    Second, $\mathrm{Gq}(P,Q)$ is unbounded unless $P$ and $Q$ are definite.
    The Cayley transform and the polar decomposition give retractions on the whole tangent bundle of the Stiefel manifold, but on $\mathrm{Gq}(P,Q)$ they fail for some tangent vectors in the indefinite and skew-symmetric cases.
    The matrix inverted by the Cayley transform may be singular, and the matrix whose square root is extracted for the polar decomposition can have negative eigenvalues.
    Retractions built from these maps therefore need explicit domains, and their closed-form inverses hold only under spectral conditions.

    Third, the choice of metric decides whether the orthogonal projection onto the tangent space has a closed form.
    For the symplectic and indefinite Stiefel manifolds, this projection can require the solution of a Lyapunov equation under tractable metrics~\cite{gao2024symplectic, tiep2026generalized}. A metric that weights the two parts of the $Q$-orthogonal splitting separately gives a closed-form, Lyapunov-free projection.
    The varying weight contributes a correction term to the Levi-Civita connection used in the Riemannian Hessian.

    In this paper, we develop the geometry of $\mathrm{Gq}(P,Q)$ for all admissible pairs at once and builds from it the objects that Riemannian methods use.
    Our contributions are as follows.
    \begin{enumerate}
        \item We express the geometry of $\mathrm{Gq}(P,Q)$ through the $(P,Q)$-adjoint $A\mapsto P^{-1}A^\top Q$, which reduces to the transpose on the Stiefel manifold and to the symplectic adjoint on the symplectic Stiefel manifold.
        We prove that $\mathrm{Gq}(P,Q)$ is a closed embedded submanifold of dimension $np-p(p+\varrho)/2$, where $\varrho=+1$ in the symmetric and $\varrho=-1$ in the skew-symmetric case.
        It is compact when $P$ and $Q$ are both symmetric and definite of the same sign.
        Its tangent space has three equivalent descriptions, as the kernel of the differential of $X\mapsto X^\top QX$, as a coordinate form adapted to a $Q$-orthogonal splitting, and as the image of the Lie algebra of $Q$-skew-symmetric matrices.
        \item We define a canonical-like metric that weights the two coordinate blocks separately.
        Under this metric we derive closed forms for the normal space, the orthogonal projection, the Riemannian gradient, the Riemannian Hessian, and the Levi-Civita connection.
        It is worth noting that the projection is Lyapunov-free.
        The metric reduces to the metrics of Gao et al.~\cite{gao2021riemannian} on the symplectic Stiefel manifold and of Tiep et al.~\cite{tiep2026generalized} on the indefinite Stiefel manifold.
        On the Stiefel manifold, it gives a one-parameter family that contains the Euclidean and canonical metrics.
        \item We construct Cayley, quasi-geodesic, and polar retractions on $\mathrm{Gq}(P,Q)$ with explicit domains, closed-form inverses of the Cayley and polar retractions, and vector transports.
        A closed-form correction makes each of them second-order.
        On the Stiefel manifold with $2\le p<n$, the correction vanishes identically only for the quasi-geodesic and polar retractions under the Euclidean metric and for the Cayley retraction under the canonical metric.
    \end{enumerate}

    The remainder of the paper is organized as follows.
    Section~\ref{sec:preliminaries} fixes the notation and recalls the objects of Riemannian optimization on embedded submanifolds.
    Section~\ref{sec:structure} introduces the $(P,Q)$-adjoint and establishes the manifold structure of $\mathrm{Gq}(P, Q)$, namely the submanifold property, the reduction to canonical pairs by congruence, adapted coordinates, and the tangent space.
    Section~\ref{sec:manifold} introduces the canonical-like metric and derives the normal space, the orthogonal projection, the Riemannian gradient and Hessian, and the Levi-Civita connection.
    Section~\ref{sec:retraction} constructs the three retractions with their vector transports and second-order corrections and compares their costs.
    Section~\ref{sec:numerical} checks the gradient and Hessian formulas by Taylor tests and runs Riemannian solvers on a trace minimization problem, and Section~\ref{sec:conclusion} concludes with open directions.

\subsection{Existing Stiefel-Type Manifolds}
\label{sec:Stiefel-type-manifold}

    The Stiefel manifold~\cite{stiefel1935richtungsfelder} consists of the matrices with orthonormal columns,
    \begin{equation*}
        \mathrm{St}(p,n)=\{X\in\mathbb{R}^{n\times p}:X^\top X=I_p\}
        =\mathrm{Gq}(I_p,I_n),
    \end{equation*}
    where $I_k$ denotes the $k$-by-$k$ identity matrix.
    When $p = n$, $\mathrm{St}(p,n)$ becomes the orthogonal group~\cite{hall2013lie} $\mathrm{O}(n)=\{X\in\mathbb{R}^{n\times n}:X^\top X=I_n\}$.
    Let $B$ be an SPD matrix of compatible size.
    The two generalized Stiefel manifolds~\cite{gao2022optimization, shustin2023riemannian, goyens2025geometric} measure orthonormality in the inner product defined by $B$ or prescribe $B$ as the Gram matrix of the columns,
    \begin{equation*}
        \begin{aligned}
            \mathrm{St}_{\mathsf{I}}^B(p,n)&=\{X\in\mathbb{R}^{n\times p}:X^\top BX=I_p\}=\mathrm{Gq}(I_p, B),&\quad B\in\mathbb{R}^{n\times n},\\
            \mathrm{St}_{\mathsf{II}}^B(p,n)&=\{X\in\mathbb{R}^{n\times p}:X^\top X=B\}=\mathrm{Gq}(B, I_n),&\quad B\in\mathbb{R}^{p\times p}.
        \end{aligned}
    \end{equation*}
    The indefinite Stiefel manifold~\cite{tiep2026riemannian}
    \begin{equation*}
        \mathrm{iSt}(J, A)=\{X\in\mathbb{R}^{n\times p}:X^\top AX=J\}=\mathrm{Gq}(J, A)
    \end{equation*}
    allows a symmetric nonsingular, possibly indefinite, matrix $A\in\mathbb{R}^{n\times n}$ and a symmetric $J\in\mathbb{R}^{p\times p}$ with $J^2=I_p$.

    Symplectic matrices are central to Hamiltonian mechanics and to its structure-preserving numerical treatment~\cite{hairer2006geometric, dopico2009parametrization}. The symplectic Stiefel manifold~\cite{sigrist1973cross} consists of the rectangular symplectic matrices,
    \begin{equation*}
        \mathrm{Sp}(p,n)=\{X\in\mathbb{R}^{n\times p}:X^\top J_{n}X=J_{p}\}=\mathrm{Gq}(J_p,J_n),
    \end{equation*}
    where $J_m := \left[\begin{smallmatrix} 0 & I_{m/2} \\ -I_{m/2} & 0 \end{smallmatrix}\right]$ is the standard $m$-by-$m$ symplectic (Poisson) matrix with even $m$.
    When $p = n$, $\mathrm{Sp}(p,n)$ reduces to the symplectic group $\mathrm{Sp}(n)$~\cite{hall2013lie}.

    The quadratic matrix Lie group
    \begin{equation*}
        \mathrm{QLie}(Q)=\{X\in\mathbb{R}^{n\times n}:X^\top Q X=Q\}=\mathrm{Gq}(Q,Q)
    \end{equation*}
    consists of the matrices that preserve the bilinear form $\langle\cdot,\cdot\rangle_Q$.
    It was studied for orthogonal $Q$ by Machado and Leite~\cite{machado2002optimization} and for $Q$ with $Q^2=\pm I_n$ and $Q^\top=\pm Q$ by Bloch et al.~\cite{bloch2008optimal} and Zhang and Noakes~\cite{zhang2020riemannian}. It is the orthogonal group for $Q=I_n$ and the symplectic group for $Q=J_{n}$.
    For $Q=\left[\begin{smallmatrix} I_{k} & 0 \\ 0 & -I_{n-k} \end{smallmatrix}\right]$, it is the indefinite orthogonal group, whose elements are unitary with respect to an indefinite inner product~\cite{mehrmann1999structured}.

    Table~\ref{tab:summary-manifolds} summarizes existing Riemannian optimization methods for Stiefel-type matrix manifolds.
    In the ``Method'' column, ``1st-order'' means that only the Riemannian gradient is considered, whereas ``1st/2nd-order'' means that both the Riemannian gradient and Hessian are derived or used.
    The tractable metric is a weighted Riemannian metric with an SPD weight matrix varying smoothly with the iteration point~\cite[Sec.~2.1]{gao2024symplectic}. The canonical-like metric, a generalization of the Stiefel canonical metric, is a special case of the tractable metric~\cite[Sec.~2.2]{gao2021riemannian}.

    \begin{table*}[htb]
        \centering
        \small
        \caption{Selected Riemannian optimization methods on Stiefel-type matrix manifolds.}
        \label{tab:summary-manifolds}
        \renewcommand{\arraystretch}{1.15}
        \setlength{\tabcolsep}{6pt}
        \begin{tabular}{@{}lllll@{}}
            \toprule
            Refs & Manifold & Method & Metric & Retraction \\
            \midrule
            \multicolumn{5}{l}{\textbf{Stiefel Manifold}} \\
            \cite{edelman1998geometry} & $\mathrm{O}(n)$, $\mathrm{St}(p,n)$ & 1st/2nd-order & Euclidean, canonical & Geodesic \\
            \cite{absil2008optimization,boumal2023introduction} & $\mathrm{St}(p,n)$ & 1st/2nd-order & Euclidean & Polar, QR \\
            \cite{wen2013feasible,zhu2017riemannian} & $\mathrm{St}(p,n)$ & 1st-order & Euclidean & Cayley \\
            \cite{jiang2015framework} & $\mathrm{St}(p,n)$ & 1st-order & Euclidean & Feasible curve \\
            \cite{bendokat2021efficient} & $\mathrm{St}(p,n)$ & 1st-order & Euclidean & Quasi-geodesic \\
            \cite{zimmermann2022computing,mataigne2025efficient} & $\mathrm{St}(p,n)$ & 1st-order & Euclidean, canonical & Geodesic \\
            \midrule
            \multicolumn{5}{l}{\textbf{Generalized Manifold}} \\
            \cite{shustin2023riemannian} & $\mathrm{St}_{\mathsf{I}}^B(p,n)$ & 1st/2nd-order & tractable & Polar, QR, Cayley \\
            \cite{gao2022optimization} & $\mathrm{St}_{\mathsf{II}}^B(p,n)$ & 1st-order & canonical-like & --- \\
            \cite{tiep2026riemannian} & $\mathrm{iSt}(J, A)$ & 1st-order & tractable & Cayley \\
            \cite{tiep2026generalized,sato2026second} & $\mathrm{iSt}(J, A)$ & 1st/2nd-order & canonical-like & Quasi-geodesic \\
            \midrule
            \multicolumn{5}{l}{\textbf{Symplectic Manifold}} \\
            \cite{birtea2020optimization} & $\mathrm{Sp}(n)$ & 1st/2nd-order & Euclidean, left-invariant & Cayley \\
            \cite{gao2021geometry} & $\mathrm{Sp}(p,n)$ & 1st-order & Euclidean & Cayley \\
            \cite{gao2021riemannian} & $\mathrm{Sp}(p,n)$ & 1st-order & canonical-like & Cayley, Quasi-geodesic \\
            \cite{yamada2023conjugate,oviedo2023collection} & $\mathrm{Sp}(p,n)$ & 1st-order & canonical-like & Cayley, Feasible curve \\
            \cite{gao2024optimization} & $\mathrm{Sp}(p,n)$ & 1st-order & Euclidean, canonical-like & SR \\
            \cite{jensen2024riemannian} & $\mathrm{Sp}(p,n)$ & 2nd-order & right-invariant & Cayley \\
            \cite{gao2024symplectic} & $\mathrm{Sp}(p,n)$ & 2nd-order & tractable & Cayley, SR \\
            \midrule
            \multicolumn{5}{l}{\textbf{Quadratic Lie Group Manifold}} \\
            \cite{zhang2020riemannian} & $\mathrm{QLie}(Q)$ & 2nd-order & right-invariant & Cayley \\
            \midrule
            \multicolumn{5}{l}{\textbf{Generalized Quadratic Matrix Manifolds}} \\
            this work & $\mathrm{Gq}(P, Q)$ & 1st/2nd-order & canonical-like & Cayley, Quasi-geodesic, Polar \\
            \bottomrule
        \end{tabular}
    \end{table*}

\section{Preliminaries and Notation}
\label{sec:preliminaries}

\subsection{Notation}
\label{sec:notation}

    Let $\mathbb{R}^{n\times p}$ denote the vector space of real $n\times p$ matrices.
    We denote the Euclidean inner product by $\langle X,Y\rangle = \mathrm{tr}(X^\top Y)$ and the Frobenius norm by $\|X\|_{\mathrm{F}} = \langle X,X\rangle^{1/2}$.
    The symbols $\oplus$ and $\otimes$ denote the direct sum of subspaces and the Kronecker product, and $\mathrm{GL}(k)$ is the group of nonsingular $k\times k$ matrices.
    The operators $\sigma(\cdot)$, $\mathrm{sign}(\cdot)$, $\exp(\cdot)$, $\mathrm{vec}(\cdot)$, $\mathrm{diag}(\cdot)$, and $\mathrm{blkdiag}(\cdot)$ denote, respectively, the spectrum of a matrix, the matrix sign function, the matrix exponential, the column-major vectorization operator, the diagonal operator, and the block-diagonal operator.
    We use $\|\cdot\|_2$ for the spectral norm and $\|\cdot\|_X$ for the Riemannian norm at $X$.
    We write $0_{m\times k}$ (or $0_m$ when $m=k$) for the $m\times k$ zero matrix.
    For a matrix $A$ that is symmetric or skew-symmetric, we define the symmetry indicator
    \begin{equation}\label{eq:rho}
        \rho(A) = \begin{cases}
            +1,&\text{if }A=+A^\top,\\
            -1,&\text{if } A=-A^\top.
        \end{cases}
    \end{equation}
    The product $X^\top QX$ is symmetric when $Q$ is symmetric and skew-symmetric when $Q$ is skew-symmetric.
    We therefore assume throughout the paper that $P\in\mathrm{GL}(p)$ and $Q\in\mathrm{GL}(n)$ satisfy $\rho(P)=\rho(Q)\in\{+1,-1\}$, and we write
    \begin{equation}\label{eq:varrho}
        \varrho:=\rho(P) = \rho(Q)
    \end{equation}
    for the common symmetry indicator of $P$ and $Q$.
    The pair $(P,Q)$ is \emph{admissible} if
    \begin{equation}\label{eq:admissibility}
        \begin{cases}
        \nu_+(P)\le \nu_+(Q),\quad \nu_-(P)\le \nu_-(Q), &\text{if } \varrho=+1,\\
        p \text{ and } n \text{ are even}, &\text{if } \varrho=-1,
        \end{cases}
    \end{equation}
    where $\nu_+(A)$ and $\nu_-(A)$ denote the numbers of positive and negative eigenvalues of a real symmetric matrix $A$, respectively.
    For $p\le n$, admissibility is necessary and sufficient for $\mathrm{Gq}(P,Q)\neq\emptyset$.
    We assume throughout that $(P,Q)$ is admissible.

    We write $(A)^{1/2}$ for the \textit{principal square root} of a square matrix $A$, the square root whose eigenvalues lie in the open right half-plane.
    The following lemma collects its existence, uniqueness, and smoothness from Higham~\cite[Thm.~1.29, Thm.~3.5]{higham2008functions}.
    \begin{lemma}\label{lem:principal}
        Let $\mathcal{D} := \{A \in \mathbb{R}^{n\times n} : \sigma(A) \cap \mathbb{R}_0^- = \emptyset\}$, where $\mathbb{R}_0^- := \{x \in \mathbb{R} : x \le 0\}$.
        Then $\mathcal{D}$ is open in $\mathbb{R}^{n\times n}$, and every $A \in \mathcal{D}$ admits a unique principal square root $A^{1/2}$ whose eigenvalues lie in the open right half-plane.
        the map $A\mapsto A^{1/2}$ on $\mathcal{D}$ is $C^\infty$.
    \end{lemma}

\subsection{Riemannian Optimization on Embedded Submanifolds}
\label{sec:prelim-riemannian}

    Let $\mathcal{M}\subseteq\mathbb{R}^{n\times p}$ be an embedded submanifold.
    For any point $X\in\mathcal{M}$, the tangent space $\mathrm{T}_X \mathcal{M}$ at $X$ consists of the velocity vectors of all smooth curves on $\mathcal{M}$ passing through $X$,
    \begin{equation*}
        \mathrm{T}_X \mathcal{M}=\{\gamma^\prime (0):\gamma(t) \mathrm{~is~a~smooth~curve~in~} \mathcal{M} \mathrm{~with~} \gamma(0)=X\}.
    \end{equation*}
    The tangent bundle is $\mathrm{T}\mathcal{M}:=\bigcup_{X\in\mathcal{M}}\{X\}\times \mathrm{T}_X\mathcal{M}$.
    We consider Riemannian metrics of the form $g_X(Z_1,Z_2)=\langle Z_1, M_X Z_2\rangle$, where $M_X\in\mathbb{R}^{n\times n}$ is SPD and depends smoothly on $X$.
    These are the tractable metrics of Gao et al.~\cite[Sec.~2.1]{gao2024symplectic}. This formula defines $g_X$ on all of $\mathbb{R}^{n\times p}$.
    The normal space $\mathrm{N}_X\mathcal{M}$ is the $g_X$-orthogonal complement of $\mathrm{T}_X\mathcal{M}$ in $\mathbb{R}^{n\times p}$,
    \begin{equation*}
        \mathrm{N}_X \mathcal{M}=\{N\in\mathbb{R}^{n\times p}:g_X(N,Z)=0 \text{ for all } Z\in \mathrm{T}_X\mathcal{M}\},
    \end{equation*}
    and $\mathcal{P}_{\mathrm{T}_X}$ denotes the $g_X$-orthogonal projector of $\mathbb{R}^{n\times p}$ onto $\mathrm{T}_X\mathcal{M}$.

    Let $\bar f$ be a smooth function on $\mathbb{R}^{n\times p}$ with Euclidean gradient $\nabla \bar f(X)$ at $X$, and let $f=\bar f|_{\mathcal{M}}$.
    The Riemannian gradient~\cite[Sec.~3.6]{absil2008optimization} $\mathrm{grad}\,f(X)\in \mathrm{T}_X\mathcal{M}$ is the tangent vector that satisfies
    \begin{equation*}
        g_X(\mathrm{grad}\,f(X),Z) = \langle \nabla \bar f(X), Z\rangle
        \quad \text{for all } Z\in \mathrm{T}_X\mathcal{M}.
    \end{equation*}
    The Riemannian Hessian~\cite[Def.~5.5.1]{absil2008optimization} of $f$ at $X$ is the linear map
    \begin{equation*}
        \mathrm{hess}\,f(X)[Z]=\nabla_Z\, \mathrm{grad}\,f(X), \qquad Z\in \mathrm{T}_X\mathcal{M},
    \end{equation*}
    where $\nabla$ denotes the Levi-Civita connection on $\mathcal{M}$ associated with $g$.

    A retraction on $\mathcal{M}$~\cite[Def.~4.1.1]{absil2008optimization} is a smooth map $\mathcal{R}:\mathrm{T}\mathcal{M}\to\mathcal{M}$, $(X,Z)\mapsto\mathcal{R}_X(Z)$, such that for every $X\in\mathcal{M}$
    \begin{equation}\label{eq:retraction}
        \mathcal{R}_X(0_X)=X
        \qquad\text{and}\qquad
        \mathrm{D}\mathcal{R}_X(0_X)=\mathrm{id}_{\mathrm{T}_X\mathcal{M}},
    \end{equation}
    where $0_X$ is the zero vector of $\mathrm{T}_X\mathcal{M}$ and $\mathrm{id}_{\mathrm{T}_X\mathcal{M}}$ is the identity map on $\mathrm{T}_X\mathcal{M}$.
    The curve $t\mapsto\mathcal{R}_X(tZ)$ thus starts at $X$ with velocity $Z$, and $\mathcal{R}_X$ turns a tangent step at $X$ into a point of $\mathcal{M}$.
    We also call a map a retraction when it satisfies~\eqref{eq:retraction} and is smooth on an open subset of $\mathrm{T}\mathcal{M}$ that contains every zero vector $0_X$.

    On an embedded submanifold, orthogonal projection defines the vector transport $\mathcal{T}_{Z}(Y):=\mathcal{P}_{\mathrm{T}_{\mathcal{R}_X(Z)}}(Y)$, which depends on the metric $g$ and the retraction $\mathcal{R}$.
    The differentiated retraction~\cite[Eq.~(8.6)]{absil2008optimization} gives a second definition,
    \begin{equation}\label{eq:differential-retraction}
        \mathcal{T}^\mathcal{R}_{\xi}(\eta)
        := \mathrm{D}\mathcal{R}_X(\xi)[\eta]
        = \frac{\mathrm{d}}{\mathrm{d}t} \mathcal{R}_X(\xi+t\eta)\big|_{t=0},
    \end{equation}
    where $\xi$ is the base direction and $\eta$ the transported vector.

\section{The Manifold Structure of \texorpdfstring{$\mathrm{Gq}(P, Q)$}{Gq(P, Q)}}
\label{sec:structure}

    This section introduces the $(P,Q)$-adjoint and proves that $\mathrm{Gq}(P,Q)$ is a closed embedded submanifold.
    It also relates $\mathrm{Gq}(P,Q)$ to canonical pairs by congruence, builds coordinates adapted to the form $\langle\cdot,\cdot\rangle_Q$, and describes the tangent space.
    All results of this section are independent of the Riemannian metric.

\subsection{The \texorpdfstring{$(P,Q)$}{(P,Q)}-Adjoint}
\label{sec:generalized-scalar-product}

    The bilinear form $\langle x,y\rangle_Q = x^\top Q y$ is a (possibly indefinite or skew-symmetric) \emph{scalar product}~\cite{mackey2005structured, benner2022stable}, and it is an inner product only when $Q$ is SPD.
    The $(P, Q)$-adjoint of a matrix $A\in\mathbb{R}^{n\times p}$ is
    \begin{equation}\label{eq:P-Q-adjoint}
        A^{(P, Q)} := P^{-1}A^\top Q \in \mathbb{R}^{p\times n},
    \end{equation}
    characterized by
    \begin{equation}
        \langle A x, y \rangle_Q = \langle x, A^{(P,Q)} y\rangle_P
    \quad \forall x \in \mathbb{R}^p,\ y \in \mathbb{R}^n.
    \end{equation}
    The invertibility of $P$ makes $\langle\cdot,\cdot\rangle_P$ nondegenerate, so that this identity determines $A^{(P,Q)}$ uniquely.
    In this notation, $\mathrm{Gq}(P, Q)$ is the set of $(P, Q)$-orthogonal matrices,
    \begin{equation*}
        \mathrm{Gq}(P, Q) = \{X\in\mathbb{R}^{n\times p}\,:\,X^{(P, Q)} X = I_p\}.
    \end{equation*}
    For square $A$ and $Q=P$, the $(P,Q)$-adjoint becomes the $P$-adjoint
    \begin{equation}\label{eq:P-adjoint}
        A^{(P)} = P^{-1}A^\top P.
    \end{equation}
    If additionally $P=P^\top$ and $P^2=I_p$, then $A^{(P)}=PA^\top P$, recovering the quadratic Lie group adjoint~\cite[Eq.~(9)]{machado2002optimization}. In the special case $P=J_p$, \eqref{eq:P-adjoint} reduces to $A^{(J_p)} = J_p^\top A^\top J_p$, the symplectic adjoint of Peng and Mohseni~\cite{peng2016symplectic} and of Bendokat and Zimmermann~\cite{bendokat2021real}, which equals $A^{-1}$ when $A$ is symplectic.
    When $P=I$, \eqref{eq:P-adjoint} recovers the transpose, $A^{(I)} = A^\top$.

    We denote the sets of $P$-symmetric and $P$-skew-symmetric matrices by
    \begin{subequations}
        \begin{align}
            \mathcal{S}_{\mathrm{sym}}(P)
            &=\{S\in\mathbb{R}^{p\times p}\,:\,S^{(P)} = +S\},
            \label{eq:P-symm}\\
            \mathcal{S}_{\mathrm{skew}}(P)
            &=\{\varOmega\in\mathbb{R}^{p\times p}\,:\,\varOmega^{(P)} = -\varOmega\}.
            \label{eq:P-skew}
        \end{align}
    \end{subequations}
    Every $A\in\mathbb{R}^{p\times p}$ splits uniquely as $A=(A+A^{(P)})/2+(A-A^{(P)})/2$, and the operators
    \begin{equation}\label{eq:proj-adjoint}
        \mathcal{P}_{\mathrm{sym}}^{(P)}(A)=\frac{1}{2}(A+A^{(P)})\quad \text{and}\quad \mathcal{P}_{\mathrm{skew}}^{(P)}(A)=\frac{1}{2}(A-A^{(P)})
    \end{equation}
    represent the $P$-symmetric part and the $P$-skew-symmetric part of $A$, respectively.
    Within $\mathcal{S}_{\mathrm{sym}}(P)$ we single out the subset
    \begin{equation}\label{eq:P-symm-plus}
        \mathcal{S}_{\mathrm{sym}}^{+}(P)=\{S\in\mathcal{S}_{\mathrm{sym}}(P)\,:\,\mathrm{sign}(S) = I_p\},
    \end{equation}
    whose members have all their eigenvalues in the open right half-plane.
    It is open in $\mathcal{S}_{\mathrm{sym}}(P)$, contains $I_p$, and is closed under inversion, since $(S^{-1})^{(P)}=(S^{(P)})^{-1}=S^{-1}$ and $\sigma(S^{-1})$ again lies in the open right half-plane.

    \begin{proposition}\label{prop:adjoint}
        The following statements hold.
        \begin{enumerate}
            \item $\mathcal{S}_{\mathrm{sym}}(P)$ and $\mathcal{S}_{\mathrm{skew}}(P)$ are closed under summation, and under inversion whenever the matrix is nonsingular.
            \item $\dim\mathcal{S}_{\mathrm{sym}}(P)=p(p+\varrho)/2$ and $\dim\mathcal{S}_{\mathrm{skew}}(P)=p(p-\varrho)/2$.
            \item The decomposition $\mathbb{R}^{p\times p} = \mathcal{S}_{\mathrm{sym}}(P)\oplus \mathcal{S}_{\mathrm{skew}}(P)$ holds.
            Moreover, if $S\in\mathcal{S}_{\mathrm{sym}}(P)$ and $\varOmega\in\mathcal{S}_{\mathrm{skew}}(P)$, then $\mathrm{tr}(S\varOmega)=0$ and $\mathrm{tr}(S^\top PP^\top\varOmega)=0$.
            \item The tangent space of $\mathcal{S}_{\mathrm{sym}}(P)$ at every point is $\mathcal{S}_{\mathrm{sym}}(P)$ itself.
            The same statement holds for $\mathcal{S}_{\mathrm{skew}}(P)$.
            \item For any $A\in\mathbb{R}^{n\times p}$, $A^{(P, Q)} A\in \mathcal{S}_{\mathrm{sym}}(P)$.
            \item Let $A \in \mathbb{R}^{p\times p}$ satisfy $\sigma(A) \cap \sigma\bigl(-A^{(P)}\bigr) = \emptyset$.
            Then, for every $B \in \mathbb{R}^{p\times p}$, the generalized Sylvester equation $AX + XA^{(P)} = B$ admits a unique solution $X \in \mathbb{R}^{p\times p}$.
            Moreover, $X \in \mathcal{S}_{\mathrm{sym}}(P)$ when $B \in \mathcal{S}_{\mathrm{sym}}(P)$, and $X \in \mathcal{S}_{\mathrm{skew}}(P)$ when $B \in \mathcal{S}_{\mathrm{skew}}(P)$.
            Both conclusions hold for any SPD matrix $A$.
        \end{enumerate}
    \end{proposition}

    \begin{proof}
        (\textit{a}) If $A,B\in \mathcal{S}_{\mathrm{sym}}(P)$, then $A+B\in \mathcal{S}_{\mathrm{sym}}(P)$, and $A^{-1}\in \mathcal{S}_{\mathrm{sym}}(P)$ for nonsingular $A$ since $(A^{-1})^{(P)}=(A^{(P)})^{-1}$.
        The same argument applies to $\mathcal{S}_{\mathrm{skew}}(P)$.

        (\textit{b}) Through the bijection $S\mapsto \tilde S := PS$, the condition $S=S^{(P)}$ is equivalent to $\tilde S = \varrho \tilde S^\top$.
        This preserves dimensions, giving $\dim\mathcal{S}_{\mathrm{sym}}(P) = \dim\{\tilde S\in\mathbb{R}^{p\times p}: \tilde S=\varrho \tilde S^\top\} = p(p+\varrho)/2$.
        The same argument gives the dimension of $\mathcal{S}_{\mathrm{skew}}(P)$.

        (\textit{c}) This decomposition holds by the splitting  $A=\mathcal{P}_{\mathrm{sym}}^{(P)}(A)+ \mathcal{P}_{\mathrm{skew}}^{(P)}(A)$ and the dimension count $\dim\mathcal{S}_{\mathrm{sym}}(P) + \dim\mathcal{S}_{\mathrm{skew}}(P) =  p^2$.
        Let $S=S^{(P)}$ and $\varOmega=-\varOmega^{(P)}$.
        The definition of the $P$-adjoint gives
        \begin{equation*}
            \mathrm{tr}(S\varOmega)
            =\mathrm{tr}\bigl(\varOmega^{(P)}S^{(P)}\bigr)
            =-\mathrm{tr}(\varOmega S)
            =-\mathrm{tr}(S\varOmega),
        \end{equation*}
        which implies $\mathrm{tr}(S\varOmega)=0$.
        For the weighted identity, observe that $\mathrm{tr}(S^\top PP^\top\varOmega)=\mathrm{tr}((PS)^\top(P\varOmega))$.
        The matrix $PS$ satisfies $(PS)^\top=\varrho PS$, whereas $P\varOmega$ satisfies $(P\varOmega)^\top=-\varrho P\varOmega$.
        One of $PS$ and $P\varOmega$ is therefore symmetric and the other skew-symmetric, and these two subspaces of $\mathbb{R}^{p\times p}$ are Frobenius orthogonal.

        (\textit{d}) $\mathcal{S}_{\mathrm{sym}}(P)$ and $\mathcal{S}_{\mathrm{skew}}(P)$ are linear subspaces of $\mathbb{R}^{p\times p}$ by~(a).
        Hence at any point the tangent space coincides with the subspace itself.

        (\textit{e}) The claim follows from
        \begin{equation*}
            (A^{(P, Q)} A)^{(P)} = P^{-1} (A^{(P, Q)} A)^\top  P = P^{-1} A^\top (P^{-1} A^\top Q )^\top  P = P^{-1} A^\top QA= A^{(P, Q)} A.
        \end{equation*}

        (\textit{f}) The similarity of $A^{(P)} = P^{-1}A^\top P$ and $A^\top$ gives $\sigma(A^{(P)}) = \sigma(A)$.
        The assumed separation $\sigma(A)\cap\sigma(-A^{(P)})=\emptyset$ is exactly the Sylvester--Rosenblum solvability condition, and the Sylvester--Rosenblum theorem~\cite{bhatia1997and} provides the existence and uniqueness of $X$.
        If $B=B^{(P)}$, applying the $P$-adjoint to the equation and using $(X^{(P)})^{(P)}=X$ gives $A(X-X^{(P)})+(X-X^{(P)})A^{(P)}=0$.
        Uniqueness of the solution gives $X=X^{(P)}$.
        The $P$-skew-symmetric case is analogous.
        When $A$ is SPD, $\sigma(A) \subset (0,\infty)$ and $\sigma(-A^{(P)}) = -\sigma(A) \subset (-\infty,0)$.
        The spectra of $A$ and $-A^{(P)}$ are therefore disjoint, and both conclusions apply.
    \end{proof}

\subsection{Embedded Submanifold}
\label{sec:embedded}

    Consider the smooth map
    \begin{equation}\label{eq:defined-function}
        F:\mathbb{R}^{n\times p} \to \mathcal{S}_{\mathrm{sym}}(P)\,:\,X\mapsto X^{(P, Q)}X-I_p.
    \end{equation}
    By Proposition~\ref{prop:adjoint}(e), $F$ indeed takes values in $\mathcal{S}_{\mathrm{sym}}(P)$, and its zero set coincides with $\mathrm{Gq}(P, Q)$.
    We show that $0$ is a regular value of $F$, that is, $\mathrm{D}F(X)$ maps $\mathbb{R}^{n\times p}$ onto $\mathcal{S}_{\mathrm{sym}}(P)$ at every $X\in\mathrm{Gq}(P, Q)$.
    The submersion theorem identifies $\mathrm{Gq}(P, Q)=F^{-1}(0)$ as an embedded submanifold.
    The codomain has to be $\mathcal{S}_{\mathrm{sym}}(P)$, since $\mathrm{D}F(X)$ takes all its values in this subspace.
    Reading $X^\top QX=P$ as $p^2$ independent scalar equations overcounts, owing to the fact that $F(X)$ is $P$-symmetric at every $X$ and therefore carries only $p(p+\varrho)/2$ independent entries by Proposition~\ref{prop:adjoint}(b).

    \begin{proposition}\label{prop:submanifold}
        The generalized quadratic matrix set $\mathrm{Gq}(P, Q)$ is a closed embedded submanifold of $\mathbb{R}^{n\times p}$ of dimension $np - p(p+\varrho)/2$.
    \end{proposition}

    \begin{proof}
        As a quadratic polynomial, $F$ is smooth.
        Continuity of $F$ and closedness of $\{0\}$ in $\mathcal{S}_{\mathrm{sym}}(P)$ imply that $F^{-1}(0) = \mathrm{Gq}(P,Q)$ is closed in $\mathbb{R}^{n \times p}$.
        The Fr\'echet derivative of $F$ at $X$ is $\mathrm{D}F(X)[Z]=X^{(P, Q)}Z+Z^{(P, Q)}X\in\mathcal{S}_{\mathrm{sym}}(P)$.
        It remains to show that $\mathrm{D}F(X)$ is surjective at every $X\in\mathrm{Gq}(P, Q)$.
        Given $\bar Z\in\mathcal{S}_{\mathrm{sym}}(P)$, the choice $Z=X \bar Z/2$ gives
        \begin{equation*}
            \mathrm{D}F(X)[Z]
            =\frac{1}{2}X^{(P,Q)}X\bar Z+\frac{1}{2} \bar Z^{(P)} X^{(P,Q)}X
            = \frac{1}{2}\bar{Z} + \frac{1}{2} \bar Z^{(P)}
            = \bar{Z},
        \end{equation*}
        where the last equality comes from $\bar Z^{(P)} = \bar Z$.
        This proves that $\mathrm{D}F(X)$ has full rank, $\mathrm{rank}(\mathrm{D}F(X)) = \dim(\mathcal{S}_{\mathrm{sym}}(P))$.
        By the submersion theorem~\cite[Prop.~3.3.3]{absil2008optimization}, $\mathrm{Gq}(P, Q)$ is a closed embedded submanifold of $\mathbb{R}^{n\times p}$.
        By Proposition~\ref{prop:adjoint}(b), we have
        \begin{equation*}
            \dim(\mathrm{Gq}(P, Q))=\dim(F^{-1}(0))=\dim(\mathbb{R}^{n\times p})-\mathrm{rank}(\mathrm{D}F)(X) = np - \frac{1}{2}p(p+\varrho),
        \end{equation*}
        where $\varrho$ is defined in \eqref{eq:varrho}.
    \end{proof}

\subsection{Congruence to Canonical Forms}
\label{sec:congruence}

    For nonsingular real $Q\in\mathbb{R}^{n\times n}$, whether symmetric or skew-symmetric, there exist $\varPhi_Q \in \mathrm{O}(n)$ and a real canonical form $\varSigma_Q$ such that $Q = \varPhi_Q \varSigma_Q \varPhi_Q^\top$, where $\varSigma_Q\in\mathbb{R}^{n\times n}$ is the diagonal or block-diagonal matrix
    \begin{equation}\label{eq:spectral-sigma}
        \varSigma_Q = \begin{cases}
            \varLambda_Q=\mathrm{diag}(\lambda_1,\ldots,\lambda_n) &\text{if}~\varrho = +1,\\
            \varLambda_Q\otimes J_2 = \mathrm{blkdiag}(\lambda_1 J_2,\ldots,\lambda_{n/2} J_2)&\text{if}~\varrho = -1,
        \end{cases}
    \end{equation}
    where $J_2 = \left[\begin{smallmatrix} 0 & 1 \\ -1 & 0 \end{smallmatrix}\right]$.
    We order the diagonal entries of $\varLambda_{Q}$ decreasingly.
    For $\varrho = +1$, the diagonal entries of $\varLambda_Q$ are the real eigenvalues of $Q$ and the decomposition is the standard \textit{eigendecomposition}.
    For $\varrho = -1$, the entries of $\varLambda_Q$ are chosen positive, so that the eigenvalues of $Q$ are $\pm i\lambda_j$, and the decomposition is the real \emph{Youla normal form}~\cite{youla1961normal}. Kressner and Mataigne~\cite{kressner2026halving} compute this decomposition by reducing the skew-symmetric eigenvalue problem to a Hermitian one of half the size through a skew-symmetric orthogonal polar factor.
    The decomposition of $P$ is defined in the same manner as that of $Q$.
    For example, for $Q=J_n$ let $\varPhi_{J_n} \in \mathbb{R}^{n \times n}$ be the perfect shuffle permutation matrix, whose $i$-th column is
    \begin{equation*}
        [\varPhi_{J_n}]_{(:, i)} = e_{\lceil i/2 \rceil + n/2 \cdot ((i-1) \bmod 2)},
    \end{equation*}
    where $\lceil \cdot \rceil$ denotes the ceiling function and $e_k$ denotes the $k$-th standard basis vector.
    Then $\varPhi_{J_n}$ is orthogonal and satisfies $J_n = \varPhi_{J_n} (I_{n/2} \otimes J_2) \varPhi_{J_n}^\top$.

    We define the matrix-valued function
    \begin{equation*}
        K_\varrho(Q) = \begin{cases}
            \mathrm{sign}(\varLambda_Q)&\text{if}~\varrho = +1,\\
            J_{n}&\text{if}~\varrho = -1.
        \end{cases}
    \end{equation*}
    The matrices $Q$ and $K_\varrho(Q)$ are \textit{congruent}, that is, $Q = \breve{Q}^\top K_\varrho(Q) \breve{Q}$ for an invertible matrix $\breve{Q}$.
    For $\varrho=+1$ this is Sylvester's law of inertia~\cite[Thm.~8.1.17]{golub2013matrix}, and for $\varrho=-1$ it follows from the Youla normal form.
    The next proposition gives $\breve{Q}$ explicitly in both cases and, inspired by Gao et al.~\cite[Sec.~4]{gao2022optimization}, uses it to relate $\mathrm{Gq}(P, Q)$ to the Stiefel-type manifold $\mathrm{Gq}(K_\varrho(P), K_\varrho(Q))$.

    \begin{proposition}\label{prop:diffeo}
        Let $\varrho\in\{+1,-1\}$.
        Define
        \begin{equation*}
            \breve{Q} = \begin{cases}
            (\mathrm{sign}(\varLambda_Q) \varLambda_Q)^{1/2} \varPhi_Q^\top &\text{if}~\varrho = +1,\\
            \varPhi_{J_{n}} (\varLambda_Q^{1/2}\otimes I_2) \varPhi_Q^\top &\text{if}~\varrho = -1,
            \end{cases}
        \end{equation*}
        so that $Q=\breve Q^\top K_\varrho(Q)\breve Q$, and define $\breve{P}$ in the same way.
        Then
        \begin{equation*}
            \theta\colon\mathrm{Gq}(P, Q)\;\to\;\mathrm{Gq}(K_\varrho(P), K_\varrho(Q)),\quad
            X \mapsto \breve Q\,X\,\breve P^{-1}
        \end{equation*}
        is a diffeomorphism with inverse $\theta^{-1}(Y)=\breve Q^{-1}Y\breve P$.
    \end{proposition}

    \begin{proof}
        For $X\in\mathrm{Gq}(P, Q)$, set $Y=\breve Q X\breve P^{-1}$.
        Then
        \begin{equation*}
            Y^\top K_\varrho(Q)Y
            =\breve P^{-\top}X^\top(\breve Q^\top K_\varrho(Q)\breve Q)X\breve P^{-1}
            =\breve P^{-\top}P\breve P^{-1}
            =K_\varrho(P),
        \end{equation*}
        establishing that $Y\in\mathrm{Gq}(K_\varrho(P), K_\varrho(Q))$.
        Preservation of the inertia and symmetry type under congruence ensures that the pair $(K_\varrho(P),K_\varrho(Q))$ is admissible.
        Both sets are consequently embedded submanifolds by Proposition~\ref{prop:submanifold}.
        The same computation applied to $\theta^{-1}(Y)=\breve Q^{-1}Y\breve P$ shows that $\theta^{-1}$ maps $\mathrm{Gq}(K_\varrho(P), K_\varrho(Q))$ into $\mathrm{Gq}(P, Q)$, proving that $\theta$ is onto.
        As restrictions of linear maps of $\mathbb{R}^{n\times p}$ to embedded submanifolds, both $\theta$ and $\theta^{-1}$ are smooth and satisfy $\theta^{-1} \circ \theta = \theta \circ \theta^{-1} = \mathrm{id}$.
        These properties make $\theta$ a diffeomorphism between the two manifolds.
    \end{proof}

    \begin{remark}\label{rem:congruence-nonunique}
        The congruence factors $\breve Q$ and $\breve P$ are not unique, and Proposition~\ref{prop:diffeo} fixes one choice.
        Congruence preserves the symmetry type, and the diffeomorphism of Proposition~\ref{prop:diffeo} relates only pairs with the same $\varrho$.
        For example, for even $p$ the manifolds $\mathrm{St}(p,n)$ and $\mathrm{Sp}(p,n)$ have the different dimensions $np-p(p+1)/2$ and $np-p(p-1)/2$ by Proposition~\ref{prop:submanifold}.
    \end{remark}

    The boundedness of $\mathrm{Gq}(P,Q)$ is governed by the symmetry type and definiteness of $P$ and $Q$.
    The linearity of the maps underlying $\theta$ and $\theta^{-1}$ in Proposition~\ref{prop:diffeo} implies that $\mathrm{Gq}(P,Q)$ is bounded if and only if $\mathrm{Gq}(K_\varrho(P),K_\varrho(Q))$ is bounded.
    It suffices to analyze the canonical pairs.
    Three regimes arise.
    \begin{enumerate}
        \item \emph{Symmetric definite case} ($\varrho=+1$, $P$ and $Q$ either both positive definite or both negative definite). $K_\varrho(P)=\pm I_p$ and $K_\varrho(Q)=\pm I_n$ with matching signs, and hence $\mathrm{Gq}(K_\varrho(P),K_\varrho(Q))$ is the Stiefel manifold $\mathrm{St}(p,n)$, which is compact.
        \item \emph{Symmetric indefinite case} ($\varrho=+1$, $P$ or $Q$ indefinite). $\mathrm{Gq}(K_\varrho(P),K_\varrho(Q))$ is the unbounded indefinite Stiefel manifold $\mathrm{iSt}(K_\varrho(P),K_\varrho(Q))$ of Tiep et al.~\cite{tiep2026generalized}.
        \item \emph{Skew-symmetric case} ($\varrho=-1$). It corresponds to $\mathrm{Gq}(K_\varrho(P),K_\varrho(Q)) = \mathrm{Sp}(p,n)$, the symplectic Stiefel manifold, which is unbounded~\cite[Sec.~3]{gao2021riemannian}.
    \end{enumerate}

    Consequently, $\mathrm{Gq}(P,Q)$ is \emph{compact} if and only if $P$ and $Q$ are both symmetric and definite of the same sign.
    In all other cases its unboundedness rules out compactness.

\subsection{Coordinate Decomposition}
\label{sec:coordinate-decomposition}
    We extend $X \in \mathrm{Gq}(P, Q)$ to a frame of $\mathbb{R}^n$ adapted to the bilinear form $\langle u, v\rangle_Q = u^\top Q v$, and for this we seek a matrix $X_\perp \in \mathbb{R}^{n \times (n-p)}$ such that
    \begin{equation}\label{eq:Q-orth}
        \langle X, X_\perp \rangle_Q = X^\top Q X_\perp = 0.
    \end{equation}
    For a subspace $V\subseteq\mathbb{R}^n$ we write $V^{\perp_Q}=\{u\in\mathbb{R}^n:\langle u,v\rangle_Q=0~\text{for all}~v\in V\}$ for its $Q$-orthogonal complement, which is well-defined since $Q^\top=\varrho Q$.
    Condition~\eqref{eq:Q-orth} implies that $\mathrm{range}(X_\perp)\subseteq\mathrm{range}(X)^{\perp_Q}$, and $X^\top Q u = 0$ holds exactly when $Qu\in\ker(X^\top)$, so that
    \begin{equation}\label{eq:perpQ}
        \mathrm{range}(X)^{\perp_Q} = Q^{-1}\ker(X^\top),
    \end{equation}
    a subspace of dimension $n-p$.
    Throughout, $X_\perp$ denotes a basis matrix of \eqref{eq:perpQ}, so that $\mathrm{range}(X_\perp)=\mathrm{range}(X)^{\perp_Q}$.

    Before fixing a basis, we verify that the construction is well-defined by showing that $\mathrm{range}(X)$ and $\mathrm{range}(X_\perp)$ are complementary in $\mathbb{R}^n$.
    When $Q$ is SPD, $\langle\cdot,\cdot\rangle_Q$ is an inner product.
    When $Q$ is symmetric indefinite or skew-symmetric, a subspace $V \subset \mathbb{R}^n$ can be degenerate, which means that $V \cap V^{\perp_Q} \neq \{0\}$.
    For instance, the line spanned by a vector $x\neq0$ with $\langle x,x\rangle_Q=0$ is degenerate, and for skew-symmetric $Q$ every nonzero vector has this property.
    The following lemma shows that $\mathrm{range}(X)$ is nondegenerate.

    \begin{lemma}\label{lem:non-degeneracy}
        Let $X \in \mathrm{Gq}(P, Q)$ with $\varrho \in \{+1, -1\}$.
        Then the bilinear form $\langle \cdot, \cdot \rangle_Q$ restricted to $\mathrm{range}(X)$ is nondegenerate.
        Equivalently, $\mathrm{range}(X) \cap \mathrm{range}(X)^{\perp_Q} = \{0\}$.
        Thus, $\mathbb{R}^n = \mathrm{range}(X) \oplus \mathrm{range}(X)^{\perp_Q}$.
    \end{lemma}

    \begin{proof}
        Let $v = X\xi \in \mathrm{range}(X) \cap \mathrm{range}(X)^{\perp_Q}$ for some $\xi \in \mathbb{R}^p$.
        Then
        \begin{equation*}
            0 = \langle X\xi, X\eta \rangle_Q = \xi^\top X^\top Q X \eta = \xi^\top P \eta,
        \end{equation*}
        for every $\eta \in \mathbb{R}^p$.
        The assumption $P\in\mathrm{GL}(p)$ ensures that the bilinear form $(\xi,\eta)\mapsto\xi^\top P\eta$ on $\mathbb{R}^p$ is nondegenerate.
        In particular, $\xi^\top P\eta = 0$ for all $\eta$ implies $\xi=0$.
        The direct-sum decomposition follows from the dimension count $\dim \mathrm{range}(X) + \dim \mathrm{range}(X)^{\perp_Q} = p + (n - p) = n$.
    \end{proof}

    Under the assumptions that $Q$ is nonsingular and $Q^\top = \varrho Q$, nondegeneracy of $\langle\cdot,\cdot\rangle_Q$ on $\mathrm{range}(X)$ carries over to $\mathrm{range}(X_\perp)$.
    We normalize the basis matrix $X_\perp$ by
    \begin{equation}\label{eq:normalization}
        X_\perp^\top Q^\top Q X_\perp = I_{n-p},
    \end{equation}
    via a weighted Gram--Schmidt step.
    Specifically, if $\tilde X_\perp$ is any basis matrix of $\mathrm{range}(X)^{\perp_Q}$ with Cholesky factorization $\tilde X_\perp^\top Q^\top Q \tilde X_\perp = L L^\top$, then setting $X_\perp := \tilde X_\perp L^{-\top}$ achieves \eqref{eq:normalization}.
    This normalization determines the basis $X_\perp$ uniquely up to right-multiplication by $\mathrm{O}(n - p)$.
    Condition~\eqref{eq:normalization} matches the normalizations of Gao et al.~\cite[Prop.~4.1]{gao2021riemannian} for $\mathrm{Sp}(p, n)$ and of Edelman et al.~\cite[Eq.~(2.5)]{edelman1998geometry} for $\mathrm{St}(p, n)$.

    The full column rank of $X$ gives the Moore--Penrose pseudoinverse $X^\dagger = (X^\top X)^{-1} X^\top$, with $XX^\dagger$ the the Euclidean orthogonal projector onto $\mathrm{range}(X)$.
    By \eqref{eq:normalization}, the columns of $QX_\perp$ are orthonormal, and $\mathrm{range}(QX_\perp)=\ker(X^\top)$ is the Euclidean orthogonal complement of $\mathrm{range}(X)$ in $\mathbb{R}^n$.
    Hence $(QX_\perp)(QX_\perp)^\top$ and $XX^\dagger$ are the Euclidean orthogonal projectors onto these complementary subspaces, and
    \begin{equation}\label{eq:Euclidean-proj}
        QX_\perp (QX_\perp)^\top = I_{n} - XX^\dagger.
    \end{equation}
    In contrast with the Euclidean case, $X_\perp X_\perp^\dagger + XX^\dagger = I_n$ holds exactly when $\mathrm{range}(X_\perp)=Q^{-1}\ker(X^\top)$ coincides with the Euclidean orthogonal complement $\ker(X^\top)$ of $\mathrm{range}(X)$, that is, when $\mathrm{range}(X)$ is $Q$-invariant.

    We denote the Gram matrix of $X_\perp$ by
    \begin{equation}\label{eq:O}
        O := X_\perp^\top Q X_\perp,
    \end{equation}
    which depends on $X$ and inherits the symmetry type of $Q$, $\rho(O) = \varrho$.
    We now state the coordinate decomposition of $\mathbb{R}^{n \times p}$.

    \begin{proposition}\label{prop:CD}
        Let $E := [X, X_\perp] \in \mathbb{R}^{n \times n}$.
        Then $E$ is invertible and block-diagonalizes the $Q$-bilinear form with diagonal blocks $P$ on $\mathrm{range}(X)$ and $O$ on $\mathrm{range}(X_\perp)$.
        Consequently, every $Z \in \mathbb{R}^{n \times p}$ admits the unique coordinate representation
        \begin{equation}\label{eq:Z-express}
            Z = E \begin{bmatrix} \varOmega_Z \\ W_Z \end{bmatrix}
            = X \varOmega_Z + X_\perp W_Z,
        \end{equation}
        with $\varOmega_Z = X^{(P, Q)} Z$ and $W_Z = X_\perp^{(O, Q)} Z$.
    \end{proposition}

    \begin{proof}
        From $X^{(P, Q)} X = I_p$ and $X^\top Q X_\perp = 0$ (by construction of $X_\perp$), a direct computation gives
        \begin{equation*}
            E^\top Q E
            = \begin{bmatrix} X^\top \\ X_\perp^\top \end{bmatrix} Q \begin{bmatrix} X & X_\perp \end{bmatrix}
            = \mathrm{blkdiag}(P, O).
        \end{equation*}
        By Lemma~\ref{lem:non-degeneracy}, $\mathbb{R}^n = \mathrm{range}(X) \oplus \mathrm{range}(X_\perp)$, which makes $E$ invertible.
        The block form $E^\top Q E = \mathrm{blkdiag}(P, O)$ is nonsingular, and $P \in \mathrm{GL}(p)$ leaves $O$ nonsingular as well.
        Applying $E^{-1} = (E^\top Q E)^{-1} E^\top Q$ yields
        \begin{equation*}
            E^{-1}
            = \mathrm{blkdiag}(P^{-1}, O^{-1})
            \begin{bmatrix} X^\top Q \\ X_\perp^\top Q \end{bmatrix}
            = \begin{bmatrix} X^{(P, Q)} \\ X_\perp^{(O, Q)} \end{bmatrix}
            \in \mathbb{R}^{n \times n}.
        \end{equation*}
        Applying $E^{-1}$ on the left to any $Z$ produces the unique coordinate pair $(\varOmega_Z, W_Z)$ in \eqref{eq:Z-express}.
    \end{proof}

    With $E = [X, X_\perp]$, Proposition~\ref{prop:CD} gives $E^{-1} = \mathrm{blkdiag}(P^{-1}, O^{-1}) E^\top Q$.
    Expanding $I_n = E E^{-1} = X P^{-1} X^\top Q + X_\perp O^{-1} X_\perp^\top Q$ yields the oblique resolution of the identity
    \begin{equation}\label{eq:oblique-resolution}
        X X^{(P, Q)} + X_\perp X_\perp^{(O, Q)} = I_n,
    \end{equation}
    which generalizes the resolution $XX^\top+X_\perp X_\perp^\top=I_n$ on $\mathrm{St}(p,n)$.

    \begin{remark}\label{rem:projectors}
        Proposition~\ref{prop:CD} provides the oblique direct-sum decomposition
        \begin{equation*}
            \mathbb{R}^{n\times p} = X\cdot\mathbb{R}^{p\times p} \;\oplus\; X_\perp\cdot\mathbb{R}^{(n-p)\times p}.
        \end{equation*}
        The oblique projectors $\Pi_X := XX^{(P,Q)}$ and $\Pi_{X_\perp} := X_\perp X_\perp^{(O,Q)}$ realize this decomposition through $X \varOmega_Z = \Pi_X Z$ and $X_\perp W_Z = \Pi_{X_\perp} Z$, and satisfy $\Pi_X^2 = \Pi_X$, $\Pi_{X_\perp}^2 = \Pi_{X_\perp}$, $\Pi_X\Pi_{X_\perp} = \Pi_{X_\perp}\Pi_X = 0$, and $\Pi_X + \Pi_{X_\perp} = I_n$.
        These projectors are self-adjoint with respect to the $Q$-weighted bilinear form $\langle\cdot,\cdot\rangle_Q$.
        Since $\Pi_X$ projects onto $\mathrm{range}(X)$ along $\mathrm{range}(QX)^\perp$, it is Frobenius-self-adjoint precisely when $\mathrm{range}(QX)=\mathrm{range}(X)$, that is, when $\mathrm{range}(X)$ is $Q$-invariant.
        If $Q$ is a scalar multiple of the identity, this holds at every $X$.
    \end{remark}

\subsection{Tangent Space}
\label{sec:tangent}
    We next characterize the tangent space, which is independent of the Riemannian metric.

    \begin{proposition}\label{prop:tangent-space}
    The tangent space of $\mathrm{Gq}(P, Q)$ at $X\in \mathrm{Gq}(P, Q)$ admits the three expressions
    \begin{subequations}
        \begin{align}
            \mathrm{T}_X\mathrm{Gq}(P, Q)
            &=\{Z\in\mathbb{R}^{n\times p}\,:\,X^{(P, Q)}Z+Z^{(P, Q)}X=0\}\label{eq:TS-Gq-a}\\
            &=\{X\varOmega +X_\perp W\,:\,\varOmega\in\mathcal{S}_{\mathrm{skew}}(P),W\in\mathbb{R}^{(n-p)\times p}\}
            \label{eq:TS-Gq-b}\\
            &=\{\varGamma X\,:\,\varGamma \in\mathcal{S}_{\mathrm{skew}}(Q)\}
            \label{eq:TS-Gq-c}.
        \end{align}
    \end{subequations}
    \end{proposition}

    \begin{proof}
        Since $\mathrm{Gq}(P, Q)$ is the zero set of the function $F$ in \eqref{eq:defined-function}, the expression \eqref{eq:TS-Gq-a} follows from $\mathrm{T}_X\mathrm{Gq}(P, Q)=\ker(\mathrm{D}F(X))$~\cite[Eq.~(3.19)]{absil2008optimization}. By Proposition~\ref{prop:CD}, every $Z$ has the form $Z=X\varOmega +X_\perp W$, and substituting it into \eqref{eq:TS-Gq-a} yields
        \begin{equation*}
        0=P^{-1}X^\top Q(X\varOmega +X_\perp W)+P^{-1}(X\varOmega +X_\perp W)^\top QX=\varOmega+\varOmega^{(P)},
        \end{equation*}
        which gives $\varOmega^{(P)}=-\varOmega$ and establishes \eqref{eq:TS-Gq-b}.

        For \eqref{eq:TS-Gq-c}, we first show with \eqref{eq:TS-Gq-a} that $\varGamma X$ is tangent for every $\varGamma \in\mathcal{S}_{\mathrm{skew}}(Q)$.
        If $\varGamma^{(Q)}=-\varGamma$, then
        \begin{equation*}
            X^{(P, Q)}\varGamma X+(\varGamma X)^{(P, Q)}X
            = X^{(P, Q)}\varGamma X+X^{(P, Q)}\varGamma^{(Q)}X=0.
        \end{equation*}
        The reverse inclusion follows from a dimension count.
        Setting $\varTheta :=E^\top Q\varGamma E$ gives $\varTheta=-\varrho \varTheta^\top$.
        We partition $\varTheta = \left[\begin{smallmatrix} \varTheta_{11}&\varTheta_{12}\\ \varTheta_{21}&\varTheta_{22}
        \end{smallmatrix}\right]$ conformally with $E=[X,X_\perp]$.
        Since $X = E\left[\begin{smallmatrix} I_p\\0\end{smallmatrix}\right]$, the subspace in \eqref{eq:TS-Gq-c} has dimension
        \begin{equation*}
            \begin{aligned}
                \dim\{\varGamma X\,:\,\varGamma \in\mathcal{S}_{\mathrm{skew}}(Q)\}
                &=\dim\{\varTheta\left[\begin{smallmatrix} I_p\\0\end{smallmatrix}\right]\,:\,\varTheta = -\varrho \varTheta^\top\}\\
                &=\dim\{\left[\begin{smallmatrix} \varTheta_{11}\\ \varTheta_{21}\end{smallmatrix}\right]\,:\,\varTheta_{11}^\top = -\varrho \varTheta_{11}\in\mathbb{R}^{p\times p},\varTheta_{21}\in\mathbb{R}^{(n-p)\times p}\}\\
                &=\frac{1}{2}p(p-\varrho)+(n-p)p
                = \dim(\mathrm{Gq}(P, Q)).
            \end{aligned}
        \end{equation*}
        Equality of dimensions proves that the subspace in \eqref{eq:TS-Gq-c} is the tangent space \eqref{eq:TS-Gq-a}.
    \end{proof}

    \begin{remark}\label{rem:tangent-charts}
        The expression \eqref{eq:TS-Gq-b} is a bijective coordinate chart, parameterizing $\mathrm{T}_X\mathrm{Gq}(P, Q)$ by exactly $\dim(\mathrm{Gq}(P,Q))$ variables.
        In contrast, $\dim(\mathcal{S}_{\mathrm{skew}}(Q)) = n(n-\varrho)/2$ exceeds $\dim(\mathrm{Gq}(P,Q))$ by $(n-p)(n-p-\varrho)/2$.
        The parameterization \eqref{eq:TS-Gq-c} is therefore overparameterized except at $p=n$ and, when $\varrho=+1$, at $p=n-1$.
        The space $\mathcal{S}_{\mathrm{skew}}(Q)$ is closed under the matrix commutator and is therefore a Lie subalgebra of $\mathbb{R}^{n\times n}$, and \eqref{eq:TS-Gq-c} realizes $\mathrm{T}_X \mathrm{Gq}(P,Q)$ as its image under the linear map $\varGamma \mapsto \varGamma X$.
        The preimage of any tangent vector is an affine subspace of $\mathcal{S}_{\mathrm{skew}}(Q)$ of dimension $(n-p)(n-p-\varrho)/2$, a single point in the two exceptional cases.
        The quadratic matrix Lie group is the instance $P=Q$ with $p=n$, where $\mathrm{Gq}(Q,Q)=\mathrm{QLie}(Q)$, the block $W$ of \eqref{eq:TS-Gq-b} is void, $X$ is invertible, and $\varGamma = X\varOmega X^{-1}$ recovers the Lie-algebra element of \eqref{eq:TS-Gq-c} from the chart.
    \end{remark}

    Among the matrices $\varGamma \in \mathcal{S}_{\mathrm{skew}}(Q)$ with $Z = \varGamma X$, which form the affine family of Remark~\ref{rem:tangent-charts}, we choose a low-rank representative.
    With $\varTheta =E^\top Q\varGamma E$ as in the proof of Proposition~\ref{prop:tangent-space}, so that $\varTheta^\top=-\varrho\varTheta$, we take
    \begin{equation*}
        \varTheta=H\begin{bmatrix} I_p&0\end{bmatrix}-\varrho \begin{bmatrix} I_p\\0\end{bmatrix}H^\top
    \end{equation*}
    with some $H\in\mathbb{R}^{n\times p}$.
    Then $\mathrm{rank}(\varGamma)=\mathrm{rank}(\varTheta)\leq2p$, and this low-rank family still parameterizes all of \eqref{eq:TS-Gq-b}.
    Letting $G:=Q^{-1}E^{-\top}H$, we have
    \begin{equation*}\begin{aligned}
            \varGamma
            &=Q^{-1}E^{-\top}H\begin{bmatrix}I_p&0\end{bmatrix}E^{-1} -\varrho Q^{-1}E^{-\top}\begin{bmatrix}I_p\\0\end{bmatrix}H^\top E^{-1}\\
            &= GX^{(P, Q)} - X G^{(P, Q)},
        \end{aligned}
    \end{equation*}
    where the second equality follows from Proposition~\ref{prop:CD}.
    Within this low-rank family, passing from \eqref{eq:TS-Gq-b} to \eqref{eq:TS-Gq-c} amounts to finding a suitable $G\in\mathbb{R}^{n\times p}$.
    Writing $G=X\varOmega_G+X_\perp W_G$ by \eqref{eq:Z-express}, we get $\varGamma X=X(\varOmega_G-\varOmega_G^{(P)})+X_\perp W_G$.
    For each $Z =\varGamma X= X\varOmega_Z +X_\perp W_Z$ with $\varOmega_Z\in\mathcal{S}_{\mathrm{skew}}(P)$, we obtain $\varOmega_G = \varOmega_Z/2$, $W_G=W_Z$, and $G = E\mathrm{blkdiag}(I_p/2,I_{n-p}) E^{-1}Z$.
    Each $Z$ in \eqref{eq:TS-Gq-b} therefore has a unique representation within this family as $Z=\varGamma_{X,Z}X$ with
    \begin{equation}\label{eq:expression-Gamma}
        \varGamma_{X,Z} = GX^{(P, Q)} - X G^{(P, Q)}~~\text{and}~~G = (I_n-\frac{1}{2}XX^{(P, Q)})Z.
    \end{equation}
    Equation~\eqref{eq:expression-Gamma} is the $\mathrm{Gq}(P, Q)$ counterpart of the skew-symmetric matrix through which a tangent vector is represented on Stiefel-type manifolds.
    On the Stiefel and symplectic Stiefel manifolds it becomes
    \begin{equation*}
        \varGamma_{X,Z} = \begin{cases}
            GX^\top-XG^\top \text{ with } G=(I_n-\frac{1}{2}XX^\top)Z, &\text{in } \mathrm{St}(p,n),\\
            (G(XJ_p)^\top+XJ_pG^\top)J_n \text{ with } G=(I_n-\frac{1}{2}XJ_pX^\top J_n^\top)Z, &\text{in } \mathrm{Sp}(p,n),
        \end{cases}
    \end{equation*}
    which correspond to the formulas of Wen and Yin~\cite[Eq.~(5)]{wen2013feasible} and of Gao et al.~\cite[Eq.~(5.3)]{gao2021riemannian}, respectively.

    At fixed $X$, the next lemma bounds $\|\varGamma_{X,Z}\|_2$ in terms of $\|Z\|_2$, for use in Section~\ref{sec:retraction} in establishing well-definedness of the retraction near the zero tangent vector.

    \begin{lemma}\label{lem:Gamma-bound}
    For every $X\in\mathrm{Gq}(P, Q)$ and every $Z\in \mathrm{T}_X\mathrm{Gq}(P,Q)$,
    \begin{equation}
    \|\varGamma_{X,Z}\|_2
    \;\le\; (2+\pi_X)\,\|X\|_2\,\|P^{-1}\|_2\,\|Q\|_2\,\|Z\|_2,
    \qquad \pi_X := \|XX^{(P,Q)}\|_2 .
    \label{eq:gamma-bound}
    \end{equation}
    Moreover $1\le\pi_X \le \|X\|_2^2\,\|P^{-1}\|_2\,\|Q\|_2$, and $\pi_X = 1$ whenever $XX^{(P,Q)}$ is a Euclidean orthogonal projector, for instance on $\mathrm{St}(p,n)$.
    In that case the factor $2+\pi_X$ in \eqref{eq:gamma-bound} equals $3$.
    In particular, $\|\varGamma_{X,Z}\|_2 \le r$ whenever
    $\|Z\|_2 \le r\big/\bigl((2+\pi_X)\|X\|_2\|P^{-1}\|_2\|Q\|_2\bigr)$.
    \end{lemma}

    \begin{proof}
        For every matrix $A$, the definition of $A^{(P,Q)}$ and the submultiplicativity of the spectral norm give $\|A^{(P,Q)}\|_2 \leq \|P^{-1}\|_2\,\|Q\|_2\,\|A\|_2$.
        The nonzero idempotent $XX^{(P,Q)}$ has spectral norm $\pi_X=\|XX^{(P,Q)}\|_2\ge 1$.
        When $p=n$, this idempotent is $I_n$.
        The bound above gives
        \begin{equation*}
            \pi_X \le \|X\|_2\,\|X^{(P,Q)}\|_2 \le \|X\|_2^2\,\|P^{-1}\|_2\,\|Q\|_2.
        \end{equation*}
        By \eqref{eq:expression-Gamma} and the triangle inequality,
        \begin{equation*}
            \|G\|_2 \leq \|I_n - \frac{1}{2}XX^{(P,Q)}\|_2\,\|Z\|_2 \leq \bigl(1+\frac{1}{2}\pi_X\bigr)\,\|Z\|_2.
        \end{equation*}
        Combining these bounds with the triangle inequality gives
        \begin{equation*}
            \begin{aligned}
                \|\varGamma_{X,Z}\|_2
                &\leq \|G\|_2\,\|X^{(P,Q)}\|_2 + \|X\|_2\,\|G^{(P,Q)}\|_2 \\
                &\leq 2\,\|G\|_2\,\|X\|_2\,\|P^{-1}\|_2\,\|Q\|_2 \\
                &\leq (2+\pi_X)\,\|X\|_2\,\|P^{-1}\|_2\,\|Q\|_2\,\|Z\|_2.
            \end{aligned}
        \end{equation*}
    \end{proof}

\section{Riemannian Geometry of \texorpdfstring{$\mathrm{Gq}(P, Q)$}{Gq(P, Q)}}
\label{sec:manifold}

    This section equips $\mathrm{Gq}(P, Q)$ with the canonical-like metric and derives the normal space, the orthogonal projection, the Riemannian gradient and Hessian, and the Levi-Civita connection, mainly following Absil et al.~\cite{absil2008optimization} and Gao et al.~\cite{gao2021riemannian}.

\subsection{The Canonical-Like Metric}
\label{sec:metric}

    The metrics of Section~\ref{sec:prelim-riemannian} are the tractable metrics $g_X(Z_1,Z_2)=\langle Z_1, M_XZ_2\rangle$ of Gao et al.~\cite[Sec.~2.1]{gao2024symplectic}. Under a tractable metric, the orthogonal projection onto the tangent space can require the solution of a Lyapunov matrix equation, as Gao et al.~\cite[Prop.~2]{gao2024symplectic} and Tiep et al.~\cite[Eq.~(8)]{tiep2026generalized} observe.
    The canonical-like metric gives a Lyapunov-free projection through an SPD weight $M_X^\beta$ adapted to the coordinate decomposition~\eqref{eq:Z-express}.

    \begin{definition}\label{def:metric-canonical}
        Given $X\in \mathrm{Gq}(P, Q)$, let $O$ be the matrix defined in \eqref{eq:O}.
        For tangent vectors $Z_k = X\varOmega_k + X_\perp W_k$, $k=1,2$, the canonical-like metric with parameter $\beta>0$ is
        \begin{equation}\label{eq:metric-canonical}
            \begin{aligned}
                g_X^\beta(Z_1,Z_2)
                &=\beta \mathrm{tr}(\varOmega_1^\top PP^\top \varOmega_2)+ \mathrm{tr}(W_1^\top OO^\top W_2)\\
                &=\mathrm{tr}(Z_1^\top \underbrace{E^{-\top}\,\mathrm{blkdiag}(\beta PP^\top, OO^\top)\,E^{-1}}_{:=M_X^\beta}Z_2).
            \end{aligned}
        \end{equation}
    \end{definition}

    The matrix $M_X^\beta$ is SPD and depends smoothly on $X$, making $g^\beta$ a Riemannian metric.
    The two coordinate blocks in \eqref{eq:metric-canonical} are $g_X^\beta$-orthogonal, and $\beta$ sets their relative weight.
    From Proposition~\ref{prop:CD} and \eqref{eq:Euclidean-proj}, the matrix $M_X^\beta$ and its inverse admit the explicit expressions
    \begin{align}
        M_X^\beta &= \beta QXX^\top Q^\top +I_n-XX^\dagger,\label{eq:metric-1}\\
        (M_X^\beta)^{-1} &= \frac{1}{\beta}X(PP^\top)^{-1}X^\top+(I_n-XX^{(P, Q)})(I_n-XX^{(P, Q)})^\top.
        \label{eq:metric-2}
    \end{align}

    When $\mathrm{Gq}(P, Q)$ specializes to the Stiefel-type manifolds of Section~\ref{sec:Stiefel-type-manifold}, \eqref{eq:metric-canonical} recovers their standard metrics.
    For $\mathrm{St}(p,n)$, the weight $M_X^\beta = I_n - (1-\beta) XX^\top$ gives the one-parameter family of metrics of H\"{u}per et al.~\cite{knut2021lagrangian}, which contains the Euclidean metric ($\beta=1$) and the canonical metric ($\beta=1/2$).
    For $\mathrm{Sp}(p,n)$, it coincides with the metric of Gao et al.~\cite[Eq.~(4.6)]{gao2021riemannian} when the complement basis there is taken as $X_\perp(X_\perp^\top J_n X_\perp)^{-1}$.
    For $\mathrm{iSt}(J, A)$, it reduces to the metric of Tiep et al.~\cite[Eq.~(14)]{tiep2026generalized}.

    The normal space and the projection admit explicit expressions independent of $\beta$.

    \begin{proposition}\label{prop:normal-space-canonical}
        The normal space at $X\in\mathrm{Gq}(P, Q)$ with respect to $g_X^\beta$ is
        \begin{equation}\label{eq:normal-space-canonical}
            \mathrm{N}_X\mathrm{Gq}(P, Q) = \{XS:S\in\mathcal{S}_{\mathrm{sym}}(P)\}.
        \end{equation}
        Given $Y\in\mathbb{R}^{n\times p}$, the orthogonal projector onto $\mathrm{T}_X\mathrm{Gq}(P, Q)$ is
        \begin{equation}\label{eq:projection-tangent-canonical}
            \mathcal P_{\mathrm{T}_X}(Y)=Y - X\mathcal{P}_{\mathrm{sym}}^{(P)}(X^{(P, Q)}Y).
        \end{equation}
    \end{proposition}

    \begin{proof}
        By \eqref{eq:TS-Gq-b}, a tangent vector has the form $Z = X\varOmega_Z+X_\perp W_Z$ with $\varOmega_Z\in\mathcal{S}_{\mathrm{skew}}(P)$ and $W_Z\in\mathbb{R}^{(n-p)\times p}$.
        By \eqref{eq:Z-express}, an ambient vector has the form $N = XS_N+X_\perp W_N$.
        The condition $g_X^{\beta}(N,Z)=0$ for all tangent $Z$ is equivalent to
        \begin{equation*}
            \beta\,\mathrm{tr}(S_N^\top PP^\top\varOmega_Z)+\mathrm{tr}(W_N^\top OO^\top W_Z)=0
            \quad\text{for all } \varOmega_Z\in\mathcal{S}_{\mathrm{skew}}(P),\ W_Z\in\mathbb{R}^{(n-p)\times p}.
        \end{equation*}
        The positive definiteness of $OO^\top$ forces $W_N=0$ in the second term.
        Observe that the first term satisfies $\mathrm{tr}(S_N^\top PP^\top\varOmega_Z)=\mathrm{tr}\bigl((PS_N)^\top(P\varOmega_Z)\bigr)$.
        Besides, the map $A\mapsto PA$ sends $\mathcal{S}_{\mathrm{skew}}(P)$ to the ordinary matrix subspace obeying $(P\varOmega_Z)^\top=-\varrho(P\varOmega_Z)$.
        Consequently, its Frobenius orthogonal complement consists of matrices $PS_N$ satisfying $(PS_N)^\top=\varrho PS_N$, which is equivalent to $S_N^{(P)}=S_N$.
        In turn, this yields $S_N\in\mathcal{S}_{\mathrm{sym}}(P)$.
        Accordingly, the residual $Y-\mathcal{P}_{\mathrm{T}_X}(Y)$ lies in $\mathrm{N}_X\mathrm{Gq}(P, Q)$, so that $Y-\mathcal{P}_{\mathrm{T}_X}(Y) = XS$ for some $S\in\mathcal{S}_{\mathrm{sym}}(P)$.
        Applying \eqref{eq:TS-Gq-a} to the tangent vector $\mathcal{P}_{\mathrm{T}_X}(Y)=Y-XS$ gives $X^{(P, Q)}(Y-XS)+(Y-XS)^{(P, Q)}X=0$.
        From Proposition~\ref{prop:adjoint}(f) with $A=I_p$, this equation has the unique solution $S = \mathcal{P}_{\mathrm{sym}}^{(P)}(X^{(P, Q)}Y)$, where $\mathcal{P}_{\mathrm{sym}}^{(P)}$ is defined in \eqref{eq:proj-adjoint}.
        Substituting this $S$ into $\mathcal{P}_{\mathrm{T}_X}(Y)=Y-XS$ yields \eqref{eq:projection-tangent-canonical}.
    \end{proof}

    The block separation in Definition~\ref{def:metric-canonical}, together with the weight $\beta PP^\top$ on the first block, reduces the projection equation to $S+S^{(P)}=X^{(P, Q)}Y+Y^{(P, Q)}X$.
    Thus \eqref{eq:projection-tangent-canonical} expresses the projection in closed form without a Lyapunov solve.
    As an embedded submanifold of $\mathbb{R}^{n\times p}$, $\mathrm{Gq}(P, Q)$ carries the vector transport by orthogonal projection~\cite[Ex.~8.1.8]{absil2008optimization},
    \begin{equation}\label{eq:VT-orth-proj}
        \mathcal{T}_{Z}^{\mathsf{proj}}(W)
        = (I_n-YY^{(P,Q)})W+Y\mathcal{P}_{\mathrm{skew}}^{(P)}(Y^{(P,Q)} W)\in \mathrm{T}_Y\mathrm{Gq}(P, Q),
    \end{equation}
    where $Y=\mathcal{R}_X(Z)$ for a retraction $\mathcal{R}$ on $\mathrm{Gq}(P, Q)$.

\subsection{Riemannian Gradient and Hessian}
\label{sec:grad-hess}
    Let $\bar f$ be a smooth extension of $f$ such that $f = \bar{f}|_{\mathrm{Gq}(P, Q)}$.
    By a result of Boumal~\cite[Prop.~3.61]{boumal2023introduction}, the Riemannian gradient of $f$ with respect to $g^\beta$ is
    \begin{equation}\label{eq:grad}
        \mathrm{grad}\,f(X)
        = \mathcal{P}_{\mathrm{T}_X}((M_X^\beta)^{-1}\nabla \bar f(X)).
    \end{equation}
    It recovers the Riemannian gradients under the canonical metric on $\mathrm{St}(p,n)$~\cite[Sec.~2.4.4]{edelman1998geometry} and on $\mathrm{Sp}(p,n)$~\cite[Eq.~(4.10)]{gao2021riemannian}.

    We denote the set of critical points of \eqref{eq:GQ-problem} by $\mathcal{X} := \{X\in\mathrm{Gq}(P, Q)\,:\, \mathrm{grad}\,f(X) =0\}$.
    \begin{remark}\label{rem:critical-points}
        With $\Pi_X=XX^{(P,Q)}$ from Remark~\ref{rem:projectors} and $V:=(M_X^\beta)^{-1}\nabla \bar f(X)$, \eqref{eq:metric-2} gives $X^{(P,Q)}V = (PP^\top)^{-1}X^\top\nabla \bar f(X)/\beta$ and $(I_n-\Pi_X)V = (I_n-\Pi_X)(I_n-\Pi_X)^\top\nabla \bar f(X)$.
        By \eqref{eq:projection-tangent-canonical}, the gradient \eqref{eq:grad} splits as $\mathrm{grad}\,f(X) = (I_n-\Pi_X)V + X\mathcal{P}_{\mathrm{skew}}^{(P)}(X^{(P,Q)}V)$, with summands in the complementary subspaces of Remark~\ref{rem:projectors}.
        Since $(I_n-\Pi_X)(I_n-\Pi_X)^\top$ and $(I_n-\Pi_X)^\top$ have the same kernel, $X\in\mathcal{X}$ if and only if $(I_n-\Pi_X)^\top\nabla \bar f(X)=0$ and $(PP^\top)^{-1}X^\top\nabla \bar f(X) \in\mathcal{S}_{\mathrm{sym}}(P)$.
        The first condition reads $\nabla \bar f(X) = QXP^{-1}X^\top\nabla \bar f(X)$, the second is equivalent to $\nabla \bar f(X)^\top X \in \mathcal{S}_{\mathrm{sym}}(P)$, and neither depends on $\beta$.
        The first condition holds exactly when $\nabla \bar f(X)=QX\varLambda$ for some $\varLambda\in\mathbb{R}^{p\times p}$, in which case $\varLambda=P^{-1}X^\top\nabla \bar f(X)$ and the second condition reads $\varLambda^\top=\varrho\varLambda$.
        Equivalently, $X\in\mathcal{X}$ if and only if $\nabla \bar f(X)=QX\varLambda$ for some $\varLambda$ with $\varLambda^\top=\varrho\varLambda$.
        On $\mathrm{St}(p,n)$ the two conditions combine into the first-order optimality condition $\nabla \bar f(X)=X\nabla \bar f(X)^\top X$ of Wen and Yin~\cite[Lem.~1]{wen2013feasible}.
    \end{remark}

    The operator-valued Levi-Civita formulas of Nguyen~\cite{nguyen2023operator} and their implementation on $\mathrm{Sp}(p,n)$ by Gao et al.~\cite{gao2024symplectic} yield the following closed form for the Riemannian Hessian.
    \begin{proposition}\label{prop:hess}
        Let $\bar f$ be a smooth extension of $f$ with $f = \bar f|_{\mathrm{Gq}(P, Q)}$, and let $X\in\mathrm{Gq}(P, Q)$.
        For every $Z\in \mathrm{T}_X\mathrm{Gq}(P, Q)$,
        \begin{equation}\label{eq:hess}
            \mathrm{hess}\,f(X)[Z] = \mathcal P_{\mathrm{T}_X}\bigl((M_X^\beta)^{-1}(\mathfrak{I}_1 + \mathfrak{I}_2 - \mathfrak{I}_3 + \mathfrak{I}_4)\bigr),
        \end{equation}
        where the terms $\mathfrak{I}_1,\dots,\mathfrak{I}_4$ and the metric-compatibility correction $\mathcal{K}$ are specified in the proof, and $\tilde{\mathcal{K}}(Z,U)$ is the Euclidean adjoint, in its third argument, of the linear map $V\mapsto \mathrm{D}_V M_X^\beta\cdot U$.
    \end{proposition}
    \begin{proof}
        We instantiate the operator-valued Levi-Civita formula of Nguyen~\cite[Thm.~6.1]{nguyen2023operator}, whose four contributions are
    \begin{equation*}
        \begin{aligned}
            \mathfrak{I}_1&=\nabla^2 \bar f(X)[Z],\\
            \mathfrak{I}_2&=(M_X^\beta)\cdot \mathrm{D}_Z\mathcal P_{\mathrm{T}_X}\big((M_X^\beta)^{-1}\nabla \bar f(X)\big),\\
            \mathfrak{I}_3&=\mathrm{D}_Z M_X^\beta \cdot \big((M_X^\beta)^{-1}\nabla \bar f(X)\big),\\
            \mathfrak{I}_4&=\mathcal{K}(Z, \mathrm{grad}\,f(X)),
        \end{aligned}
    \end{equation*}
    where the correction term $\mathcal{K}(Z,U)$, which makes the connection metric-compatible, is
    \begin{equation*}
        \mathcal{K}(Z,U)=\frac{1}{2}(\mathrm{D}_ZM_X^\beta \cdot U+\mathrm{D}_UM_X^\beta \cdot Z-\tilde{\mathcal{K}}(Z,U)),
    \end{equation*}
    and $\tilde{\mathcal{K}}(Z,U)$ is defined for $Z,U\in \mathrm{T}_X\mathrm{Gq}(P, Q)$ by $\langle \tilde{\mathcal{K}}(Z,U), V \rangle = \langle Z, \mathrm{D}_VM_X^\beta \cdot U\rangle$ for all $V\in\mathbb{R}^{n\times p}$.
    The directional derivatives of $\mathcal P_{\mathrm{T}_X}$ and $M_X^\beta$ are
    \begin{equation*}
        \begin{aligned}
            \mathrm{D}_Z\mathcal P_{\mathrm{T}_X}(Y) &=-X\mathcal{P}_{\mathrm{sym}}^{(P)}(Z^{(P, Q)}Y) - Z\mathcal{P}_{\mathrm{sym}}^{(P)}(X^{(P, Q)}Y),\\
            \mathrm{D}_ZM_X^\beta&=2\mathcal{P}_{\mathrm{sym}}^{(I_n)} (\beta QXZ^\top Q^\top-(I_n-XX^\dagger)ZX^{\dagger}),
        \end{aligned}
    \end{equation*}
    and $\tilde{\mathcal{K}}(Z,U)$ has the closed form
    \begin{equation*}
        \tilde{\mathcal{K}}(Z,U) = 2\beta Q^\top \mathcal{P}_{\mathrm{sym}}^{(I_n)}(UZ^\top)QX-2(I_n-XX^\dagger)\mathcal{P}_{\mathrm{sym}}^{(I_n)}(UZ^\top)(X^\dagger)^\top.
    \end{equation*}
    A direct computation shows that this expression satisfies $\langle \tilde{\mathcal{K}}(Z,U), V \rangle = \langle Z, \mathrm{D}_VM_X^\beta \cdot U\rangle$ for every $V\in\mathbb{R}^{n\times p}$, including directions outside $\mathrm{T}_X\mathrm{Gq}(P, Q)$.
    This identity characterizes $\tilde{\mathcal{K}}(Z,U)$ as the Euclidean adjoint of the map $V\mapsto \mathrm{D}_VM_X^\beta\cdot U$ with respect to its third argument.
    With these derivatives, the four terms are
    \begin{equation*}
        \begin{aligned}
            \mathfrak{I}_1&=\nabla^2 \bar f(X)[Z],\\
            \mathfrak{I}_2&= -M_X^\beta Z\mathcal{P}_{\mathrm{sym}}^{(P)}(X^{(P, Q)}(M_X^\beta)^{-1}\nabla \bar f(X)),\\
            \mathfrak{I}_3&=2\mathcal{P}_{\mathrm{sym}}^{(I_n)} (\beta QXZ^\top Q^\top-(I_n-XX^\dagger)ZX^{\dagger}) \cdot (M_X^\beta)^{-1}\nabla \bar f(X),\\
            \mathfrak{I}_4&=\mathcal{P}_{\mathrm{sym}}^{(I_n)} (\beta QXZ^\top Q^\top-(I_n-XX^\dagger)ZX^{\dagger})\cdot \mathrm{grad}\,f(X)\\
            &\quad+\mathcal{P}_{\mathrm{sym}}^{(I_n)} (\beta QX\mathrm{grad}\,f(X)^\top Q^\top-(I_n-XX^\dagger)\mathrm{grad}\,f(X)X^{\dagger})\cdot Z\\
            &\quad-\beta Q^\top \mathcal{P}_{\mathrm{sym}}^{(I_n)}(\mathrm{grad}\,f(X)Z^\top)QX+(I_n-XX^\dagger)\mathcal{P}_{\mathrm{sym}}^{(I_n)}(\mathrm{grad}\,f(X)Z^\top)(X^\dagger)^\top.
        \end{aligned}
    \end{equation*}
    \begin{sloppypar}
    In $\mathfrak{I}_2$ we have omitted the term $-M_X^\beta X\mathcal{P}_{\mathrm{sym}}^{(P)}(Z^{(P, Q)}(M_X^\beta)^{-1}\nabla \bar f(X))$ of $M_X^\beta\,\mathrm{D}_Z\mathcal P_{\mathrm{T}_X}\big((M_X^\beta)^{-1}\nabla \bar f(X)\big)$, since $(M_X^\beta)^{-1}$ maps it into the normal space \eqref{eq:normal-space-canonical}, on which $\mathcal{P}_{\mathrm{T}_X}$ in \eqref{eq:hess} vanishes.
    \end{sloppypar}
    Substituting these terms into \eqref{eq:hess} gives the Riemannian Hessian in closed form.
    \end{proof}

    \begin{remark}\label{rem:christoffel}
        The proof of Proposition~\ref{prop:hess} determines the full Levi-Civita connection of $g^\beta$.
        With $V:=(M_X^\beta)^{-1}\nabla \bar f(X)$, the product rule gives $\mathrm{D}_ZV=(M_X^\beta)^{-1}(\mathfrak{I}_1-\mathfrak{I}_3)$, the construction of $\mathfrak{I}_2$ gives $\mathcal{P}_{\mathrm{T}_X}(\mathrm{D}_Z\mathcal{P}_{\mathrm{T}_X}(V))=\mathcal{P}_{\mathrm{T}_X}((M_X^\beta)^{-1}\mathfrak{I}_2)$, and the idempotence of $\mathcal{P}_{\mathrm{T}_X}$ yields $\mathcal{P}_{\mathrm{T}_X}(\mathrm{D}_Z\,\mathrm{grad}\,f)=\mathcal{P}_{\mathrm{T}_X}((M_X^\beta)^{-1}(\mathfrak{I}_1+\mathfrak{I}_2-\mathfrak{I}_3))$.
        With $\mathfrak{I}_4=\mathcal{K}(Z,\mathrm{grad}\,f(X))$, \eqref{eq:hess} states that
        \begin{equation}\label{eq:connection}
            \nabla_ZY=\mathcal{P}_{\mathrm{T}_X}\bigl(\mathrm{D}_ZY+(M_X^\beta)^{-1}\mathcal{K}(Z,Y)\bigr)
        \end{equation}
        for $Y=\mathrm{grad}\,f$.
        Under $Y\mapsto hY$ with $h$ smooth, both sides of \eqref{eq:connection} acquire the same term $(Zh)Y$.
        Their difference is therefore tensorial in $Y$ and is determined by $Y(X)$.
        Every $Y_0\in \mathrm{T}_X\mathrm{Gq}(P, Q)$ is the Riemannian gradient at $X$ of the linear function $W\mapsto\langle M_X^\beta Y_0,W\rangle$.
        Equation~\eqref{eq:connection} consequently holds for every smooth tangent vector field $Y$ and every $Z\in \mathrm{T}_X\mathrm{Gq}(P, Q)$.
        The tensorial part of \eqref{eq:connection},
        \begin{equation}\label{eq:christoffel}
            \mathfrak{G}_X^\beta(Z,U):=\mathcal{P}_{\mathrm{T}_X}\bigl((M_X^\beta)^{-1}\mathcal{K}(Z,U)\bigr),
        \end{equation}
        is the Christoffel operator of $g^\beta$ in the ambient coordinates of $\mathbb{R}^{n\times p}$.
        Consequently, a smooth curve $c$ on $\mathrm{Gq}(P, Q)$ with $c(0)=X$ and $\dot c(0)=Z$ has covariant acceleration
        \begin{equation}\label{eq:cov-accel}
            \frac{\mathrm{D}}{\mathrm{d}t}\dot c(0)=\mathcal{P}_{\mathrm{T}_X}(\ddot c(0))+\mathfrak{G}_X^\beta(Z,Z)
        \end{equation}
        at $t=0$, where $\ddot c$ is the second derivative of $c$ as a curve in $\mathbb{R}^{n\times p}$, and $c$ is a geodesic of $g^\beta$ exactly when $\mathcal{P}_{\mathrm{T}_c}(\ddot c)+\mathfrak{G}_c^\beta(\dot c,\dot c)=0$ holds along it.
    \end{remark}

    The procedure in the proof of Proposition~\ref{prop:hess} extends to all tractable metrics by replacing $M_X^\beta$ with the appropriate weight.
    For the weighted Euclidean metric~\cite[Sec.~2.3]{gao2024symplectic}, the weight $M_X^\beta \equiv M$ is a constant SPD matrix, for which $\mathfrak{I}_3 = \mathfrak{I}_4 = 0$ and the Riemannian Hessian reduces to
    \begin{equation*}
        \mathrm{hess}\,f(X)[Z] = \mathcal{P}_{\mathrm{T}_X} \left( M^{-1} \nabla^2 \bar f(X)[Z] + \mathrm{D}_Z \mathcal{P}_{\mathrm{T}_X} \left( M^{-1} \nabla \bar f(X) \right) \right),
    \end{equation*}
    which matches the formulas of Boumal~\cite[Eq.~(7.29)]{boumal2023introduction} and of Gao et al.~\cite[Eq.~(36)]{gao2024symplectic}. For canonical-like metrics, the procedure recovers the Riemannian Hessians on $\mathrm{St}(p,n)$~\cite[Thm.~6.1]{nguyen2023operator}, on $\mathrm{iSt}(J, A)$~\cite[Sec.~4.4]{sato2026second}, and on $\mathrm{Sp}(p,n)$~\cite[Thm.~2]{gao2024symplectic}.

\section{Retraction}
\label{sec:retraction}
    A retraction on $\mathrm{Gq}(P,Q)$ maps a tangent step at $X$ to a point of $\mathrm{Gq}(P,Q)$.
    This section generalizes the Cayley, quasi-geodesic, and polar retractions from Stiefel-type manifolds to $\mathrm{Gq}(P, Q)$, together with their vector transports and their second-order corrections.

\subsection{Cayley-Based Retraction}
\label{sec:Cay-retraction}

    For $A\in\mathbb{R}^{n\times n}$ with $I_n-A$ nonsingular, the Cayley transform is
    \begin{equation*}
        \mathrm{cay}(A)=(I_n-A)^{-1}(I_n+A),
    \end{equation*}
    which maps elements of quadratic Lie algebras to quadratic Lie groups~\cite[Lem.~8.7]{hairer2006geometric} and is therefore a common building block of structure-preserving retractions.
    In view of \eqref{eq:TS-Gq-c}, we define the following Cayley-based retraction.

    \begin{theorem}\label{thm:Cayley-retraction}
        Given $X\in\mathrm{Gq}(P, Q)$, let $Z\in \mathrm{T}_X\mathrm{Gq}(P, Q)$ be a tangent vector satisfying
        \begin{equation*}
            \|Z\|_2 < \frac{2}{(2+\pi_X)\|X\|_2 \|P^{-1}\|_2\|Q\|_2}.
        \end{equation*}
        For $\varGamma_{X,Z}$ as defined in \eqref{eq:expression-Gamma}, the mapping
        \begin{equation}\label{eq:Cayley-retraction}
            \mathcal R_X^{\mathsf{cay}}(Z)
            = \mathrm{cay}\left(\frac{1}{2} \varGamma_{X,Z}\right) X
            = (I_n-\frac{1}{2}\varGamma_{X,Z})^{-1}(I_n+\frac{1}{2}\varGamma_{X,Z})X
        \end{equation}
        defines the Cayley-based retraction on $\mathrm{Gq}(P, Q)$.

        Write
        \begin{equation*}
            \mathcal{D}_X^{\mathsf{cay}}=\bigl\{Z\in \mathrm{T}_X\mathrm{Gq}(P, Q):I_n-\frac{1}{2}\varGamma_{X,Z}~\text{is nonsingular}\bigr\}
        \end{equation*}
        for the domain of \eqref{eq:Cayley-retraction}, the largest set on which the right-hand side is defined.
        This set is open, $\mathcal R_X^{\mathsf{cay}}$ is smooth on it, and the norm ball displayed above is a neighborhood of $0_X$ contained in it.

        The map $\mathcal R_X^{\mathsf{cay}}$ restricts to a diffeomorphism from an open neighborhood $\mathcal{N}_X\subseteq\mathcal{D}_X^{\mathsf{cay}}$ of $0_X$ onto its image $\mathcal{V}_X:=\mathcal R_X^{\mathsf{cay}}(\mathcal{N}_X)$.
        For every $Y\in\mathcal{V}_X$, the inverse Cayley-based retraction admits the closed-form expression
        \begin{equation}\label{eq:inverse-Cayley}
            (\mathcal R_X^{\mathsf{cay}})^{-1}(Y) = 2Y(I_p+X^{(P,Q)}Y)^{-1}+2X(I_p+Y^{(P,Q)}X)^{-1}-2X.
        \end{equation}
        On $\mathcal{V}_X$ the spectral separation $\sigma(X^{(P,Q)}Y)\cap\{-1\}=\emptyset$ holds, which makes $I_p+X^{(P,Q)}Y$ invertible.

        More generally, let $Y\in \mathrm{Gq}(P, Q)$ satisfy that separation condition and let $Z$ denote the right-hand side of \eqref{eq:inverse-Cayley}.
        Then $Z\in \mathrm{T}_X\mathrm{Gq}(P, Q)$, and $\mathcal R_X^{\mathsf{cay}}(Z)=Y$ whenever $Z\in\mathcal{D}_X^{\mathsf{cay}}$.
        Conversely, \eqref{eq:inverse-Cayley} recovers every $Z\in\mathcal{D}_X^{\mathsf{cay}}$ for which $I_p+X^{(P,Q)}\mathcal R_X^{\mathsf{cay}}(Z)$ is invertible.
    \end{theorem}

    \begin{proof}
        By Lemma~\ref{lem:Gamma-bound} with $r=2$, the condition on $Z$ gives $\|\varGamma_{X,Z}\|_2 < 2$.
        The Neumann series shows that $I_n - \varGamma_{X,Z}/2$ is invertible and $\mathcal{R}_X^{\mathsf{cay}}(Z)$ is defined.

        The relation $\varGamma_{X,Z} \in \mathcal{S}_{\mathrm{skew}}(Q)$ gives $I_n + \varGamma_{X,Z}/2 = I_n - \varGamma_{X,Z}^{(Q)}/2$ and also $(I_n-\varGamma_{X,Z}/2)^{-1} = (I_n+\varGamma_{X,Z}^{(Q)}/2)^{-1}$.
        The Cayley transform thus maps $Q$-skew-symmetric matrices into $\mathrm{QLie}(Q)$,
        \begin{equation*}
            \mathrm{cay}(\frac12\varGamma_{X,Z})^{(Q)} = (I_n+\frac12\varGamma_{X,Z})^{-1} (I_n-\frac12\varGamma_{X,Z}) = \mathrm{cay}(\frac12\varGamma_{X,Z})^{-1},
        \end{equation*}
        and consequently, by \eqref{eq:TS-Gq-c}, $\mathcal{R}_X^{\mathsf{cay}}(Z) \in \mathrm{Gq}(P, Q)$.
        It remains to verify the two conditions in~\eqref{eq:retraction}.
        The linear dependence of $\varGamma_{X,Z}$ on $Z$ and the smoothness of matrix inversion imply that $\mathcal{D}_X^{\mathsf{cay}}$ is open and $\mathcal R_X^{\mathsf{cay}}(\cdot)$ is smooth on it.
        The norm ball of the statement lies in $\mathcal{D}_X^{\mathsf{cay}}$ by the Neumann series argument.
        At $Z=0_X$, $\varGamma_{X,0_X}=0_n$ gives $\mathcal R_X^{\mathsf{cay}}(0_X)=X$.
        Using $\mathrm{D}\mathrm{cay}(0)[B]=2B$, we obtain
        \begin{equation*}
            \left.\frac{\mathrm{d}}{\mathrm{d}t}\mathcal R_X^{\mathsf{cay}}(tZ)\right|_{t=0} = \left.\mathrm{D}\mathrm{cay}\left(\frac{1}{2}\varGamma_{X,tZ}\right)\left[\frac{1}{2}\varGamma_{X,Z}\right]X\right|_{t=0} =\varGamma_{X,Z}X=Z.
        \end{equation*}
        It follows that $\mathrm{D}\mathcal R_X^{\mathsf{cay}}(0_X) = \mathrm{id}_{\mathrm{T}_X\mathrm{Gq}(P, Q)}$.
        These identities verify the retraction conditions for $\mathcal R^{\mathsf{cay}}$.

        For the inverse, write $B=(I_p+X^{(P,Q)}Y)^{-1}$ and $C=(I_p+Y^{(P,Q)}X)^{-1}$.
        Since the $P$-adjoint commutes with inversion and $(X^{(P,Q)}Y)^{(P)}=Y^{(P,Q)}X$, we have $B^{(P)}=C$ and $C^{(P)}=B$, and $C$ is defined whenever $B$ is.

        We first prove that the $Z$ of \eqref{eq:inverse-Cayley} is tangent and that $\mathcal R_X^{\mathsf{cay}}(Z)=Y$ as soon as $Z\in\mathcal{D}_X^{\mathsf{cay}}$.
        From $X^{(P,Q)}X=I_p$ and $X^{(P,Q)}YB=I_p-B$ we obtain
        \begin{equation*}
            X^{(P,Q)}Z = 2(I_p-B)+2C-2I_p = 2(C-B),
            \qquad
            Z^{(P,Q)}X = \bigl(X^{(P,Q)}Z\bigr)^{(P)} = 2(B-C),
        \end{equation*}
        which implies $X^{(P,Q)}Z+Z^{(P,Q)}X=0$ and $Z\in \mathrm{T}_X\mathrm{Gq}(P, Q)$ by \eqref{eq:TS-Gq-a}.
        Consequently $G=(I_n-XX^{(P, Q)}/2)Z=2YB+XB+XC-2X$ and $X^{(P,Q)}G=C-B$, whence, using $Y^{(P,Q)}Y=I_p$,
        \begin{equation*}
            X^{(P,Q)}(X+Y)=B^{-1},
            \qquad
            G^{(P,Q)}(X+Y)=B+C+(B+C-2I_p)X^{(P,Q)}Y.
        \end{equation*}
        Substituting these into $\varGamma_{X,Z}=GX^{(P, Q)} - X G^{(P, Q)}$ from \eqref{eq:expression-Gamma} and simplifying gives the identity
        \begin{equation*}
            \varGamma_{X,Z}(X+Y)=GB^{-1}-X\bigl(B+C+(B+C-2I_p)X^{(P,Q)}Y\bigr)=2(Y-X),
        \end{equation*}
        which rearranges to $(I_n+\varGamma_{X,Z}/2)X=(I_n-\varGamma_{X,Z}/2)Y$.
        If $Z\in\mathcal{D}_X^{\mathsf{cay}}$, then $I_n-\varGamma_{X,Z}/2$ is invertible and this identity reads $\mathcal R_X^{\mathsf{cay}}(Z)=Y$.
        This argument uses only $X,Y\in\mathrm{Gq}(P, Q)$ and the invertibility of $I_p+X^{(P,Q)}Y$ and of $I_n-\varGamma_{X,Z}/2$, with no smallness condition on $Z$.

        For the remaining direction, fix $Z\in\mathcal{D}_X^{\mathsf{cay}}$, set $Y=\mathcal R_X^{\mathsf{cay}}(Z)$ and assume $I_p+X^{(P,Q)}Y$ invertible.
        As in the work of Zhu and Sato~\cite[Sec.~5.1]{zhu2020riemannian}, substituting \eqref{eq:expression-Gamma} into \eqref{eq:Cayley-retraction} yields
        \begin{equation*}
            2(Y-X)=G(I_p+X^{(P,Q)}Y) -XG^{(P,Q)}(X+Y).
        \end{equation*}
        Right-multiplication by $B=(I_p+X^{(P,Q)}Y)^{-1}$ shows that $G$ has the form $G=2YB+XS$ for some $S\in\mathbb{R}^{p\times p}$.
        Substitution into the preceding identity, using $B^{(P)}=C$ and $Y^{(P,Q)}(X+Y)=C^{-1}$, gives $X(S-S^{(P)})B^{-1}=0$.
        The full column rank of $X$ and invertibility of $B$ imply $S=S^{(P)}$, or equivalently $S\in\mathcal{S}_{\mathrm{sym}}(P)$.
        It remains to determine $S$.
        It follows from~\eqref{eq:TS-Gq-a}, \eqref{eq:expression-Gamma}, and the identity $(I_n-XX^{(P,Q)}/2)^{-1}=I_n+XX^{(P,Q)}$ that
        \begin{equation*}
            \begin{aligned}
                0&=X^{(P,Q)}(I_n+XX^{(P,Q)})G+G^{(P,Q)}(I_n+XX^{(P,Q)})X\\
                &=4X^{(P,Q)}Y(I_p+X^{(P,Q)}Y)^{-1} + 4(I_p+Y^{(P,Q)}X)^{-1}Y^{(P,Q)}X +4S\\
                &=4(S-(I_p+X^{(P,Q)}Y)^{-1}-(I_p+Y^{(P,Q)}X)^{-1}+2I_p),
            \end{aligned}
        \end{equation*}
        which determines $S$ as
        \begin{equation*}
            S = (I_p+Y^{(P,Q)}X)^{-1} +(I_p+X^{(P,Q)}Y)^{-1} -2I_p.
        \end{equation*}
        The invertibility of $I_p+X^{(P,Q)}Y$ ensures that $S$ is well-defined.
        Combining these results with $Z=(I_n-XX^{(P,Q)}/2)^{-1}G=(I_n+XX^{(P,Q)})G$ gives
        \begin{equation*}
            \begin{aligned}
                Z&=(I_n+XX^{(P,Q)})(2Y(I_p+X^{(P,Q)}Y)^{-1}+XS)\\
                &=2(I_n+XX^{(P,Q)})Y(I_p+X^{(P,Q)}Y)^{-1}+2XS\\
                &=2Y(I_p+X^{(P,Q)}Y)^{-1}
                +2X(I_p+Y^{(P,Q)}X)^{-1}-2X,
            \end{aligned}
        \end{equation*}
        so that \eqref{eq:inverse-Cayley} returns $Z$ itself.
        Finally, $\mathcal R_X^{\mathsf{cay}}(0_X)=X$ and $I_p+X^{(P,Q)}X=2I_p$.
        By continuity, both hypotheses hold on an open neighborhood of $0_X$ inside $\mathcal{D}_X^{\mathsf{cay}}$.
        Shrinking that neighborhood to $\mathcal{N}_X$ if necessary, $\mathrm{D}\mathcal R_X^{\mathsf{cay}}(0_X)=\mathrm{id}$ and the inverse function theorem make $\mathcal R_X^{\mathsf{cay}}$ a diffeomorphism from $\mathcal{N}_X$ onto its image.
        The two directions above identify \eqref{eq:inverse-Cayley} with its inverse there.
    \end{proof}

    The factor $1/2$ in $\mathrm{cay}(\varGamma_{X,Z}/2)$ and the factor $1/2$ in \eqref{eq:expression-Gamma} play different roles.
    The first gives $\mathrm{D}\mathcal R_X^{\mathsf{cay}}(0_X)=\mathrm{id}$, and the second gives $\varGamma_{X,Z}X = Z$.

    Theorem~\ref{thm:Cayley-retraction} extends the Cayley-based retractions on $\mathrm{St}(p,n)$~\cite[Eq.~(7)]{wen2013feasible} and on $\mathrm{Sp}(p,n)$~\cite[Prop.~5.3]{gao2021riemannian} to $\mathrm{Gq}(P, Q)$.
    A geometric reading of the same map starts from \eqref{eq:TS-Gq-c}, which leads to the differential equation
    \begin{equation}\label{eq:differential-equation}
        \dot{\gamma}(t) = \varGamma \gamma(t), \quad \gamma(0) \in \mathrm{Gq}(P, Q),
    \end{equation}
    where $\varGamma \in \mathcal{S}_{\mathrm{skew}}(Q)$ and $I_n - t\varGamma/2$ is invertible for all $t \in [0,1]$.
    As in the works of Gao et al.~\cite[Eq.~(5.5)]{gao2021riemannian} and Wen and Yin~\cite[Eq.~(15)]{wen2013feasible}, the trapezoidal rule applied to \eqref{eq:differential-equation} over $[0, t]$ gives
    \begin{equation*}
        \gamma(t) = \gamma(0) +\frac{t}{2}\varGamma\big(\gamma(0)+\gamma(t)\big),
    \end{equation*}
    whose solution $\gamma(t)=\mathrm{cay}(t\varGamma/2)\gamma(0)\in \mathrm{Gq}(P, Q)$ is the Cayley transform of Theorem~\ref{thm:Cayley-retraction}.

    The next proposition derives the differential of the Cayley-based retraction.
    \begin{proposition}\label{prop:VT-Cayley}
        The differentiated retraction of \eqref{eq:Cayley-retraction} is
        \begin{equation}\label{eq:VT-Cayley}
            \mathcal{T}_Z^{\mathsf{cay}}(W)
            = (I_n-\frac{1}{2}\varGamma_{X,Z})^{-1}\varGamma_{X,W}(I_n-\frac{1}{2}\varGamma_{X,Z})^{-1}X.
        \end{equation}
        For $Z=tW$,
        \begin{equation}\label{eq:VT-Cayley-ZW}
            \mathcal{T}_{t W}^{\mathsf{cay}}(W) = (I_n-\frac{t}{2}\varGamma_{X,W})^{-2}W.
        \end{equation}
    \end{proposition}

    \begin{proof}
        By \eqref{eq:Cayley-retraction},
        \begin{equation*}
            (I_n-\frac{1}{2}\varGamma_{X,Z}-\frac{t}{2}\varGamma_{X,W})\mathcal R_X^{\mathsf{cay}}(Z+tW) = (I_n+\frac{1}{2}\varGamma_{X,Z}+\frac{t}{2}\varGamma_{X,W})X.
        \end{equation*}
        Differentiating both sides with respect to $t$ gives
        \begin{equation*}
            -\frac{1}{2}\varGamma_{X,W}\mathcal R_X^{\mathsf{cay}}(Z+tW)+(I_n-\frac{1}{2}\varGamma_{X,Z}-\frac{t}{2}\varGamma_{X,W})\frac{\mathrm{d}}{\mathrm{d}t}\mathcal R_X^{\mathsf{cay}}(Z+tW)
            = \frac{1}{2}\varGamma_{X,W} X.
        \end{equation*}
        By \eqref{eq:differential-retraction}, the differentiated retraction is
        \begin{equation*}
            \begin{aligned}
                \mathcal{T}_Z^{\mathsf{cay}}(W)
                &= \left.\frac{\mathrm{d}}{\mathrm{d}t}\mathcal R_X^{\mathsf{cay}}(Z+tW)\right|_{t=0}\\
                &= \frac{1}{2}(I_n-\frac{1}{2}\varGamma_{X,Z})^{-1}\varGamma_{X,W}(X + \mathcal R_X^{\mathsf{cay}}(Z))\\
                &= \frac{1}{2}(I_n-\frac{1}{2}\varGamma_{X,Z})^{-1}\varGamma_{X,W}\left(I_n + (I_n-\frac{1}{2}\varGamma_{X,Z})^{-1}(I_n+\frac{1}{2}\varGamma_{X,Z})\right)X\\
                &= (I_n-\frac{1}{2}\varGamma_{X,Z})^{-1}\varGamma_{X,W}(I_n-\frac{1}{2}\varGamma_{X,Z})^{-1}X.
            \end{aligned}
        \end{equation*}
        This is \eqref{eq:VT-Cayley}, and \eqref{eq:VT-Cayley-ZW} follows with $Z=tW$ since $\varGamma_{X,W}$ commutes with $(I_n-t\varGamma_{X,W}/2)^{-1}$.
    \end{proof}

    The main cost of \eqref{eq:Cayley-retraction} and \eqref{eq:VT-Cayley-ZW} is the inversion of the $n\times n$ matrix $I_n-t\varGamma_{X,Z}/2$, which takes $\mathcal{O}(n^3)$ operations.
    The low-rank representation~\eqref{eq:expression-Gamma} factors $\varGamma_{X,Z}=UV$ with $U=\begin{bmatrix}G&X\end{bmatrix}\in\mathbb{R}^{n\times 2p}$ and $V=\begin{bmatrix} X^{(P, Q)}\\-G^{(P, Q)} \end{bmatrix}\in\mathbb{R}^{2p\times n}$.
    The Sherman--Morrison--Woodbury (SMW) formula~\cite[Eq.~(2.1.4)]{golub2013matrix} gives
    \begin{equation*}
        (I_n-\frac{1}{2}\varGamma_{X,Z})^{-1} = I_n+\frac{1}{2}U(I_{2p}-\frac{1}{2}V U)^{-1}V,
    \end{equation*}
    reducing \eqref{eq:Cayley-retraction} to the update
    \begin{equation*}
        \mathcal R_X^{\mathsf{cay}}(Z)  = X + U(I_{2p} - \frac{1}{2}V U)^{-1}V X.
    \end{equation*}
    The vector transports \eqref{eq:VT-Cayley} and \eqref{eq:VT-Cayley-ZW} admit the analogous updates
    \begin{equation*}
        \begin{aligned}
            \mathcal{T}_Z^{\mathsf{cay}}(W)
            &= \bigl(I_n + \frac{1}{2} U(I_{2p}-\frac{1}{2} VU)^{-1}V\bigr)\varGamma_{X,W}\bigl(I_n + \frac{1}{2} U(I_{2p}-\frac{1}{2} VU)^{-1}V\bigr)X,\\
            \mathcal{T}_{t W}^{\mathsf{cay}}(W)
            &= \bigl(I_n+\frac{t}{2}U_W(I_{2p}-\frac{t}{2}V_W U_W)^{-1}V_W\bigr)^2 W,
        \end{aligned}
    \end{equation*}
    where $\varGamma_{X,W}=U_WV_W$ is the low-rank representation associated with the direction $W$.

\subsection{Quasi-Geodesic-Based Retraction}
\label{sec:qgeo-Retraction}

    For a Riemannian metric $g$, a smooth curve $\gamma:[0,1]\to\mathrm{Gq}(P,Q)$ is a geodesic exactly when it is stationary for the energy
    \begin{equation*}
        E(\gamma)=\frac{1}{2}\int_0^1 g_{\gamma(t)}(\dot\gamma(t),\dot\gamma(t))\,\mathrm{d}t
    \end{equation*}
    under smooth variations with both endpoints fixed.
    Sufficiently short geodesic segments minimize length between their endpoints.
    Remark~\ref{rem:christoffel} makes the geodesic equation of the metric $g^\beta$ in \eqref{eq:metric-canonical} explicit through the covariant acceleration \eqref{eq:cov-accel}, namely $\mathcal{P}_{\mathrm{T}_\gamma}(\ddot\gamma)+\mathfrak{G}_\gamma^\beta(\dot\gamma,\dot\gamma)=0$.
    The Christoffel operator \eqref{eq:christoffel} can contribute a nonzero tangential term.
    The ambient acceleration $\ddot\gamma$ of a $g^\beta$-geodesic therefore need not lie in the normal space \eqref{eq:normal-space-canonical}.
    The weight $M_X^\beta$ of \eqref{eq:metric-1} also depends on $X$ through the Euclidean projector $XX^\dagger$.
    To obtain an explicit curve, we use the auxiliary bilinear form
    \begin{equation*}
        \tilde{g}:\mathbb{R}^{n\times p}\times \mathbb{R}^{n\times p}\to \mathbb{R}, (Z_1,Z_2)\mapsto \mathrm{tr}(Z_1^{(P, Q)} Z_2),
    \end{equation*}
    which is induced by the $(P,Q)$-adjoint and leads to a quasi-geodesic.
    The form $\tilde g$ is symmetric for both $\varrho = \pm 1$, since $P^\top = \varrho P$ and $Q^\top = \varrho Q$ give
    \begin{equation*}
        \begin{aligned}
            \mathrm{tr}(Z_1^{(P, Q)} Z_2)
            &= \mathrm{tr}(P^{-1}Z_1^\top Q Z_2)
            = \mathrm{tr}(Z_2^\top Q^\top Z_1 P^{-\top})\\
            &= \varrho^2\,\mathrm{tr}(P^{-1}Z_2^\top Q Z_1)
            = \mathrm{tr}(Z_2^{(P, Q)} Z_1).
        \end{aligned}
    \end{equation*}
    It is nondegenerate and in general indefinite, making $\tilde g$ a pseudo-Riemannian metric on $\mathbb{R}^{n\times p}$.

    \begin{lemma}\label{lem:tildeg-normal}
        Let $X\in\mathrm{Gq}(P,Q)$ and let $\mathrm{N}_X^{\tilde g}$ be the normal space under the metric $\tilde g$.
        Then $\mathrm{N}_X^{\tilde g}=\{XS:S\in\mathcal{S}_{\mathrm{sym}}(P)\}$, which is the normal space \eqref{eq:normal-space-canonical} under $g^\beta$.
    \end{lemma}
    \begin{proof}
        For $N=XS+X_\perp W$ and $Z=X\varOmega+X_\perp W_Z\in \mathrm{T}_X\mathrm{Gq}(P,Q)$ with $\varOmega\in\mathcal{S}_{\mathrm{skew}}(P)$, direct computation using $X^{(P,Q)}X=I_p$ and $X^{(P,Q)}X_\perp=0$ gives $\tilde g(N,Z)=\mathrm{tr}(S^{(P)}\varOmega)+\mathrm{tr}(W^{(P,O)}W_Z)$.
        The first term vanishes for all $\varOmega\in\mathcal{S}_{\mathrm{skew}}(P)$ if and only if $S\in\mathcal{S}_{\mathrm{sym}}(P)$, and the second vanishes for all $W_Z\in\mathbb{R}^{(n-p)\times p}$ if and only if $W=0$.
    \end{proof}

    Following H\"{u}per et al.~\cite[Sec.~5]{knut2021lagrangian}, we consider the Lagrangian
    \begin{equation*}\mathcal{L}(\gamma,\ell) = \int_0^1 \left\{ \tilde g(\dot\gamma,\dot\gamma) -\mathrm{tr}\bigl(\ell(\gamma^{(P,Q)}\gamma-I_p)\bigr) \right\}\,\mathrm{d}t,
    \end{equation*}
    with the Lagrange multiplier $\ell(t)\in\mathcal{S}_{\mathrm{sym}}(P)$.
    Collecting the kinetic terms gives the stationarity condition
    \begin{equation*}
        \frac{\mathrm{d}}{\mathrm{d}t}\bigl(Q\dot\gamma P^{-1}\bigr)=-Q\gamma \ell P^{-1},\qquad
        \gamma^{(P, Q)}\gamma=I_p.
    \end{equation*}
    We next determine $\ell$.
    The acceleration $\ddot\gamma$ of the quasi-geodesic lies in the normal space under $\tilde{g}$, which equals \eqref{eq:normal-space-canonical} by Lemma~\ref{lem:tildeg-normal}.
    The resulting equation is $\ddot \gamma + \gamma \ell=0$ with $\ell \in \mathcal{S}_{\mathrm{sym}}(P)$.
    Differentiating both sides of $\gamma^{(P, Q)}\gamma=I_p$ gives
    \begin{equation*}
        \begin{aligned}
            \gamma^{(P, Q)} \dot \gamma+\dot \gamma^{(P, Q)} \gamma&=0, \\
            \gamma^{(P, Q)} \ddot \gamma+2\dot \gamma^{(P, Q)} \dot \gamma+\ddot \gamma^{(P, Q)} \gamma&=0.
        \end{aligned}
    \end{equation*}
    Combining these relations, we have $\ell = -\gamma^{(P, Q)} \ddot\gamma = -\ddot \gamma^{(P, Q)} \gamma= \dot \gamma^{(P, Q)} \dot \gamma$.
    This gives the quasi-geodesic equation
    \begin{equation}\label{eq:geodesic-ODE}
        \ddot \gamma +\gamma \dot\gamma^{(P, Q)}\dot \gamma=0,
    \end{equation}
    which is a second-order, time-invariant, matrix-valued ODE on $\mathrm{Gq}(P, Q)$.
    The following proposition solves it in closed form.

    \begin{proposition}\label{prop:geodesics}
        Let $\gamma$ solve the quasi-geodesic equation \eqref{eq:geodesic-ODE} with the initial conditions
        \begin{equation}\label{eq:geodesic-initial-condition}
            \gamma(0)=X\in\mathrm{Gq}(P, Q) \quad\text{and}\quad \dot \gamma(0)=Z\in \mathrm{T}_X\mathrm{Gq}(P, Q).
        \end{equation}
        Then
        \begin{subequations}
            \begin{align}
                \gamma(t)
                &=\begin{bmatrix}X&Z\end{bmatrix}
                    \exp \left(t\begin{bmatrix}
                    \tilde{A}_1&-\tilde{A}_2 \\I_p&\tilde{A}_1
                    \end{bmatrix}\right)\begin{bmatrix}
                    I_p\\0_p
                    \end{bmatrix}\exp(-t\tilde{A}_1)
                    \label{eq:geodesics-I}\\
                &=\exp(t\tilde{A}_3)X\exp(-t\tilde{A}_1)
                    \label{eq:geodesics-II}\\
                &=\begin{bmatrix}X&I_n\end{bmatrix}
                    \exp \left(t\begin{bmatrix}2 \tilde{A}_1&-\tilde{A}_4^{(P, Q)}\\\tilde{A}_4&0_n\end{bmatrix} \right)
                    \begin{bmatrix}I_p\\0_{n,p}\end{bmatrix}
                    \exp(-t \tilde{A}_1)
                    \label{eq:geodesics-III},\end{align}
        \end{subequations}
        where $\tilde{A}_1 = X^{(P, Q)}Z$, $\tilde{A}_2=Z^{(P, Q)}Z$, $\tilde{A}_3 = ZX^{(P, Q)}-XZ^{(P, Q)}$, and $\tilde{A}_4=(I_n-XX^{(P, Q)})Z$.
    \end{proposition}

    \begin{proof}
        Let $A_0=\gamma^{(P, Q)}\gamma, A_1=\gamma^{(P, Q)} \dot \gamma$, and $A_2=\dot \gamma^{(P, Q)} \dot \gamma$.
        Differentiating these integrals of motion and using \eqref{eq:geodesic-ODE} gives
        \begin{equation*}
            \dot A_0=A_1+A_1^{(P)},\quad
            \dot A_1=A_2-A_0A_2,\quad
            \dot A_2=A_1A_2-A_2A_1.
        \end{equation*}
        By \eqref{eq:TS-Gq-a} and Proposition~\ref{prop:adjoint}(e), $A_1\in\mathcal{S}_{\mathrm{skew}}(P)$ and $A_2\in\mathcal{S}_{\mathrm{sym}}(P)$.
        With $(\mathrm{d}/\mathrm{d}t)\exp(At)=A\exp(At)=\exp(At)A$ and the initial conditions \eqref{eq:geodesic-initial-condition}, these equations integrate to
        \begin{equation*}
            A_0\equiv I,\quad
            A_1\equiv \tilde{A}_1, \quad
            A_2=\exp(t\tilde{A}_1)\tilde{A}_2\exp(-t\tilde{A}_1),
        \end{equation*}
        where $\tilde{A}_1:=X^{(P, Q)}Z$ and $\tilde{A}_2:=Z^{(P, Q)}Z$ are constant matrices.
        Substituting this $A_2$ into \eqref{eq:geodesic-ODE} gives $\ddot \gamma + \gamma\exp(t\tilde{A}_1)\tilde{A}_2\exp(-t\tilde{A}_1)=0$.
        Using $\exp(-t\tilde{A}_1)\exp(t\tilde{A}_1)=I$, differentiation of $\gamma\exp(t\tilde{A}_1)$ and $\dot\gamma\exp(t\tilde{A}_1)$ yields
        \begin{equation*}
        \begin{aligned}
            \frac{\mathrm{d}}{\mathrm{d}t}\gamma\exp(t\tilde{A}_1)
            &=\dot \gamma\exp(t\tilde{A}_1)+\gamma\exp(t\tilde{A}_1)\tilde{A}_1,\\
            \frac{\mathrm{d}}{\mathrm{d}t}\dot \gamma\exp(t\tilde{A}_1)
            &=\ddot \gamma\exp(t\tilde{A}_1)+\dot\gamma\exp(t\tilde{A}_1)\tilde{A}_1
            =-\gamma\exp(t\tilde{A}_1)\tilde{A}_2 + \dot\gamma\exp(t\tilde{A}_1) \tilde{A}_1.
        \end{aligned}
        \end{equation*}
        Define the $n$-by-$2p$ block row
        \begin{equation*}
            B(t)=\begin{bmatrix}\gamma(t)\exp(t\tilde A_1)&\dot\gamma(t)\exp(t\tilde A_1)\end{bmatrix}.
        \end{equation*}
        Together these derivatives give the linear equation
        \begin{equation*}
            \dot B(t)=B(t)\begin{bmatrix}\tilde A_1&-\tilde A_2\\I_p&\tilde A_1\end{bmatrix},
            \qquad B(0)=\begin{bmatrix}X&Z\end{bmatrix},
        \end{equation*}
        whose solution gives \eqref{eq:geodesics-I} by taking its first $p$ columns and multiplying on the right by $\exp(-t\tilde A_1)$.

        We next recast the nonlinear equation \eqref{eq:geodesic-ODE} in linear form.
        With $A_3 := \dot \gamma\gamma^{(P, Q)}-\gamma\dot \gamma^{(P, Q)}$, a direct computation gives $\dot \gamma=A_3\gamma-\gamma A_1$.
        Since $\ddot \gamma = -\gamma \ell$ with $\ell\in\mathcal{S}_{\mathrm{sym}}(P)$,
        \begin{equation*}
            \dot A_3
            = \ddot\gamma \gamma^{(P, Q)}-\gamma \ddot\gamma^{(P, Q)}
            = -\gamma \ell\gamma^{(P, Q)} + \gamma \ell^{(P)}\gamma^{(P, Q)}=0,
        \end{equation*}
        which establishes $A_3 \equiv \tilde{A}_3 :=ZX^{(P, Q)}-XZ^{(P, Q)}$.
        Differentiating once more yields $\ddot \gamma = \tilde{A}_3\dot\gamma-\dot\gamma \tilde{A}_1$, and the linear equation $\dot \gamma=\tilde{A}_3\gamma-\gamma \tilde{A}_1$ with constant coefficients yields the shorter expression \eqref{eq:geodesics-II}.
        The nonlinearity is absorbed into $\tilde{A}_3$.

        For \eqref{eq:geodesics-III}, we show by induction that the terms of the power series of $\exp(t\tilde{A}_3)X$ satisfy
        \begin{equation}\label{eq:geodesics-induction}
            \tilde{A}_3^kX = \begin{bmatrix}X&I_n\end{bmatrix} \begin{bmatrix}2 \tilde{A}_1&-\tilde{A}_4^{(P, Q)} \\ \tilde{A}_4 & 0_n \end{bmatrix}^k \begin{bmatrix}I_p\\0_{n,p}\end{bmatrix},
        \end{equation}
        where $\tilde{A}_4 = (I_n-XX^{(P, Q)})Z$.
        The case $k=1$ is a direct computation.
        Assume that \eqref{eq:geodesics-induction} holds for $k-1$.
        The identity $X^{(P, Q)} (I_n-XX^{(P, Q)})=0$ gives
        \begin{equation*}
            \begin{aligned}
                (\tilde{A}_3+X\tilde{A}_4^{(P, Q)})\tilde{A}_4
                &=(ZX^{(P, Q)}-XZ^{(P, Q)}+XZ^{(P, Q)}(I_n-XX^{(P, Q)}))(I_n-XX^{(P, Q)})Z\\
                &=(Z-XZ^{(P, Q)}X)X^{(P, Q)}(I_n-XX^{(P, Q)})Z=0.
            \end{aligned}
        \end{equation*}
        With this identity and the induction hypothesis,
        \begin{equation*}
            \begin{aligned}
                \tilde{A}_3^kX
                &=\begin{bmatrix} \tilde{A}_3X & \tilde{A}_3 \end{bmatrix} \begin{bmatrix}2 \tilde{A}_1&-\tilde{A}_4^{(P, Q)}\\\tilde{A}_4&0_n\end{bmatrix}^{k-1} \begin{bmatrix}I_p\\0_{n,p}\end{bmatrix}\\
                &=\begin{bmatrix} \tilde{A}_3X & \tilde{A}_3 +X\tilde{A}_4^{(P, Q)}-X\tilde{A}_4^{(P, Q)}\end{bmatrix} \begin{bmatrix}2 \tilde{A}_1&-\tilde{A}_4^{(P, Q)}\\\tilde{A}_4&0_n\end{bmatrix}^{k-1} \begin{bmatrix}I_p\\0_{n,p}\end{bmatrix}\\
                &\stackrel{(a)}{=}\begin{bmatrix}
                2X\tilde{A}_1+\tilde{A}_4
                &-X\tilde{A}_4^{(P, Q)}
                \end{bmatrix}\begin{bmatrix}2 \tilde{A}_1&-\tilde{A}_4^{(P, Q)}\\\tilde{A}_4&0_n\end{bmatrix}^{k-1} \begin{bmatrix}I_p\\0_{n,p}\end{bmatrix}\\
                &=\begin{bmatrix}X&I_n\end{bmatrix} \begin{bmatrix}2 \tilde{A}_1&-\tilde{A}_4^{(P, Q)}\\\tilde{A}_4&0_n\end{bmatrix}^k \begin{bmatrix}I_p\\0_{n,p}\end{bmatrix},
            \end{aligned}
        \end{equation*}
        where $(a)$ uses $\tilde{A}_3X = 2X\tilde{A}_1+\tilde{A}_4$ and \eqref{eq:TS-Gq-a}.
        This completes the induction step in \eqref{eq:geodesics-induction}.
        Summing the series gives
        \begin{equation*}
            \begin{aligned}
                \exp(t\tilde{A}_3)X
                &=\sum_{k=0}^\infty \frac{t^k}{k!}(\tilde{A}_3)^k X\\
                &=\sum_{k=0}^\infty \frac{t^k}{k!}\begin{bmatrix}X&I_n\end{bmatrix} \begin{bmatrix}2 \tilde{A}_1&-\tilde{A}_4^{(P, Q)}\\\tilde{A}_4&0_n\end{bmatrix}^k \begin{bmatrix}I_p\\0_{n,p}\end{bmatrix} \\
                &=\begin{bmatrix}X&I_n\end{bmatrix}\exp \left(t\begin{bmatrix}2 \tilde{A}_1&-\tilde{A}_4^{(P, Q)}\\\tilde{A}_4&0_n\end{bmatrix} \right) \begin{bmatrix}I_p\\0_{n,p}\end{bmatrix},
            \end{aligned}
        \end{equation*}
        which yields \eqref{eq:geodesics-III}.
    \end{proof}

    Comparing \eqref{eq:cov-accel} with \eqref{eq:geodesic-ODE} settles when a quasi-geodesic is a geodesic of $g^\beta$.
    Its ambient acceleration $\ddot\gamma=-\gamma\ell$ with $\ell\in\mathcal{S}_{\mathrm{sym}}(P)$ lies in the normal space \eqref{eq:normal-space-canonical}, giving $\mathcal{P}_{\mathrm{T}_\gamma}(\ddot\gamma)=0$.
    The covariant acceleration of \eqref{eq:geodesics-I} reduces to $\mathfrak{G}_\gamma^\beta(\dot\gamma,\dot\gamma)$.
    A quasi-geodesic is therefore a geodesic of $g^\beta$ exactly when the Christoffel term $\mathfrak{G}_\gamma^\beta(\dot\gamma,\dot\gamma)$ of \eqref{eq:christoffel} vanishes along it.
    On $\mathrm{St}(p,n)$ with $\beta=1$, the formulas in the proof of Proposition~\ref{prop:hess} give $\mathcal K(Z,Z)=X(X^\top Z)^2$ for tangent $Z$.
    The symmetry of $(X^\top Z)^2$ makes this term normal, giving $\mathfrak G_X^1(Z,Z)=0$.
    Thus \eqref{eq:geodesics-I} is the Stiefel geodesic of Edelman et al.~\cite[Eqs.~(2.45)--(2.47)]{edelman1998geometry}.
    For general $(P,Q)$ the Christoffel operator does not vanish identically, and \eqref{eq:geodesics-I} describes a curve on $\mathrm{Gq}(P, Q)$ that need not be a geodesic.

    Among these three expressions, \eqref{eq:geodesics-I} is the most efficient for $p\ll n$, requiring only the exponential of a $2p\times 2p$ matrix.
    Equation~\eqref{eq:geodesics-II} is the most concise but requires the $n\times n$ matrix exponential, while \eqref{eq:geodesics-III} uses an $(n+p)\times(n+p)$ exponential.
    The quasi-geodesic-based retraction is therefore built on \eqref{eq:geodesics-I}.

    \begin{theorem}\label{thm:quasi-geodesic-retraction}
        Let $X\in\mathrm{Gq}(P, Q)$ and $Z\in \mathrm{T}_X\mathrm{Gq}(P, Q)$, and let $\tilde{A}_1 = X^{(P, Q)}Z$ and $\tilde{A}_2=Z^{(P, Q)}Z$.
        The mapping
        \begin{equation}\label{eq:quasi-geodesic-retraction}
            \mathcal R_X^{\mathsf{qgeo}}(Z)
            =\begin{bmatrix}X&Z\end{bmatrix}
                \exp\left(\begin{bmatrix}
                \tilde{A}_1&-\tilde{A}_2 \\I_p&\tilde{A}_1
                \end{bmatrix}\right)
                \begin{bmatrix} I_p\\0_p \end{bmatrix}
                \exp(-\tilde{A}_1)
        \end{equation}
        defines the quasi-geodesic-based retraction on $\mathrm{Gq}(P, Q)$, with domain $\mathcal{D}_X^{\mathsf{qgeo}}=\mathrm{T}_X\mathrm{Gq}(P, Q)$.
    \end{theorem}

    \begin{proof}
        Smoothness of $\mathcal{R}_X^{\mathsf{qgeo}}$ follows from the smoothness of the matrix exponential.
        By Proposition~\ref{prop:geodesics} with $t=0$, $\mathcal{R}_X^{\mathsf{qgeo}}(0_X) = X$ and $\tfrac{\mathrm{d}}{\mathrm{d}t}\mathcal{R}_X^{\mathsf{qgeo}}(tZ)\big|_{t=0} = Z$, verifying the two conditions in~\eqref{eq:retraction}.
        The quasi-geodesic \eqref{eq:geodesics-I} is defined for all $Z\in \mathrm{T}_X\mathrm{Gq}(P,Q)$, which establishes global well-definedness.
        The detailed verification parallels that of Gao et al.~\cite[Lem.~5.1]{gao2021riemannian}.
    \end{proof}

    The retraction corresponding to a geodesic is the exponential map on the manifold, and its inverse is the logarithmic map.
    The inverse of the quasi-geodesic-based retraction has no known closed form.
    Already on $\mathrm{St}(p,n)$ with $\beta=1$, where \eqref{eq:quasi-geodesic-retraction} is the exponential map, the inverse is computed by iterative schemes such as the shooting method~\cite{bryner2017endpoint} and the matrix-algebraic approach~\cite{zimmermann2022computing}.
    We therefore pair \eqref{eq:quasi-geodesic-retraction} with the projection transport~\eqref{eq:VT-orth-proj}.

\subsection{Polar-Based Retraction}
\label{sec:pol-Retraction}

    The polar-based retraction rests on a generalized polar decomposition, a rectangular specialization of the structured polar decompositions of Mackey et al.~\cite[Thm.~6.2]{mackey2005structured}, Higham et al.~\cite[Thm.~3.9]{higham2010canonical}, and Benner et al.~\cite[Thm.~2.2]{benner2022stable}.
    In the present full-column-rank setting, the relevant existence condition reduces to the spectral condition on $A^{(P,Q)}A$.

    \begin{proposition}
        \label{prop:generalized-polar}
        Let $A\in\mathbb{R}^{n\times p}$ have full column rank.
        There exists a unique generalized polar decomposition $A = US$, where $U\in\mathrm{Gq}(P, Q)$ and $S\in\mathcal{S}_{\mathrm{sym}}^{+}(P)$, if and only if $A^{(P, Q)} A$ has no eigenvalues on $\mathbb{R}_0^-$.
    \end{proposition}

    \begin{proof}
        If $A=US$ with $U^{(P, Q)} U=I_p$ and $S^{(P)}=S$, then
        \begin{equation*}
            A^{(P, Q)} A = S^{(P)} U^{(P, Q)} US = S^{(P)} S = S^2.
        \end{equation*}
        The condition $\mathrm{sign}(S)=I_p$ shows that $S=(A^{(P, Q)} A)^{1/2}$ is the principal square root.
        Lemma~\ref{lem:principal} excludes eigenvalues of $A^{(P, Q)} A$ on $\mathbb{R}_0^-$.

        Conversely, if $A^{(P, Q)} A$ has no eigenvalues on $\mathbb{R}_0^-$, Lemma~\ref{lem:principal} gives its principal square root $S=(A^{(P, Q)} A)^{1/2}$ with $\mathrm{sign}(S)=I_p$.
        The relation $A^{(P, Q)} A\in\mathcal{S}_{\mathrm{sym}}(P)$ gives $S^{(P)}S^{(P)} = (S\cdot S)^{(P)} = A^{(P, Q)}A$, identifying $S^{(P)}$ as a square root of $A^{(P, Q)} A$.
        Similarity to $S$ places the eigenvalues of $S^{(P)}$ in the open right half-plane, making $S^{(P)}$ the principal square root of $A^{(P, Q)}A$.
        By the uniqueness of the principal square root, $S^{(P)}=S$, and $\mathrm{sign}(S)=I_p$ places $S$ in $\mathcal{S}_{\mathrm{sym}}^{+}(P)$.
        Finally, $U=AS^{-1}$ satisfies $U^{(P, Q)} U=S^{-{(P)}}A^{(P, Q)} AS^{-1}=I_p$, hence $U\in\mathrm{Gq}(P, Q)$.
    \end{proof}

    Proposition~\ref{prop:generalized-polar} defines the generalized polar factor
    \begin{equation*}
        \mathrm{pol}(A) := A (A^{(P,Q)}A)^{-1/2}.
    \end{equation*}
    The polar-based retraction takes the generalized polar factor of $X+Z$.

    \begin{theorem}\label{thm:polar-retraction}
        Let $X\in\mathrm{Gq}(P, Q)$, and let $Z\in \mathrm{T}_X\mathrm{Gq}(P, Q)$ satisfy $\|Z^{(P,Q)}Z\|_2 < 1$, which holds, for example, when $\|Z\|_2 < \bigl(\|P^{-1}\|_2\|Q\|_2\bigr)^{-1/2}$.
        Then the mapping
        \begin{equation}\label{eq:polar-retraction}
            \mathcal{R}_X^{\mathsf{pol}}(Z) = \mathrm{pol}(X+Z) = (X+Z)(I_{p}+Z^{(P, Q)} Z)^{-1/2}
        \end{equation}
        defines the polar-based retraction on $\mathrm{Gq}(P, Q)$.

        Write $\mathcal{D}_X^{\mathsf{pol}}=\{Z\in \mathrm{T}_X\mathrm{Gq}(P, Q):\|Z^{(P,Q)}Z\|_2<1\}$ for the domain of \eqref{eq:polar-retraction} cut out by that ball.
        The map $\mathcal{R}_X^{\mathsf{pol}}$ restricts to a diffeomorphism from an open neighborhood $\mathcal{N}_X\subseteq\mathcal{D}_X^{\mathsf{pol}}$ of $0_X$ onto its image $\mathcal{V}_X:=\mathcal{R}_X^{\mathsf{pol}}(\mathcal{N}_X)$.
        For every $Y\in\mathcal{V}_X$, the inverse polar-based retraction admits the closed-form expression
        \begin{equation}\label{eq:inverse-polar}
            (\mathcal R_X^{\mathsf{pol}})^{-1}(Y) = YS-X,
        \end{equation}
        where $S\in\mathcal{S}_{\mathrm{sym}}^{+}(P)$ of \eqref{eq:P-symm-plus} solves the generalized Sylvester equation
        \begin{equation}\label{eq:polar-Sylvester}
            SY^{(P,Q)}X+X^{(P,Q)}YS=2I_p.
        \end{equation}
        On $\mathcal{V}_X$ the spectral separation $\sigma(X^{(P,Q)}Y)\cap\sigma(-Y^{(P,Q)}X)=\emptyset$ holds, which makes $S$ the unique $P$-symmetric solution of \eqref{eq:polar-Sylvester}.
        Uniqueness holds in all of $\mathcal{S}_{\mathrm{sym}}(P)$, whereas $\mathcal{S}_{\mathrm{sym}}^{+}(P)$ selects the branch on which that solution is the $P$-symmetric factor of the polar decomposition of $X+Z$.

        More generally, let $Y\in \mathrm{Gq}(P, Q)$ satisfy the spectral separation condition and let $S\in\mathcal{S}_{\mathrm{sym}}(P)$ be the unique solution of \eqref{eq:polar-Sylvester}.
        If $S\in\mathcal{S}_{\mathrm{sym}}^{+}(P)$ and $Z:=YS-X$ lies in $\mathcal{D}_X^{\mathsf{pol}}$, then $Z\in \mathrm{T}_X\mathrm{Gq}(P, Q)$ and $\mathcal{R}_X^{\mathsf{pol}}(Z)=Y$.
    \end{theorem}

    \begin{proof}
        The inverse of the generalized polar decomposition is the map
        \begin{equation*}
            \psi: \mathrm{Gq}(P, Q)\times \mathcal{S}_{\mathrm{sym}}^{+}(P) \to \mathcal{U}:(U,S) \mapsto US,
        \end{equation*}
        where $\mathcal{U}:=\{A\in\mathbb{R}^{n\times p} : \mathrm{rank}(A)=p,\; \sigma(A^{(P,Q)}A)\cap\mathbb{R}_0^-=\emptyset\}$ is the set of matrices that have a generalized polar decomposition.
        It follows from Proposition~\ref{prop:generalized-polar} that $\mathcal{U}$ is an open subset of $\mathbb{R}^{n\times p}$.
        The identity matrix $I_p\in\mathcal{S}_{\mathrm{sym}}^{+}(P)$ satisfies $\psi(U,I_p)=U$ for all $U\in\mathrm{Gq}(P,Q)$.
        By Proposition~\ref{prop:submanifold}, $\dim(\mathrm{Gq}(P, Q))=np-p(p+\varrho)/2$, and Proposition~\ref{prop:adjoint}(b) and the openness of $\mathcal{S}_{\mathrm{sym}}^{+}(P)$ in $\mathcal{S}_{\mathrm{sym}}(P)$ give $\dim(\mathcal{S}_{\mathrm{sym}}^{+}(P))=p(p+\varrho)/2$, giving
        \begin{equation*}
            \dim(\mathrm{Gq}(P, Q))+\dim(\mathcal{S}_{\mathrm{sym}}^{+}(P))=np=\dim(\mathbb{R}_*^{n\times p}),
        \end{equation*}
        where $\mathbb{R}_*^{n\times p}$ is the open set of full-column-rank matrices in $\mathbb{R}^{n\times p}$.
        As the restriction of matrix multiplication to a submanifold, $\psi$ is smooth.
        We next show that the map $\psi^{-1}(A) = (A(A^{(P, Q)} A)^{-1/2},(A^{(P, Q)} A)^{1/2})$ is a diffeomorphism.
        By Proposition~\ref{prop:generalized-polar}, every $A\in\mathcal{U}$ has a unique factorization $A=US$ with $U\in\mathrm{Gq}(P,Q)$ and $S=(A^{(P,Q)}A)^{1/2}\in\mathcal{S}_{\mathrm{sym}}^{+}(P)$.
        This unique factorization makes $\psi$ a bijection onto $\mathcal{U}$.
        The principal square root of $A^{(P,Q)} A$ is real-analytic by Lemma~\ref{lem:principal}, which makes $\psi^{-1}$ a $C^\infty$ map on $\mathcal{U}$.
        For $X+Z\in\mathcal{U}$, write $X+Z=US$ for its unique generalized polar decomposition and define $\mathcal{R}_X^{\mathsf{pol}}(Z):=U$.
        This defines a smooth local map on an open neighborhood of the zero section.
        To verify the two conditions in~\eqref{eq:retraction}, use $X^{(P,Q)}Z+Z^{(P,Q)}X=0$ for tangent $Z$ from~\eqref{eq:TS-Gq-a} to obtain
        \begin{equation*}
            \begin{aligned}
                (X+tZ)^{(P,Q)}(X+tZ)&=I_p+t^2Z^{(P,Q)}Z,\\
                \mathcal{R}_X^{\mathsf{pol}}(tZ)&=(X+tZ)(I_p+t^2Z^{(P,Q)}Z)^{-1/2}=X+tZ+\mathcal{O}(t^2),
            \end{aligned}
        \end{equation*}
        which gives $\mathcal{R}_X^{\mathsf{pol}}(0_X)=X$ and $\mathrm{D}\mathcal{R}_X^{\mathsf{pol}}(0_X)[Z]=Z$.
        The same identity gives \eqref{eq:polar-retraction},
        \begin{equation*}
            \begin{aligned}
                \mathcal{R}_X^{\mathsf{pol}}(Z)
                &= (X+Z)((X+Z)^{(P, Q)} (X+Z))^{-1/2}\\
                &= (X+Z)(X^{(P, Q)} X+Z^{(P, Q)} Z+X^{(P, Q)} Z+Z^{(P, Q)} X)^{-1/2} \\
                &= (X+Z)(I_{p}+Z^{(P, Q)} Z)^{-1/2}.
            \end{aligned}
        \end{equation*}

        The principal square root of $I_p + Z^{(P,Q)}Z$ exists when this matrix has no eigenvalues on $\mathbb{R}_0^-$.
        Under the condition $\|Z^{(P,Q)}Z\|_2<1$, every eigenvalue $\lambda$ of $Z^{(P,Q)}Z$ satisfies $|\lambda|<1$, so that $1+\lambda\notin\mathbb{R}_0^-$.
        The bound $\|Z^{(P,Q)}Z\|_2\le \|P^{-1}\|_2\|Q\|_2\|Z\|_2^2$ gives the stated norm-based sufficient condition.

        For the inverse, the retraction property $\mathrm{D}\mathcal{R}_X^{\mathsf{pol}}(0_X)=\mathrm{id}$ and the inverse function theorem yield an open neighborhood $\mathcal{N}_X\subseteq\mathcal{D}_X^{\mathsf{pol}}$ of $0_X$ on which $\mathcal{R}_X^{\mathsf{pol}}$ is a diffeomorphism onto $\mathcal{V}_X=\mathcal{R}_X^{\mathsf{pol}}(\mathcal{N}_X)$.
        Fix $Y\in\mathcal{V}_X$ with unique preimage $Z\in\mathcal{N}_X$, and set $\tilde{X}=X+Z$.
        The $P$-symmetric factor $S=(\tilde{X}^{(P,Q)} \tilde{X})^{1/2}=(I_p+Z^{(P,Q)}Z)^{1/2}\in\mathcal{S}_{\mathrm{sym}}^{+}(P)$ of its polar decomposition satisfies $Y=\tilde{X}S^{-1}$, so that $\tilde{X}=YS$ and $Z=YS-X$.
        The tangency condition \eqref{eq:TS-Gq-a}, together with $X\in \mathrm{Gq}(P, Q)$, gives
        \begin{equation*}
            0 = (YS-X)^{(P,Q)}X+X^{(P,Q)}(YS-X)
            = S Y^{(P,Q)}X + X^{(P,Q)}YS - 2I_p,
        \end{equation*}
        which is precisely the equation for $S$ in \eqref{eq:polar-Sylvester}.
        At $Y=X$ the separation condition reads $\sigma(I_p)\cap\sigma(-I_p)=\emptyset$, and by continuity it holds throughout $\mathcal{V}_X$ after shrinking $\mathcal{N}_X$ if necessary.
        Proposition~\ref{prop:adjoint}(f) applied with $A=X^{(P,Q)}Y$ identifies $S$ as the unique $P$-symmetric solution, which the Bartels--Stewart algorithm~\cite{bartels1972solution} computes.

        For the last assertion, let $Y$ and $S$ be as stated, so that $S=S^{(P)}$ and $Z=YS-X$.
        Reversing the derivation of \eqref{eq:polar-Sylvester}, $X^{(P,Q)}Z+Z^{(P,Q)}X=X^{(P,Q)}YS+SY^{(P,Q)}X-2I_p=0$, which implies $Z\in \mathrm{T}_X\mathrm{Gq}(P, Q)$ by~\eqref{eq:TS-Gq-a}.
        The assumption $\|Z^{(P,Q)}Z\|_2<1$ places $Z$ in the forward domain, where $\mathcal{R}_X^{\mathsf{pol}}(Z)$ is defined by \eqref{eq:polar-retraction}.
        Using $Y^{(P,Q)}Y=I_p$, we obtain $I_p+Z^{(P,Q)}Z=(X+Z)^{(P,Q)}(X+Z)=(YS)^{(P,Q)}(YS)=SY^{(P,Q)}YS=S^2$, and $\mathrm{sign}(S)=I_p$ makes $S$ the principal square root of $S^2$.
        Substitution into the polar retraction gives $\mathcal{R}_X^{\mathsf{pol}}(Z)=(X+Z)(I_p+Z^{(P,Q)}Z)^{-1/2}=YSS^{-1}=Y$.
    \end{proof}

    \begin{remark}\label{rem:polar-branch}
        The branch condition $S\in\mathcal{S}_{\mathrm{sym}}^{+}(P)$ in Theorem~\ref{thm:polar-retraction} is needed in addition to the spectral separation.
        On the unit sphere $\mathrm{Gq}(1, I_2)$, take $X=(1, 0)^\top$ and its antipode $Y=-X$.
        The separation condition holds, yet \eqref{eq:polar-Sylvester} reads $-2S=2$ and returns $S=-1$, whence $YS-X=0$, whereas $\mathcal{R}_X^{\mathsf{pol}}(0)=X\neq Y$.
        In this example, $\mathrm{sign}(S)=-1$, and the closed form recovers the tangent vector only for $S\in\mathcal{S}_{\mathrm{sym}}^{+}(P)$.
    \end{remark}

    Theorem~\ref{thm:polar-retraction} generalizes the polar decomposition on the orthogonal group and the Stiefel manifold~\cite[Eqs.~(4.6)--(4.7)]{absil2008optimization} and on the symplectic group~\cite[Thm.~2.5]{teretenkov2022symplectic}.
    For $P=I_p$ and SPD $Q=B$, \eqref{eq:polar-retraction}, \eqref{eq:inverse-polar}, and \eqref{eq:polar-Sylvester} coincide with Eqs.~(3.3)--(3.5) of Shustin and Avron~\cite{shustin2023riemannian}.
    On $\mathrm{Sp}(p,n)$, Zimmermann~\cite{zimmermann2026polar} builds a retraction with a closed-form inverse from the same symplectic polar decomposition~\cite{teretenkov2022symplectic}, under the same $\mathbb{R}_0^-$ eigenvalue exclusion that makes the principal square root available.
    His map post-multiplies a matrix-exponential (or Cayley) update by a symplectifier and inverts it by a matrix logarithm.
    In contrast, \eqref{eq:polar-retraction} takes the polar factor of $X+Z$ and inverts it through \eqref{eq:polar-Sylvester} for general $(P, Q)$.
    The differentiated retraction follows the derivations of Huang et al.~\cite[Sec.~4]{huang2015broyden} and Zhu~\cite[Sec.~3.5]{zhu2017riemannian}.

    \begin{proposition}\label{prop:VT-polar}
        The differentiated retraction of \eqref{eq:polar-retraction} is
        \begin{equation}\label{eq:VT-polar}
            \mathcal{T}_{Z}^{\mathsf{pol}}(W)
            = Y \varOmega + (I_n-YY^{(P, Q)})W S^{-1},
        \end{equation}
        where $Y = \mathcal R_X^{\mathsf{pol}}(Z)$, $S = Y^{(P, Q)}(X+Z)\in\mathcal{S}_{\mathrm{sym}}^{+}(P)$, and $\varOmega\in\mathcal{S}_{\mathrm{skew}}(P)$ solves the generalized Sylvester equation
        \begin{equation}\label{eq:Sylvester-Omega}
            \varOmega S+S\varOmega = Y^{(P, Q)} W -W^{(P, Q)} Y.
        \end{equation}
        For $Z=tW$, the transport \eqref{eq:VT-polar} has the closed form
        \begin{equation}\label{eq:VT-polar-closed-form}
            \mathcal T_{tW}^{\mathsf{pol}}(W) = (I_n- t YS^{-1} W^{(P, Q)}) WS^{-1}.
        \end{equation}
    \end{proposition}

    \begin{proof}
        Part~1 follows the steps of the Stiefel case~\cite[Lem.~10.2.1]{huang2013optimization}, and Part~2 treats collinear $Z$ and $W$.

        \emph{Part 1.} By \eqref{eq:differential-retraction},
        \begin{equation*}
            \mathcal{T}_Z^{\mathsf{pol}}(W)
            = \mathrm{D}\mathcal R_X^{\mathsf{pol}}(Z)[W]
            = \left.\frac{\mathrm{d}}{\mathrm{d}t}\mathcal R_X^{\mathsf{pol}}(Z+tW)\right|_{t=0},
        \end{equation*}
        which is the derivative $\mathrm{D} \mathrm{pol}(X+Z)[W]$ along the curve $\upsilon(t) = X+Z+tW$, with $\upsilon(0)=X+Z$ and $\dot \upsilon(0)=W$.
        Let $\upsilon(t)=U(t)S(t)$ be its generalized polar decomposition, so that
        \begin{equation}\label{eq:VT-dotupsilon}
            \dot \upsilon = \dot US+U\dot S,
        \end{equation}
        and $\mathcal{T}_Z^{\mathsf{pol}}(W) = \dot U(0)$.
        The identity $(I_n-UU^{(P, Q)}) + UU^{(P, Q)} = I_n$ splits $\dot U$ as
        \begin{equation}\label{eq:VT-dotU}
            \dot U=(I_n-UU^{(P, Q)})\dot U + UU^{(P, Q)}\dot U.
        \end{equation}
        Multiplying \eqref{eq:VT-dotupsilon} on the left by $(I_n-UU^{(P, Q)})$ and on the right by $S^{-1}$ and using $U^{(P, Q)} U = I_p$ gives
        \begin{equation*}
            (I_n-UU^{(P, Q)})\dot \upsilon S^{-1}
            = (I_n-UU^{(P, Q)})\dot U+(I_n-UU^{(P, Q)}) U\dot SS^{-1}
            = (I_n-UU^{(P, Q)})\dot U,
        \end{equation*}
        which is the first term of \eqref{eq:VT-dotU}.
        For the second term, apply \eqref{eq:TS-Gq-a} along the curve $U(t)\in\mathrm{Gq}(P, Q)$ to obtain $U^{(P, Q)}\dot U + \dot U^{(P, Q)} U=0$.
        Thus $\varOmega := U^{(P, Q)}\dot U$ satisfies $\varOmega^{(P)} = \dot U^{(P, Q)}U = -U^{(P, Q)}\dot U = -\varOmega$, proving that $\varOmega\in\mathcal{S}_{\mathrm{skew}}(P)$.
        Left-multiplying \eqref{eq:VT-dotupsilon} by $U^{(P, Q)}$ yields $U^{(P, Q)}\dot \upsilon = \varOmega S+ \dot S$.
        Differentiating the identity $S(t)=S(t)^{(P)}$ and applying Proposition~\ref{prop:adjoint}(d) yields $\dot S = \dot S^{(P)}$.
        Subtracting the $P$-adjoint of this identity shows that $\varOmega$ solves the generalized Sylvester equation
        \begin{equation}\label{eq:VT-Lyapunov}
             \varOmega S+S\varOmega = U^{(P, Q)}\dot \upsilon - \dot \upsilon^{(P, Q)} U,
        \end{equation}
        whose right-hand side is $P$-skew-symmetric.
        The principal square root $S$ has spectrum in the open right half-plane, hence $\sigma(S)\cap\sigma(-S^{(P)})=\emptyset$.
        By Proposition~\ref{prop:adjoint}(f), the Sylvester equation has a unique solution, which is $P$-skew-symmetric and given in vectorized form by
        \begin{equation*}
            \mathrm{vec}(\varOmega)
            = (S^\top\otimes I_p+I_p\otimes S)^{-1} \mathrm{vec}(U^{(P, Q)}\dot \upsilon-\dot \upsilon^{(P, Q)}U).
        \end{equation*}

        With these two terms, \eqref{eq:VT-dotU} becomes
        \begin{equation}\label{eq:VT-dotU0}
            \dot U = (I_n-UU^{(P, Q)})\dot \upsilon S^{-1} + U\varOmega.
        \end{equation}
        At $t=0$, the decomposition $\upsilon(t) = U(t)S(t)$ of $\upsilon(t)=X+Z+tW$ gives
        \begin{equation*}
            Y:=U(0)=\mathcal R_X^{\mathsf{pol}}(Z),\quad
            S(0)=Y^{(P, Q)} \upsilon(0) = Y^{(P, Q)}(X+Z),\quad
            \dot \upsilon(0) = W.
        \end{equation*}
        Substituting these terms into \eqref{eq:VT-dotU0} and \eqref{eq:VT-Lyapunov} yields~\eqref{eq:VT-polar} and \eqref{eq:Sylvester-Omega}.

        \emph{Part 2.} We write $S=S(0)\in\mathcal{S}_{\mathrm{sym}}^{+}(P)$.
        For $Z = tW$, substituting $Y = (X+Z)S^{-1}$ into the right-hand side of \eqref{eq:VT-Lyapunov} gives
        \begin{equation*}
            \begin{aligned}
                (Y^{(P, Q)} W - W^{(P, Q)} Y)
                &=\frac{1}{t}(Y^{(P, Q)} (X+Z-X) + (X-X-Z)^{(P, Q)} Y)\\
                &=\frac{1}{t}(X^{(P, Q)}(X+Z) S^{-1}-((X+Z)S^{-1})^{(P, Q)} X + S - S^{(P)} )\\
                &\stackrel{(a)}{=}\frac{1}{t}(X^{(P, Q)}ZS^{-1}-S^{-1}Z^{(P, Q)} X)\\
                &\stackrel{(b)}{=}\frac{1}{t}(X^{(P, Q)}ZS^{-1} + S^{-1}X^{(P, Q)} Z)\\
                &\stackrel{(c)}{=}S(S^{-1}X^{(P, Q)}W S^{-1}) +(S^{-1}X^{(P, Q)} W S^{-1})S,
            \end{aligned}
        \end{equation*}
        where $(a)$ uses $X^{(P, Q)} X=I_p$ and $S^{-1}\in\mathcal{S}_{\mathrm{sym}}(P)$ from Proposition~\ref{prop:adjoint}(a).
        Step $(b)$ follows from \eqref{eq:TS-Gq-a}, which gives $Z^{(P, Q)} X\in\mathcal{S}_{\mathrm{skew}}(P)$.
        Step $(c)$ uses $Z=tW$ and puts the result in the form of the left-hand side of~\eqref{eq:VT-Lyapunov}.
        Hence $\varOmega = S^{-1} X^{(P, Q)}W S^{-1}\in\mathcal{S}_{\mathrm{skew}}(P)$ solves \eqref{eq:VT-Lyapunov}.
        Substituting this $\varOmega$ into \eqref{eq:VT-polar} and using $X+tW=YS$ yields \eqref{eq:VT-polar-closed-form}.
    \end{proof}

    Evaluating \eqref{eq:polar-retraction} or \eqref{eq:VT-polar-closed-form} requires a principal square root and linear solves of size $p$.
    In MATLAB, $\texttt{sqrtm}$ computes this square root by a Schur-based method~\cite[Ch.~6]{higham2008functions}.
    These operations cost $\mathcal O(p^3)$, while the adjoint products and matrix multiplications contribute $\mathcal O(n^2p+np^2)$ for dense $P$ and $Q$.
    The total cost is $\mathcal O(n^2p+np^2+p^3)$.

\subsection{Second-Order Corrections}
\label{sec:second-order}

    A retraction is second-order when the covariant acceleration \eqref{eq:cov-accel} of the curve $t\mapsto\mathcal{R}_X(tZ)$ vanishes at $t=0$ for every $X$ and every $Z\in \mathrm{T}_X\mathrm{Gq}(P, Q)$~\cite[Sec.~5.10]{boumal2023introduction}.
    For such a retraction and a smooth $f$, the Riemannian Hessian of $f$ at $X$ is the Hessian of $f\circ\mathcal{R}_X$ at $0_X$.
    The second-order model of $f$ built from $\mathrm{grad}\,f(X)$ and $\mathrm{hess}\,f(X)$ accordingly matches $f\circ\mathcal{R}_X$ up to a third-order remainder.
    For a general admissible pair $(P,Q)$ and a general $\beta$, the covariant accelerations of the three retractions built so far can be nonzero.
    These retractions need not be second-order.
    The closed-form expression \eqref{eq:cov-accel} also provides the tangent vector needed to cancel the covariant acceleration.

    \begin{proposition}\label{prop:second-order}
        Let $X\in\mathrm{Gq}(P, Q)$ and let $\mathcal{R}_X$ be one of \eqref{eq:Cayley-retraction}, \eqref{eq:quasi-geodesic-retraction}, and \eqref{eq:polar-retraction}, with domain $\mathcal{D}_X$, namely $\mathcal{D}_X^{\mathsf{cay}}$ of Theorem~\ref{thm:Cayley-retraction}, $\mathcal{D}_X^{\mathsf{qgeo}}$ of Theorem~\ref{thm:quasi-geodesic-retraction}, and $\mathcal{D}_X^{\mathsf{pol}}$ of Theorem~\ref{thm:polar-retraction}.
        For $Z\in \mathrm{T}_X\mathrm{Gq}(P, Q)$, put $\tilde{A}_1=X^{(P, Q)}Z$, $c(t)=\mathcal{R}_X(tZ)$, and
        \begin{equation}\label{eq:alpha}
            \alpha_X(Z)=\mathcal{P}_{\mathrm{T}_X}(\ddot c(0))+\mathfrak{G}_X^\beta(Z,Z),
            \qquad
            \mathcal{P}_{\mathrm{T}_X}(\ddot c(0))=
            \begin{cases}
                Z\tilde{A}_1-X\tilde{A}_1^2, & \mathcal{R}_X=\mathcal{R}_X^{\mathsf{cay}},\\
                0, & \mathcal{R}_X\in\{\mathcal{R}_X^{\mathsf{qgeo}},\mathcal{R}_X^{\mathsf{pol}}\}.
            \end{cases}
        \end{equation}
        Then $\alpha_X(Z)$, which depends on the choice of $\mathcal{R}_X$ through $c$, belongs to $\mathrm{T}_X\mathrm{Gq}(P, Q)$, and the mapping
        \begin{equation}\label{eq:second-order-retraction}
            \widetilde{\mathcal{R}}_X(Z)=\mathcal{R}_X\bigl(Z-\frac{1}{2}\alpha_X(Z)\bigr)
        \end{equation}
        is a second-order retraction on the open neighborhood $\{Z\in \mathrm{T}_X\mathrm{Gq}(P, Q):Z-\alpha_X(Z)/2\in\mathcal{D}_X\}$ of $0_X$.
        Write $\widetilde{\mathcal{R}}_X^{\mathsf{cay}}$, $\widetilde{\mathcal{R}}_X^{\mathsf{qgeo}}$, and $\widetilde{\mathcal{R}}_X^{\mathsf{pol}}$ for the three maps that \eqref{eq:second-order-retraction} produces.
        The neighborhood of $\widetilde{\mathcal{R}}_X^{\mathsf{qgeo}}$ is all of $\mathrm{T}_X\mathrm{Gq}(P, Q)$, while those of $\widetilde{\mathcal{R}}_X^{\mathsf{cay}}$ and $\widetilde{\mathcal{R}}_X^{\mathsf{pol}}$ are the preimages of $\mathcal{D}_X^{\mathsf{cay}}$ and $\mathcal{D}_X^{\mathsf{pol}}$ under $Z\mapsto Z-\alpha_X(Z)/2$.
    \end{proposition}

    \begin{proof}
        \textit{Step 1: Compute the covariant accelerations.}
        Write $\tilde{A}_2=Z^{(P, Q)}Z$.
        By Proposition~\ref{prop:geodesics}, the curve $t\mapsto\mathcal{R}_X^{\mathsf{qgeo}}(tZ)$ solves \eqref{eq:geodesic-ODE} under the initial conditions \eqref{eq:geodesic-initial-condition}.
        Evaluation at $t=0$ gives $\ddot c(0)=-X\tilde{A}_2$.
        By \eqref{eq:polar-retraction} and \eqref{eq:TS-Gq-a}, $\mathcal{R}_X^{\mathsf{pol}}(tZ)=(X+tZ)(I_p+t^2\tilde{A}_2)^{-1/2}=X+tZ-t^2X\tilde{A}_2/2+\mathcal{O}(t^3)$, again giving $\ddot c(0)=-X\tilde{A}_2$.
        For \eqref{eq:Cayley-retraction}, $\varGamma_{X,tZ}=t\varGamma_{X,Z}$ by \eqref{eq:expression-Gamma} and $\mathrm{cay}(t\varGamma_{X,Z}/2)=I_n+t\varGamma_{X,Z}+t^2\varGamma_{X,Z}^2/2+\mathcal{O}(t^3)$.
        Differentiating twice at $t=0$ gives $\ddot c(0)=\varGamma_{X,Z}^2X=\varGamma_{X,Z}Z=Z\tilde{A}_1-X\tilde{A}_1^2-X\tilde{A}_2$, where the last equality uses $G=Z-X\tilde{A}_1/2$ and $G^{(P, Q)}Z=\tilde{A}_2+\tilde{A}_1^2/2$.
        By Proposition~\ref{prop:adjoint}(e) and \eqref{eq:TS-Gq-a}, $\tilde{A}_2\in\mathcal{S}_{\mathrm{sym}}(P)$ and $\tilde{A}_1\in\mathcal{S}_{\mathrm{skew}}(P)$, hence $(\tilde{A}_1^2)^{(P)}=(\tilde{A}_1^{(P)})^2=\tilde{A}_1^2$ by \eqref{eq:P-adjoint}.
        Therefore $X\tilde{A}_2$ and $X\tilde{A}_1^2$ belong to the normal space \eqref{eq:normal-space-canonical}, and \eqref{eq:projection-tangent-canonical} gives $\mathcal{P}_{\mathrm{T}_X}(Z\tilde{A}_1)=Z\tilde{A}_1-X\tilde{A}_1^2$, which proves \eqref{eq:alpha}.
        Both summands of $\alpha_X(Z)$ are values of $\mathcal{P}_{\mathrm{T}_X}$, and thus $\alpha_X(Z)$ is tangent.

        \textit{Step 2: Verify the retraction conditions.}
        Each summand of $\alpha_X$ is quadratic in $Z$ and depends smoothly on $X$.
        The correction map $\varPsi_X(Z):=Z-\alpha_X(Z)/2$ is smooth and satisfies $\varPsi_X(0_X)=0_X$ and $\mathrm{D}\varPsi_X(0_X)=\mathrm{id}$.
        By continuity of $\varPsi_X$, the stated preimage $\varPsi_X^{-1}(\mathcal{D}_X)$ is open.
        It contains $0_X$, as $\varPsi_X(0_X)=0_X\in\mathcal{D}_X$.
        On it $\widetilde{\mathcal{R}}_X$ is smooth, $\widetilde{\mathcal{R}}_X(0_X)=X$, and $\mathrm{D}\widetilde{\mathcal{R}}_X(0_X)=\mathrm{D}\mathcal{R}_X(0_X)\circ\mathrm{D}\varPsi_X(0_X)=\mathrm{id}$, which are the two conditions in~\eqref{eq:retraction}.

        \textit{Step 3: Verify the second-order condition.}
        The homogeneity relation $\alpha_X(tZ)=t^2\alpha_X(Z)$ gives the corrected curve $\tilde c(t)=\mathcal{R}_X(u(t))$ with $u(t)=tZ-t^2\alpha_X(Z)/2$.
        Its derivatives satisfy $\dot u(0)=Z$ and $\ddot u(0)=-\alpha_X(Z)$.
        Using $\mathrm{D}\mathcal{R}_X(0_X)=\mathrm{id}$,
        \begin{equation*}
            \dot{\tilde c}(0)=Z,\qquad
            \ddot{\tilde c}(0)=\mathrm{D}^2\mathcal{R}_X(0_X)[Z,Z]+\mathrm{D}\mathcal{R}_X(0_X)[\ddot u(0)]=\ddot c(0)-\alpha_X(Z).
        \end{equation*}
        By \eqref{eq:cov-accel} and $\mathcal{P}_{\mathrm{T}_X}(\alpha_X(Z))=\alpha_X(Z)$, the covariant acceleration of $\tilde c$ at $t=0$ equals $\mathcal{P}_{\mathrm{T}_X}(\ddot c(0))-\alpha_X(Z)+\mathfrak{G}_X^\beta(Z,Z)$, which vanishes by \eqref{eq:alpha}.
    \end{proof}

    The Stiefel specialization identifies the value of $\beta$ at which each of the three retractions needs no correction.

    \begin{proposition}\label{prop:stiefel-second-order}
        Let $P=I_p$ and $Q=I_n$, so that $\mathrm{Gq}(P, Q)=\mathrm{St}(p,n)$, and let $Z\in \mathrm{T}_X\mathrm{St}(p,n)$ with $\tilde{A}_1=X^\top Z$ and $\varPi=I_n-XX^\top$.
        Then $\varPi Z\tilde{A}_1=\mathcal{P}_{\mathrm{T}_X}(Z\tilde{A}_1)$ and
        \begin{equation}\label{eq:stiefel-alpha}
            \mathfrak{G}_X^\beta(Z,Z)=2(\beta-1)\varPi Z\tilde{A}_1,
            \qquad
            \alpha_X(Z)=
            \begin{cases}
                (2\beta-1)\varPi Z\tilde{A}_1, & \mathcal{R}_X=\mathcal{R}_X^{\mathsf{cay}},\\
                2(\beta-1)\varPi Z\tilde{A}_1, & \mathcal{R}_X\in\{\mathcal{R}_X^{\mathsf{qgeo}},\mathcal{R}_X^{\mathsf{pol}}\}.
            \end{cases}
        \end{equation}
        Hence $\mathcal{R}_X^{\mathsf{qgeo}}$ and $\mathcal{R}_X^{\mathsf{pol}}$ are second-order for the Euclidean metric $g^1$, and $\mathcal{R}_X^{\mathsf{cay}}$ is second-order for the canonical metric $g^{1/2}$.
        For $2\le p<n$ no other value of $\beta$ gives either property.
        At $p=1$ one has $\tilde{A}_1=0$ and at $p=n$ one has $\varPi=0$.
        Both cases of \eqref{eq:stiefel-alpha} therefore vanish for every $\beta$.
    \end{proposition}

    \begin{proof}
        For these choices, $X^{(P, Q)}=X^\top$, $X^\dagger=X^\top$, $M_X^\beta=I_n-(1-\beta)XX^\top$, and $(M_X^\beta)^{-1}=XX^\top/\beta+\varPi$ by \eqref{eq:metric-2}.
        Write $\tilde{A}_2=Z^\top Z$, and recall $\varPi Z=Z-X\tilde{A}_1$ and $Z^\top X=-\tilde{A}_1$.
        The expressions for $\mathrm{D}_ZM_X^\beta$ and $\tilde{\mathcal{K}}$ in the proof of Proposition~\ref{prop:hess} reduce to
        \begin{equation*}
            \mathrm{D}_ZM_X^\beta=(\beta-1)(XZ^\top+ZX^\top),
            \qquad
            \tilde{\mathcal{K}}(Z,Z)=2(1-\beta)Z\tilde{A}_1-2X\tilde{A}_1^2,
        \end{equation*}
        so that $\mathcal{K}(Z,Z)=\mathrm{D}_ZM_X^\beta\cdot Z-\tilde{\mathcal{K}}(Z,Z)/2=(\beta-1)X\tilde{A}_2+2(\beta-1)Z\tilde{A}_1+X\tilde{A}_1^2$.
        Applying $(M_X^\beta)^{-1}$ leaves $X\tilde{A}_2$ and $X\tilde{A}_1^2$ multiples of themselves and sends $Z\tilde{A}_1$ to $X\tilde{A}_1^2/\beta+\varPi Z\tilde{A}_1$.
        The projector $\mathcal{P}_{\mathrm{T}_X}$ annihilates $X\tilde{A}_2$ and $X\tilde{A}_1^2$ by \eqref{eq:normal-space-canonical}.
        Substitution into \eqref{eq:christoffel} gives $\mathfrak{G}_X^\beta(Z,Z)=2(\beta-1)\varPi Z\tilde{A}_1$, and the two cases of \eqref{eq:stiefel-alpha} follow from \eqref{eq:alpha}.
        For the converse, take $\varOmega\in\mathcal{S}_{\mathrm{skew}}(I_p)$ nonzero, which requires $p\ge2$, and $W\in\mathbb{R}^{(n-p)\times p}$ with $W\varOmega\neq0$, which requires $p<n$.
        For the tangent vector $Z=X\varOmega+X_\perp W$ of \eqref{eq:TS-Gq-b}, we obtain $\tilde{A}_1=\varOmega$ and $\varPi Z\tilde{A}_1=X_\perp W\varOmega\neq0$.
        It follows that the first case of \eqref{eq:stiefel-alpha} vanishes for every $Z$ only at $\beta=1/2$ and the second only at $\beta=1$.
    \end{proof}

    On $\mathrm{St}(p,n)$ with $\beta=1$ the polar retraction \eqref{eq:polar-retraction} is the metric projection onto the manifold.
    Proposition~\ref{prop:stiefel-second-order} therefore recovers, for $\mathrm{St}(p,n)$, the second-order property of the metric projection retraction on an embedded submanifold with the induced metric~\cite{absil2012projection}.
    The projection transport \eqref{eq:VT-orth-proj} depends only on the endpoint of the tangent step and thus applies unchanged to $\widetilde{\mathcal{R}}_X$.

\subsection{Complexity and Specialization}
\label{subsec:complexity}

    Table~\ref{tab:complexity} summarizes the floating-point operation (flop) counts of the three retractions and their vector transports, together with their Stiefel and symplectic Stiefel specializations.
    Three observations follow from the table.
    First, the $n^2p$ term counts the $(P,Q)$-adjoints, and it vanishes for $(P,Q)=(I_p,I_n)$ and $(P,Q)=(J_p,J_n)$, where the $(P,Q)$-adjoint reduces to a transpose or a signed permutation of it.
    Second, the three retractions differ in the constant of the $p^3$ term.
    Computing the matrix exponential by scaling and squaring gives the quasi-geodesic-based retraction a larger constant than the polar-based and Cayley-based retractions~\cite[Sec.~10.3]{higham2008functions}.
    Third, in the general $\mathrm{Gq}(P,Q)$ column, the Cayley retraction has the largest $np^2$ coefficient, as the SMW formula operates on $n$-by-$2p$ matrices.
    For dense $P$ and $Q$, the general $\mathrm{Gq}(P,Q)$ formulas share the cost $\mathcal{O}(n^{2}p+np^{2}+p^{3})$.
    When the structure of $P$ and $Q$ lets the adjoint products be applied in $\mathcal{O}(np)$ operations, as in the Stiefel and symplectic Stiefel cases, the $n^2p$ contribution disappears.
    Evaluating $\alpha_X(Z)$ in \eqref{eq:alpha} costs $4n^{2}p+\mathcal{O}(np^{2}+p^{3})$ in the same convention.
    Each second-order retraction \eqref{eq:second-order-retraction} therefore has the same asymptotic cost as the map it corrects.
    The choice among the three retractions is therefore determined by the structure of $P$ and $Q$ and by the relative magnitude of $n$ and $p$.

    \begin{table}[htbp]
        \centering
        \small
        \caption{Formulas and flop counts under the canonical-like metric.}
        \label{tab:complexity}
        \renewcommand{\arraystretch}{1.15}
        \setlength{\tabcolsep}{6pt}
        \begin{tabular}{@{}l l l l@{}}
        \toprule
        Operation & $\mathrm{St}(p, n)$ & $\mathrm{Sp}(p, n)$ & $\mathrm{Gq}(P, Q)$ \\
        \midrule
        \multirow{2}{*}{$\mathcal{P}_{\mathrm{T}_{X}}(Y)$}
        & \cite[Eq.~(7.26)]{boumal2023introduction} & \cite[Eq.~(4.4)]{gao2021riemannian} & \eqref{eq:projection-tangent-canonical} \\
        & $4np^2+\mathcal{O}(np+p^2)$ & $3np^2+\mathcal{O}(np)$ & $2n^2p+4np^2+\mathcal{O}(p^3)$ \\
        \addlinespace[4pt]
        \multirow{2}{*}{$\mathcal{R}^{\mathsf{cay}}_{X}(Z)$}
        & \cite[Eq.~(15)]{zhu2017riemannian} & \cite[Eq.~(5.8)]{gao2021riemannian} & \eqref{eq:Cayley-retraction} \\
        & $6np^2+\mathcal{O}(np+p^3)$ & $6np^2+\mathcal{O}(p^3)$ & $4n^2p+20np^2+\mathcal{O}(p^3)$ \\
        \addlinespace[4pt]
        \multirow{2}{*}{$\mathcal{R}^{\mathsf{qgeo}}_{X}(Z)$}
        & \cite[Eqs.~(2.45)--(2.47)]{edelman1998geometry} & \cite[Eq.~(5.1)]{gao2021riemannian} & \eqref{eq:quasi-geodesic-retraction} \\
        & $12np^2+\mathcal{O}(np+p^3)$ & $9np^2+\mathcal{O}(p^3)$ & $4n^2p+10np^2+\mathcal{O}(p^3)$ \\
        \addlinespace[4pt]
        \multirow{2}{*}{$\mathcal{R}^{\mathsf{pol}}_{X}(Z)$}
        & \cite[Eq.~(4.7)]{absil2008optimization} & --- & \eqref{eq:polar-retraction} \\
        & $3np^2+\mathcal{O}(np+p^3)$ & --- & $2n^2p+4np^2+\mathcal{O}(p^3)$ \\
        \addlinespace[4pt]
        \multirow{2}{*}{$\mathcal{T}^{\mathsf{proj}}_{Z}(W)$}
        & --- & --- & \eqref{eq:VT-orth-proj} \\
        & --- & --- & $2n^2p+6np^2+\mathcal{O}(p^3)$ \\
        \addlinespace[4pt]
        \multirow{2}{*}{$\mathcal{T}^{\mathsf{cay}}_{Z}(W)$}
        & \cite[Eq.~(17)]{zhu2017riemannian} & --- & \eqref{eq:VT-Cayley} \\
        & $4np^2+\mathcal{O}(p^3)$ & --- & $4n^2p+20np^2+\mathcal{O}(p^3)$ \\
        \addlinespace[4pt]
        \multirow{2}{*}{$\mathcal{T}^{\mathsf{pol}}_{Z}(W)$}
        & \cite[Lem.~10.2.1]{huang2013optimization} & --- & \eqref{eq:VT-polar} \\
        & $4np^2+\mathcal{O}(np+p^3)$ & --- & $4n^2p+10np^2+\mathcal{O}(p^3)$ \\
        \bottomrule
        \end{tabular}
    \end{table}

\section{Numerical Validation}
\label{sec:numerical}

    All geometric objects derived in this paper are available in closed form, namely the canonical-like metric, the orthogonal projection, the Riemannian gradient and Hessian, and the three retractions together with their vector transports and their second-order corrections.
    This section verifies them numerically, first by a Taylor test of the gradient and Hessian and then by optimization runs.
    Both parts use one test problem, trace minimization.
    One implementation handles the symmetric positive-definite, the skew-symmetric, and the symmetric indefinite cases.

    \textbf{Setup}.
    We ran all experiments on a MacBook Pro (macOS Ventura 13.3) equipped with a 12-core Apple M2 Max processor, 32 GB of memory, and MATLAB R2024b.
    We implemented $\mathrm{Gq}(P, Q)$ as a manifold factory \texttt{generalfactory} for \texttt{Manopt}~\cite{boumal2014manopt} version~8.0\footnote{\url{https://www.manopt.org}.}, which works with the built-in solvers of the toolbox.
    All experiments use the canonical-like metric with $\beta=1$.
    The source code behind these experiments is available at \url{https://github.com/pxxyyz/GQmanifold}.

    \textbf{Test problem}.
    We consider the trace minimization problem
    \begin{equation*}
        \min_{X\in\mathrm{Gq}(P, Q)} f(X) = \mathrm{tr}(X^\top AX),
    \end{equation*}
    where $A\in\mathbb{R}^{n\times n}$ is a randomly generated SPD matrix whose eigenvalues are linearly spaced over $[1/\kappa,\,1]$ for a prescribed condition number $\kappa=10^{2}$.
    The associated Euclidean gradient and Hessian are $\nabla f(X)=2AX$ and $\nabla^2 f(X)[\dot X]=2A\dot X$.

    \textbf{Configurations of $(P,Q)$}.
    We test one instance of $\mathrm{Gq}(P, Q)$ in each of the three cases.
    For the symmetric positive-definite case, we draw $B\sim\mathcal{N}(0,1)^{n\times n}$ and set $Q=\texttt{sqrtm}(BB^\top)+10^{-6}I_n$, an SPD matrix, and generate $P$ by the same rule at size $p\times p$.
    For the skew-symmetric case, we set $Q=(B-B^\top)/2$ and generate $P$ analogously at even size $p\times p$.
    A random skew-symmetric matrix of even size is almost surely nonsingular, making this pair admissible with probability one.
    For the symmetric indefinite case, we set $Q=U\,\mathrm{diag}(\lambda)\,U^\top$ with a random orthogonal matrix $U$ and a spectrum $\lambda$ of $\lceil 0.6n\rceil$ positive entries linearly spaced over $[1,10]$ and the remaining negative entries linearly spaced over $[-10,-1]$.
    With $P=\mathrm{blkdiag}(I_{p/2},-I_{p/2})$, which satisfies $P^2=I_p$, this instance is an indefinite Stiefel manifold.
    All instances have $n=100$ and $p=10$.

    \textbf{Initialization}.
    Within each case, all solver and retraction configurations start from the same point $X_0 \in \mathrm{Gq}(P, Q)$ returned by the factory's random point generator.
    This point is constructed from the structured factorization of $P$ and $Q$, namely the eigendecomposition for $\varrho=+1$, the Youla normal form for $\varrho=-1$, and the $(P,Q)$-adapted hyperbolic factorization in the indefinite case, followed by a random orthogonal-type perturbation.
    The fixed seed \texttt{rng(1)} makes the runs reproducible.

    \textbf{Termination criteria}.
    Each solver terminates as soon as the Riemannian gradient norm meets the absolute stopping tolerance, the iteration count reaches its maximum, or the step size drops below the minimum step size.
    All experiments use $\texttt{tol} = 10^{-8}$, $\texttt{maxiter} = 5 \times 10^3$, and $\texttt{minstepsize} = 10^{-10}$.

\subsection{Checking the Gradient and Hessian}
\label{sec:check}
    The closed forms \eqref{eq:grad} and \eqref{eq:hess} hold for every smooth objective.
    We check them for the test problem with the \texttt{checkgradient} and \texttt{checkhessian} routines of \texttt{Manopt} at the random point $X_0$, where $\mathrm{grad}\,f(X_0)$ is nonzero.
    For a tangent direction $Z\in \mathrm{T}_X\mathrm{Gq}(P, Q)$ and a retraction $\mathcal{R}_X$, both routines fit the log-log slope of the Taylor remainders
    \begin{equation*}
        \begin{aligned}
            \varepsilon(t)&=\bigl|f(\mathcal{R}_X(tZ))-f(X)-t\,g_X^\beta(\mathrm{grad}\,f(X),Z)\bigr|,\\
            \varepsilon_H(t)&=\bigl|f(\mathcal{R}_X(tZ))-f(X)-t\,g_X^\beta(\mathrm{grad}\,f(X),Z)-\frac{t^2}{2}g_X^\beta(\mathrm{hess}\,f(X)[Z],Z)\bigr|.
        \end{aligned}
    \end{equation*}
    Under any retraction, a correct gradient gives $\varepsilon(t)=\mathcal{O}(t^2)$ at every point.
    At a point where the Riemannian gradient is nonzero, the leading term of $\varepsilon_H$ is $(t^2/2)\bigl|g_X^\beta(\mathrm{grad}\,f(X),\alpha_X(Z))\bigr|$, with $\alpha_X(Z)$ the covariant acceleration \eqref{eq:alpha}.
    A correct Hessian therefore gives $\varepsilon_H(t)=\mathcal{O}(t^2)$ when $\alpha_X(Z)$ is nonzero, and $\varepsilon_H(t)=\mathcal{O}(t^3)$ under a second-order retraction \eqref{eq:second-order-retraction}~\cite[Sec.~5.10]{boumal2023introduction}.
    Figure~\ref{fig:check} carries one column for each of the three retractions and one curve in every panel for each of the three cases.
    Its first two rows plot $\varepsilon$ and $\varepsilon_H$ under the retraction of the column, its third row plots $\varepsilon_H$ under the correction \eqref{eq:second-order-retraction} of that same map, and the fitted slope is two in the first two rows and three in the third.
    In the last panel the symmetric indefinite curve is drawn only where $\widetilde{\mathcal{R}}_{X_0}^{\mathsf{pol}}$ is defined.

    Slope two in the second row rules out $\alpha_{X_0}(Z)=0$, and none of the three retractions is second-order for $g^1$ in this experiment.
    The cancellation producing slope three relates the metric-compatibility correction $\mathcal{K}$ of \eqref{eq:hess} to the same $\mathcal{K}$ inside \eqref{eq:christoffel}.
    The third row tests this contribution of $\mathcal{K}$, whereas the second row does not.
    That test needs a point where the gradient is nonzero, since $\mathfrak{I}_4=\mathcal{K}(Z,\mathrm{grad}\,f(X))$ carries the Riemannian gradient in its second argument.
    The tangency, linearity, and self-adjointness residuals of a Riemannian Hessian~\cite[Def.~5.14, Prop.~5.15]{boumal2023introduction} stay at rounding level.

    \begin{figure}[htbp]
        \centering
        \includegraphics[width=0.95\linewidth]{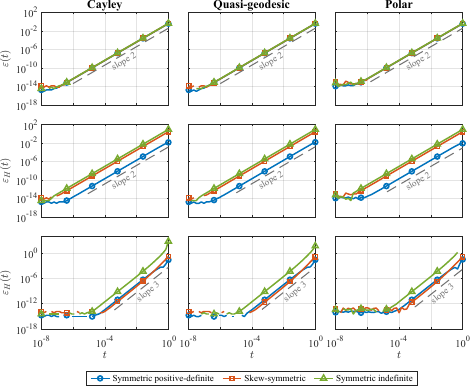}
        \caption{Taylor remainders of $f$ at the random point $X_0$.}
        \label{fig:check}
    \end{figure}

\subsection{Trace Minimization}
\label{sec:trace}

    This subsection runs two Riemannian solvers with each of the three retractions in the three cases.
    A complete run shows whether the gradient norm reaches the stopping tolerance and whether the retractions keep every iterate on $\mathrm{Gq}(P, Q)$, which a local derivative check leaves open.

    \textbf{Monitored quantities}.
    Along the iterates $\{X_k\}$ we record two quantities.
    The Riemannian gradient norm $\|\mathrm{grad}\,f(X_{k})\|_{X_{k}}$ is the quantity that the solvers drive to the stopping tolerance.
    The feasibility error $\|X_{k}^\top Q X_{k}-P\|_{\mathrm{F}}$ is the residual of the defining equation of $\mathrm{Gq}(P, Q)$ and stays at rounding level as long as every retraction maps into $\mathrm{Gq}(P, Q)$.
    We denote these two quantities \texttt{gradf} and \texttt{feasi}.
    We also report the final cost function value \texttt{fval}, the iteration count \texttt{iter}, and the wall-clock time \texttt{time}.

    \textbf{Solvers}.
    We use two built-in solvers of \texttt{Manopt}, the Riemannian trust-region method (\texttt{RTR}) and the Riemannian conjugate gradient method (\texttt{RCG}).
    \texttt{RCG} builds its search direction with its default Hestenes--Stiefel formula~\cite{hestenes1952methods}.
    At every iteration, \texttt{RCG} evaluates the Riemannian gradient, a retraction, and a vector transport, and \texttt{RTR} evaluates the Riemannian gradient, the Riemannian Hessian, and a retraction.
    Apart from the stopping parameters above, the solvers run at their \texttt{Manopt} defaults.
    The indefinite \texttt{RTR} runs set $\texttt{rho\_regularization} = 0$.

    \textbf{Retractions and transports}.
    Each solver runs with the Cayley-based, polar-based, and quasi-geodesic-based retractions, denoted \texttt{cay}, \texttt{pol}, and \texttt{qgeo}.
    The polar retraction applies in all three cases, including the indefinite one, through the real principal square root, which exists at every step accepted in these runs.
    For the \texttt{RCG} solver, each retraction is paired with its matching vector transport.
    The Cayley and polar retractions use their associated transports, whereas the quasi-geodesic retraction uses the orthogonal projection onto the tangent space as transport.
    \texttt{Manopt} calls the transporter as \texttt{M.transp(x, newx, grad)}, which specifies the transport direction only through the endpoints \texttt{x} and \texttt{newx}.
    We therefore recover $Z$ by the inverse retraction before evaluating the vector transport $\mathcal{T}_Z(W)$.
    The closed forms of Theorem~\ref{thm:Cayley-retraction} and Theorem~\ref{thm:polar-retraction} serve the endpoint pairs lying in their inversion domains, and the projection transport \eqref{eq:VT-orth-proj} serves the remaining ones.

    Table~\ref{tab:trace} reports the final convergence results for the two solvers, each paired with the three retractions, across the three cases.
    Figure~\ref{fig:trace} shows the corresponding iteration histories, with one subfigure per solver.
    Within each subfigure, the three columns are the symmetric positive-definite, the skew-symmetric, and the symmetric indefinite cases.
    The top row plots the Riemannian gradient norm and the bottom row the feasibility error against the iteration count.
    Each curve corresponds to one retraction (Cayley, polar, quasi-geodesic), and the dotted line marks the stopping tolerance.
    The feasibility axis is shared across the three cases.

    \begin{table}[htbp]
    \centering
    \small
    \caption{Final results of \texttt{RCG} and \texttt{RTR} on the trace minimization problem.}
    \label{tab:trace}
    \renewcommand{\arraystretch}{1.15}
    \setlength{\tabcolsep}{6pt}
    \pgfplotstableset{
        col sep=comma,
        every head row/.style={before row=\toprule, after row=\midrule},
        every last row/.style={after row=\midrule},
        columns/cond/.style={column name=$\texttt{cond}(A)$, fixed, precision=0, column type=w{r}{1cm}},
        columns/method/.style={column name=\texttt{method}, string type, column type=w{l}{1.2cm}},
        columns/retr/.style={column name=\texttt{retr}, string type, column type=w{l}{1.0cm}},
        columns/fval/.style={column name=\texttt{fval}, sci, sci zerofill, precision=2, column type=w{r}{2.1cm}},
        columns/gradf/.style={column name=\texttt{gradf}, sci, sci zerofill, precision=2, column type=w{r}{2.1cm}},
        columns/feasi/.style={column name=\texttt{feasi}, sci, sci zerofill, precision=2, column type=w{r}{2.1cm}},
        columns/iter/.style={column name=\texttt{iter}, fixed, precision=0, column type=w{r}{1.1cm}},
        columns/time/.style={column name={\texttt{time} (s)}, fixed, fixed zerofill, precision=2, column type=w{r}{1.4cm}},
        columns/case/.style={string type},
        columns={method,retr,fval,gradf,feasi,iter,time},
        column type=r,
        every column/.append style={font=\small},
    }
    \begin{filecontents*}{symm.csv}
case,method,retr,fval,gradf,feasi,iter,time,stop_reason,stepsize
symm,RCG,cay,0.110436567587129,6.32898691794157e-08,5.10364839437892e-14,5000,13.537313375,maxiter,3.29394612127758e-07
symm,RCG,pol,0.110436567586717,1.5296640497451e-08,1.82381990166115e-14,4422,15.3795000833333,minstepsize,0
symm,RCG,qgeo,0.110436567585679,9.9513794851054e-09,1.20301072852728e-13,4693,9.429632625,gradtol,3.09864210860254e-08
symm,RTR,cay,0.11043656758483,3.83679931761856e-12,5.54661313161367e-14,303,13.6748940416667,gradtol,3.23286939183141e-06
symm,RTR,pol,0.11043656758483,7.73094529941815e-12,2.9148587383386e-14,259,13.508915375,gradtol,5.74354633383055e-06
symm,RTR,qgeo,0.11043656758483,7.94412184762832e-10,6.18720056472387e-14,291,12.6159225416667,gradtol,5.55972813034716e-05
\end{filecontents*}
\pgfplotstabletypeset[
        every first row/.style={
            before row={
                \multicolumn{7}{l}{
                \textbf{Symmetric Positive-Definite Case}
                } \\
            },
            after row=,
        },
    ]{symm.csv}
    \begin{filecontents*}{skew.csv}
case,method,retr,fval,gradf,feasi,iter,time,stop_reason,stepsize
skew,RCG,cay,0.143833113790737,8.80368896671854e-09,3.31278886068542e-13,672,2.40291545833333,gradtol,1.99819348168177e-08
skew,RCG,pol,0.14383311379071,9.8233448533278e-09,9.69533869670124e-15,673,2.38612666666667,gradtol,2.21403258882829e-08
skew,RCG,qgeo,0.143833113790705,8.11400005451723e-09,3.47735939015578e-13,587,1.49393916666667,gradtol,1.00778797359506e-08
skew,RTR,cay,0.143833113790691,4.05380417259856e-11,2.54076241022934e-13,81,2.82583908333333,gradtol,6.68892313059759e-06
skew,RTR,pol,0.14383311379069,5.22374915807825e-10,6.97245822864718e-15,84,3.61183941666667,gradtol,2.93612498546704e-05
skew,RTR,qgeo,0.143833113790691,9.36738311701666e-11,1.95451780328055e-13,79,3.32838741666667,gradtol,1.24841241503071e-05
\end{filecontents*}
\pgfplotstabletypeset[
        every first row/.style={
            before row={
                \multicolumn{7}{l}{
                \textbf{Skew-Symmetric Case}
                } \\
            },
            after row=,
        },
        every head row/.style={output empty row},
    ]{skew.csv}
    \begin{filecontents*}{indef.csv}
case,method,retr,fval,gradf,feasi,iter,time,stop_reason,stepsize
indef,RCG,cay,0.151449136095982,8.57347388968927e-09,9.53523642105552e-14,97,0.323559875,gradtol,8.9886507293504e-09
indef,RCG,pol,0.151449136095983,8.71085656904515e-09,5.78472180507528e-15,98,0.346109833333333,gradtol,1.43071636861105e-08
indef,RCG,qgeo,0.151449136095983,6.95566674707796e-09,8.40401748384926e-14,98,0.24033425,gradtol,2.67171130252271e-08
indef,RTR,cay,0.151449136095982,2.35370010421233e-14,9.21579256096714e-14,42,2.89962129166667,gradtol,1.84410828547887e-07
indef,RTR,pol,0.151449136095983,8.80866622496191e-09,5.55228770543969e-15,42,3.055715625,gradtol,0.0291699825578926
indef,RTR,qgeo,0.151449136095983,2.53558102105878e-15,1.7166912354465e-14,39,2.7692605,gradtol,5.92114769962707e-08
\end{filecontents*}
\pgfplotstabletypeset[
        every first row/.style={
            before row={
                \multicolumn{7}{l}{
                \textbf{Symmetric Indefinite Case}
                } \\
            },
            after row=,
        },
        every head row/.style={output empty row},
        every last row/.style={after row=\bottomrule},
    ]{indef.csv}
    \end{table}

    \begin{figure}[htbp]
        \centering
        \begin{minipage}{0.95\linewidth}
            \centering
            \includegraphics[width=\linewidth]{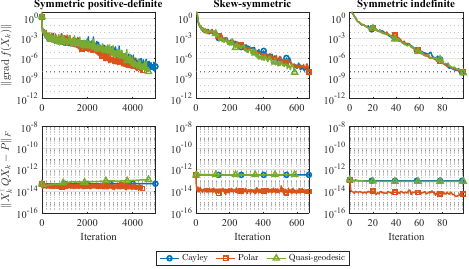}
            \par\small(a) \textbf{Riemannian conjugate gradient method (\texttt{RCG})}
        \end{minipage}

        \vspace{0.4em}

        \begin{minipage}{0.95\linewidth}
            \centering
            \includegraphics[width=\linewidth]{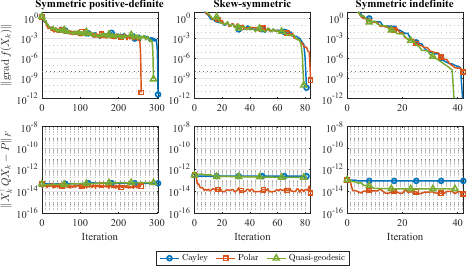}
            \par\small(b) \textbf{Riemannian trust-region method (\texttt{RTR})}
        \end{minipage}

        \caption{Iteration histories on the trace minimization problem.}
        \label{fig:trace}
    \end{figure}

    For each case, all runs end at the same objective value.
    The \texttt{RTR} solver meets the gradient tolerance in every case and with every retraction.
    \texttt{RCG} does the same except in the symmetric positive-definite case, where the Cayley and the polar runs stop at \texttt{maxiter} and at \texttt{minstepsize} with final gradient norms less than one order of magnitude above the tolerance.
    Across all three cases the final feasibility error stays at the level of the rounding error committed in forming $X^\top QX-P$.
    These errors confirm that every retraction keeps the iterates on $\mathrm{Gq}(P, Q)$ to working precision.
    The \texttt{RTR} solver needs substantially fewer iterations than \texttt{RCG} in all three regimes, at a higher per-iteration cost arising from its inner subproblem solves.

\section{Conclusions and Perspectives}
\label{sec:conclusion}

    We have shown that $\mathrm{Gq}(P,Q)$ is a closed embedded submanifold for every admissible pair $(P,Q)$, and we have developed its first- and second-order Riemannian geometry under a canonical-like metric.
    We have also constructed three retractions with their vector transports and second-order corrections.
    The Stiefel, generalized Stiefel, indefinite Stiefel, and symplectic Stiefel manifolds and the quadratic matrix Lie groups are special cases.
    On the Stiefel and symplectic Stiefel manifolds, the projection, the Riemannian gradient, and the Cayley and quasi-geodesic retractions reduce to the known formulas, and the polar retraction reduces to the known one on the Stiefel manifold.
    The unification rests on the $(P,Q)$-adjoint, under which the tangent space admits three equivalent characterizations, namely a defining-equation form, a bijective coordinate chart, and a Lie-algebra image.
    The first yields the orthogonal projection, the second stores tangent vectors with the minimal number of parameters, and the third yields the Cayley and quasi-geodesic retractions.
    On the bijective chart, the canonical-like metric separates the two coordinate blocks and delivers a closed-form, Lyapunov-free projector.

    Several questions remain open.
    First, convergence guarantees for Riemannian methods on the unbounded sets $\mathrm{Gq}(P,Q)$ with the locally defined Cayley and polar retractions remain to be established.
    Second, the group $\{U\in\mathbb{R}^{p\times p}:U^\top PU=P\}$ acts on $\mathrm{Gq}(P,Q)$ by $X\mapsto XU$, and the geometry of the quotient, with its horizontal and vertical decomposition, is still to be described.
    Finally, extending the QR retraction on $\mathrm{St}(p,n)$ and the SR retraction on $\mathrm{Sp}(p,n)$ to $\mathrm{Gq}(P,Q)$ via a concise algebraic formula remains open.

\section*{Conflict of Interest}
The authors declare no conflicts of interest.

\section*{Data Availability Statement}
The source code for generating the numerical results is openly available at \url{https://github.com/pxxyyz/GQmanifold}.

\section*{Funding}

This work was supported by the National Key R\&D Program of China (grant No. 2023YFA1009300).
BG was supported by the National Natural Science Foundation of China (grant No. 12288201).
ZP was supported by the China Postdoctoral Science Foundation (grant No. 2024M763472).

\bibliographystyle{unsrtnat}
\bibliography{main}

\end{document}